\documentclass[11pt, a4paper, oneside]{book}
% ...

\usepackage[utf8]{inputenc}
\usepackage[main = english]{babel} 			
\usepackage[T1]{fontenc}							
\usepackage{newpxtext}

\usepackage[ 
		layoutoffset=0pt,						
		bottom=4cm, 
		top=2.5cm,	
		left=4cm, 
		right=2.5cm,			
		bindingoffset=0pt,						
		headheight=26pt				
		]{geometry}

\usepackage{comment}
					
\usepackage{microtype, fancyhdr}	 
\usepackage{setspace}					 
\usepackage{pdflscape} 					 

\usepackage{amsfonts,amsmath,amsthm,amssymb, mathrsfs, amscd, bm}	 
\numberwithin{equation}{chapter}	

\usepackage{url, enumitem} 			
\usepackage[noadjust]{cite}				
\usepackage{epigraph}						
\usepackage[nottoc]{tocbibind}			
\usepackage{blindtext}						
\usepackage{imakeidx}					
\makeindex

\usepackage{graphicx} 				
\graphicspath{{Pictures/}}				
\usepackage{xcolor}
\usepackage{longtable} 				
\usepackage{lscape} 					
\usepackage{booktabs} 				
\usepackage{multicol} 					
\usepackage{multirow}		
\usepackage{wrapfig} 					

\usepackage{caption}

\let\headruleORIG\headrule
\renewcommand{\headrule}{\color{black} \headruleORIG}

\usepackage{colortbl}
\arrayrulecolor{black}

\fancypagestyle{plain}{
  \fancyhead{}
  \fancyfoot{}
  
}

\makeatletter
\def\cleardoublepage{\clearpage\if@twoside \ifodd\c@page\else%
    \hbox{}%
    \thispagestyle{empty}
    \newpage%
    \if@twocolumn\hbox{}\newpage\fi\fi\fi}
\makeatother

\usepackage{etoolbox}
\patchcmd{\chapter}{plain}{empty}{}{}
\makeatletter
  \let\ps@plain\ps@empty
\makeatother

\usepackage[linktocpage,bookmarksnumbered,pagebackref]{hyperref} 
\usepackage{breqn} 
\usepackage{pdfpages} 

\usepackage{bbm}
% ...

\newcommand{\E}{\mathbb E}
\newcommand{\J}{\mathbb J}
\newcommand{\Tr}{\mathrm{Tr}}
\newcommand{\T}{\mathbb{T}}
\newcommand{\PP}{\mathbb P}

\newcommand{\x}{\mathbf{x}}
\newcommand{\xx}{{\bf x}}

\newcommand{\TT}{\mathcal{T}}
\newcommand{\UU}{\mathcal{U}}

\newcommand{\beq}{\begin{equation}}
\newcommand{\eeq}{\end{equation}}

\newcommand\numberthis{\addtocounter{equation}{1}\tag{\theequation}}

\newtheorem{thm}{Theorem}[section]
\newtheorem{prop}[thm]{Proposition}
\newtheorem{cor}[thm]{Corollary}
\newtheorem{lem}[thm]{Lemma}
\newtheorem{definition}{Definition}[section]

\begin{document}
\begin{titlepage}


\begin{center}

\begin{spacing}{2.5}
{\Huge \textsc{\textbf{Spectral properties of random matrices}}}
\end{spacing}

\vspace{1.0 cm}

\begin{center}
\includegraphics[width=0.5\textwidth]{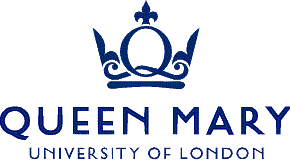}
\end{center}
\vspace{0.9cm}

{\LARGE Anastasis Kafetzopoulos} 

\vspace{1cm}

{\large \sffamily
Doctor of Philosophy\\
\vspace{0.9 cm}
School of Mathematics\\
\vspace{0.5 cm}
Queen Mary University of London\\
\vspace{0.5cm}
April 2023 
}

\end{center}

\end{titlepage}

\doublespacing		
\chapter*{Statement of originality}
\label{C:Statement}
\addcontentsline{toc}{section}{\nameref{C:Statement}}

I confirm that the research included
within this thesis is my own work or that where it has been carried out in collaboration with, or supported by others, that this is duly acknowledged below and my contribution indicated. Previously published material is also acknowledged below.
\medskip

I attest that I have exercised reasonable care to ensure that the work is original, and does not to the best of my knowledge break any UK law, infringe any third party’s copyright or other Intellectual Property Right, or contain any confidential material.
\medskip

I accept that Queen Mary, University of London has the right to use plagiarism detection software to check the electronic version of the thesis.
\medskip

I confirm that this thesis has not been previously submitted for the award of a degree by this or any other university.
\medskip

The copyright of this thesis rests with the author and no quotation from it or information derived from it may be published without the prior written consent of the author.
\medskip

Signature: Anastasis Kafetzopoulos

Date: April 2023
\medskip

Details of publications: 

1. Local Marchenko-Pastur law. [Journal of Mathematical Physics] \newline
2. Universality of correlated covariance matrices. [in preparation]
\chapter*{Abstract}
\label{C:Abstract}
\addcontentsline{toc}{section}{\nameref{C:Abstract}}


In the first part of this thesis, we give the theoretical foundations of random matrix theory through the definitions of a random matrix, a random probability measure and the corresponding empirical spectral distribution we will be working with. The main technical tool of the first paper is also defined rigorously and analyzed deeply, which is the Stieltjes transform method.

We then use this tool to prove optimal convergence of the empirical spectral distribution of random sample covariance matrices to the deterministic Marchenko-Pastur distribution. We also give new results about the rigidity of the eigenvalues of this random sample covariance matrix as well as about the rate of their convergence.

In the second part of this thesis, we define another important and more general technical tool which works additionally well with non-Hermitian random matrices and that is the Dyson equation method which was used in the second paper. Just like the Stieltjes transform method, it is also defined rigorously and analyzed deeply.

We then prove new local laws about a random matrix model that interpolates between the Marchenko-Pastur distribution, the elliptical law and the circular law. Through our work these local laws can now be considered universal, which means that they are independent of the initial distribution of the random matrix entries.

We finally give an overview of our new results and provide new directions of study.

\chapter*{Acknowledgements}
\label{C:Acknowledgements}
\addcontentsline{toc}{section}{\nameref{C:Acknowledgements}}

In this section, I would like to thank everyone that made this work possible. Firstly, many thanks should go to my talented supervisor Anna Maltsev who firstly accepted me to this PhD program and then guided me through all these years with patient advice, valuable insights, high-level scientific discussions and excellent professionalism. Secondly, I would like to thank my second supervisor, Alexander Sodin, who as a world-renowned researcher in this scientific field uncovered his way of thinking to me through our meetings and some invaluable discussions. 

Next, many thanks should go to the Royal Society for funding this project for four years and also providing an excellent travelling package for scientific development and training in many conferences and summer schools. I would also like to thank Queen Mary, University of London for its excellent organisation and professionalism through all these years and its support for the research students while providing a very decent academic environment for conducting research. I would like to thank all my relatives and friends for supporting me while facing this interesting and challenging research project and experience. 

Lastly, I would like to thank my two examiners, Nick Simm and Alex Gnedin for their very useful comments and observations about the thesis.
\begin{spacing}{1}
\tableofcontents
\end{spacing} 
\chapter{Introduction}

\begin{section}{Random matrix theory}

Random matrix theory is an active research area of modern mathematics combining Mathematical and Theoretical Physics, Mathematical Analysis and Probability, and with numerous applications, for example in Theoretical Physics, Number Theory, Combinatorics and further in Statistics, Financial Mathematics, Biology and Engineering. \cite{akemann2011oxford}, (ch.3: applications of random matrices).

The main goal of Random Matrix Theory is to provide an understanding of the properties of random matrices, that is matrices that have as elements random variables, real, complex or quaternion. One important direction of investigation is the study of the eigenvalue distribution of such matrices when the matrix dimension is large. Many quantities in physics and other applied areas are modeled by the eigenvalues of such matrices. 

The origins of random matrices can be traced back to the works of Wishart in 1928 and James in 1960 in the discipline of Statistics, where random matrices were used to derive distributions for numerous statistics of normal multivariate random variables. A lot of development in this mathematical field was due to the work of Wigner in Nuclear Physics around 1950. Wigner was the first to suggest that the fluctuations of nuclei resonances can be described in terms of the properties of eigenvalues of very large symmetric matrices with random entries. 

An observable in quantum mechanics can be characterized by a self-adjoint linear operator in a Hilbert space, its Hamiltonian, which we may think informally of as a matrix of infinitely many dimensions. This suggests that random matrix theory should be dealing with general properties of the underlying generic Hamiltonians, most importantly such global features as the Hermiticity, the time-inversion invariance as well as other symmetries Hamiltonians may obey from general principles. Wigner hoped that the output of the model, that is the distribution of the eigenvalues for large-dimensional Hamiltonians will be universal and independent of the details of the entries. It would be common to the majority of systems sharing the corresponding symmetries. \cite{mehta2004random} (section 1.1: random matrices in nuclear physics). 

Along those lines Wigner succeeded in calculating nontrivial spectral characteristics of random real symmetric matrices with independent, identically distributed entries, the mean density of the eigenvalues and demonstrated that in the limit of large matrix size it is given quite generally by the semi-circle law. 
\end{section}

\begin{section}{The Wigner semi-circle law}
Consider a random symmetric $N\times N$ matrix $X$ with independent and identically distributed entries $(x_{ij})$ for $i\leq j \in \{1,...,N\}$ and then consider the distribution of its random eigenvalues. 

It turns out that as the dimension of the matrix increases the random distribution of the eigenvalues becomes deterministic provided we impose these two conditions on the matrix entries so that they don't get too large:
\[
\E [x_{ij}] = 0, \quad \mathrm{Var} [x_{ij}] = 1.
\]
Since the dimension is also getting large we need to normalize the matrix so that the eigenvalue distribution is confined to an interval independent of $N.$ For this reason, we consider the eigenvalues of the random matrix $W:=\frac{1}{\sqrt{N}}X$ with 
\[
\E [w_{ij}] = 0, \quad \mathrm{Var} [w_{ij}] = \frac{1}{N}.
\]
It turns out that the relative frequency histogram $\rho_N$ of the positions of the $N$ scaled random eigenvalues for large $N,$ is close to the following deterministic law $f$:
\[
\rho_N \approx f,
\]
where $f$ is a probability density supported in $[-2,2],$ which is a scaled semi-circle centered at $(0,0):$
\[
f(x) = \frac{1}{2\pi} \sqrt{4-x^2}.
\]
\begin{figure}[ht]
\centering
\includegraphics[width=0.8\textwidth]{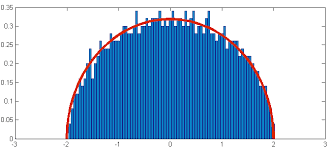}
\caption{Semi-circle law from \cite{feier2012methods}.}
\end{figure}

This was proved back in 1958 by Wigner, \cite{wigner1958distribution}. The approximation details will be clarified in the following chapters.
\end{section}

\pagebreak

\begin{section}{The Marchenko-Pastur law}

Another famous example of a deterministic distribution comes from the consideration of a large random sample covariance matrix.

For this, we consider firstly a random $N \times M$ matrix $X$ with i.i.d. centered entries with unit variance and then multiply it with its transpose. That means that we are considering the random covariance matrix $XX^T$ and get to study its random eigenvalues, after we normalize it as $W:=\frac{1}{N} XX^T.$

For this model, it turns out that as the dimensions $N$ and $M$ increase, the relative frequency histogram $\mu_N$ of the positions of the $N$ random eigenvalues of the Hermitian and positive semi-definite square matrix $\frac{1}{N}XX^T$ approaches the following law $g$:
\[
\mu_N \approx g,
\]
where $g$ is a probability density supported in $[\lambda_-,\lambda_+],$ which is called the Marchenko-Pastur distribution. The constants $\lambda_-,\lambda_+$ are defines as follows. First, we take the limit:
$$
d := \lim \limits_{N,M \to \infty} \frac{M}{N} \in (0,1],
$$
and the we define the endpoints:
\[
\lambda_\pm := (1 \pm d )^2.
\]
The distribution $g$ is then given by the following density,
\[
g(x) = \frac{1}{2\pi \lambda } \frac{\sqrt{(\lambda_+-x)(x-\lambda_-)}}{ x}, \quad x \in [\lambda_-,\lambda_+].
\]

\begin{figure}[ht]
\centering
\includegraphics[width=0.8\textwidth]{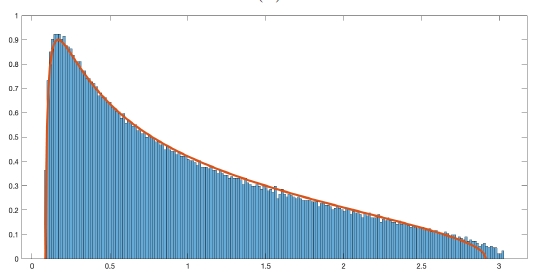}
\caption{Marchenko-Pastur law from \cite{bryson2021marchenko}.}
\end{figure}

This was proved back in 1973 by Marchenko and Pastur, \cite{marchenko1967distribution}. Again, the approximation details will be clarified in the next chapters.
\end{section}

\pagebreak

\begin{section}{The circular law}
We now turn our attention to non-Hermitian random matrix models which generally have complex eigenvalues. 
 
One famous such example is the case of a random matrix $X$ that has independent identically distributed complex entries, without any required symmetry in them. 
 
This means that we are considering an $N \times N$ random matrix $X$ with i.i.d. complex entries $(x_{ij})$ such that:
\[
\E [\mathrm{Re}(x_{ij})] = 0, \quad \mathrm{Var}[\mathrm{Re}(x_{ij})]=1/2, \]
\[
\E [\mathrm{Im}(x_{ij})] = 0, \quad \mathrm{Var}[\mathrm{Im}(x_{ij})]=1/2.
\]
If we consider now the 2-dimensional relative frequency histogram $c_N$ of the positions of the $N$ random eigenvalues of the scaled matrix $W:= \frac{1}{\sqrt{N}}X$ with expectation zero and complex variance $\frac{1}{N},$ we get that:
\[
c_N \approx h,
\]
 where $h$ is the complex uniform distribution on the unit disk, i.e.:
 \[
 h(z) = 1_{|z|\leq 1}.
 \]
 
\begin{figure}[ht]
\centering
\includegraphics[width=0.7\textwidth]{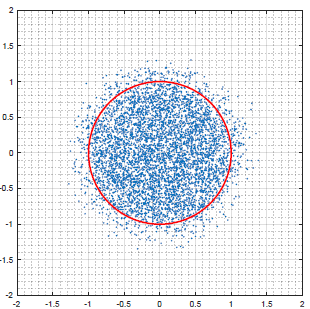}
\caption{Circular law from \cite{bothner2022complex}.}
\end{figure}
\end{section}

This is known as the circular law. For random matrices with a Gaussian distribution for the entries, the circular law was established in 1965 by Ginibre, see \cite{ginibre1965statistical}. Around 1980, Girko introduced an approach which allowed to establish the circular law for more general distributions of the matrix entries, see \cite{girko1985circular}.

\pagebreak

\begin{section}{The elliptic law}
The circular law result was extended in 1988 by Sommers, Crisanti, Sompolinsky and Stein to an elliptical law for ensembles of matrices with the properties described next, (see \cite{sommers1988spectrum}). 

Let $(\xi_1,\xi_2)$ be a random vector in $\mathbb{C}^2$ where both $\xi_1$ and $\xi_2$ have mean zero and unit complex variance. We say that a random matrix $X=(x_{ij})_{i,j=1}^N$ belongs to the complex elliptic ensemble if the following conditions hold:
\begin{itemize}
    \item (Independence). $\{ x_{ii} : i \geq 1\} \cup \{ (x_{ij},x_{ji}) : 1\leq i < j\}$ is a collection of independent random elements.
    \item (Common distribution). Each pair $(x_{ij},x_{ji}),$ $1\leq i < j$ is an i.i.d. copy of $(\xi_1,\xi_2).$
    \item (Flexibility of the main diagonal). The diagonal elements $\{ x_{ii} : i \geq 1 \}$ are i.i.d. with mean zero and finite variance.
    \item (Correlations). We have that $\E[\xi_1 \xi_2 ] = \rho, $ where $\rho \in (-1,1)$ is a universal correlation constant for the random matrix $X.$
\end{itemize}
If $d_N$ is the 2-dimensional relative frequency histogram of the positions of the random eigenvalues of the scaled version of this random matrix model $W:= \frac{1}{\sqrt{N}} X$, then we have that:
\[
d_N \approx k,
\]
where $k$ is the complex uniform distribution on the ellipse centered at $(0,0)$ and with its edges at $(1-\rho,0)$ and $(1+\rho,0):$
\[
k(z) := \frac{1}{\pi (1-\rho^2)} \mathbf{1}_{ \left \{ z \in \mathbb{C} \ \middle | \ \frac{(\mathrm{Re} z)^2}{(1+\rho)^2} + \frac{(\mathrm{Im} z)^2}{(1-\rho)^2} \leq 1 \right \} }.
\]
\begin{center}
\centering
\captionsetup{type=figure}
\includegraphics[scale=2]{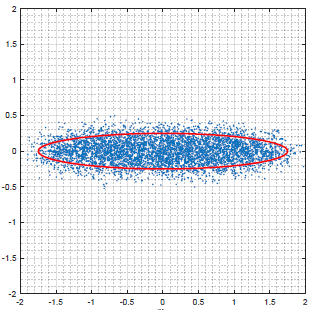}
\captionof{figure}{Elliptic law from \cite{bothner2022complex}.}
\end{center}

\end{section}

\begin{section}{Local laws}
Since the random matrix $W$ we are usually dealing with is normalized, we have that its spectrum is of constant order, or in other words the operator norm $\|W\|$ is of order $1.$ Since $W$ has $N$ eigenvalues, the typical separation of the eigenvalues is of order $\frac{1}{N}.$ Individual eigenvalues are expected to fluctuate from their mean locations but as it turns out for eigenvalues of random matrices these fluctuation are really small, eigenvalues are rigid.

This allows us to talk about convergence of the eigenvalue distribution in short scales on an interval where there is just a constant $\mathcal{O}(1)$ amount of them in contrast to a large $\mathcal{O}(N)$ amount of them. In general, if the fluctuations of a random process are not small this can turn out to be impossible, as it is for example in the Poisson point process.
\begin{center}
\centering
\captionsetup{type=figure}
\includegraphics[scale=0.3]{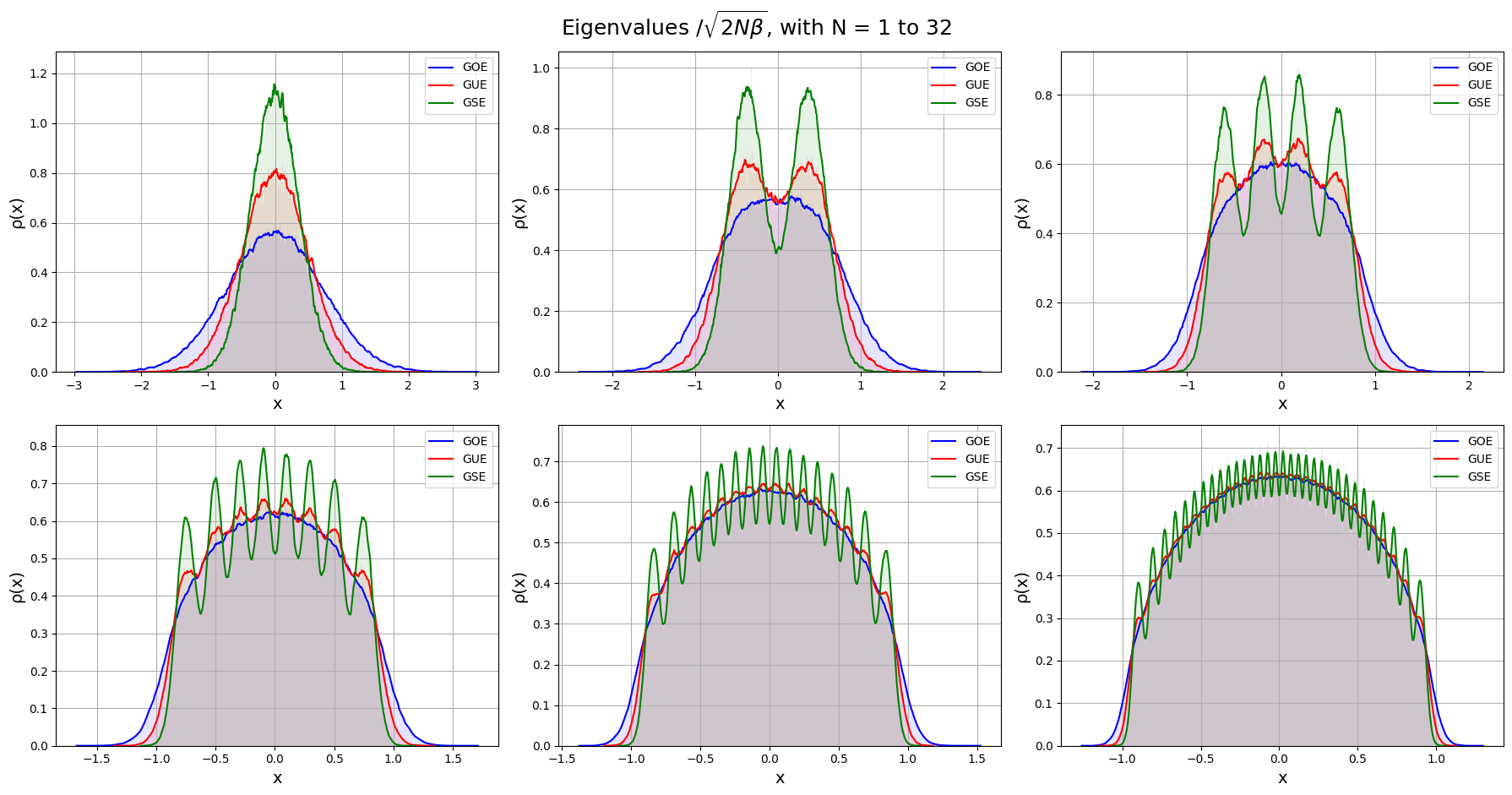}
\captionof{figure}{Local laws.}
\end{center}
This convergence on short scales for the eigenvalue distribution has numerous applications and is technically useful. We give a  description of it for our previous models:

\underline{Wigner matrix.} We will define a certain "zooming mechanism" regarding the "scale" of convergence of the random distribution of the eigenvalues, that is a mechanism for creating peaks at each interval in the real line with a constant amount of eigenvalues and a smoothing-out until the next one. Here, the typical separation of the eigenvalues is of order $\mathcal{O}(\frac{1}{N})$ everywhere in the support except from the edges and this mechanism can create peaks in each such small interval of that order where there is a constant amount of eigenvalues. Near the edges we have a typical separation of order $\mathcal{O}(\frac{1}{N^{2/3}}).$ We still get convergence to the semi-circle law with such a "bandwidth" aligned to the random eigenvalue distribution.

\underline{Covariance matrix.} We will define the same mechanism for the eigenvalue distribution of a random covariance matrix. There is a strong accumulation of eigenvalues near zero when $d=1$, which creates a singularity in the limiting Marchenko-Pastur distribution which approximates the random eigenvalue distribution. Due to this, the typical separation of eigenvalues there is of smaller order $\mathcal{O}(\frac{1}{N^2})$. Away from this point, the separation is of typical order $\mathcal{O}(\frac{1}{N})$ but again, near the right edge there must be a smaller amount of eigenvalues and so the separation is of slightly larger order $\mathcal{O}(\frac{1}{N^{2/3}}).$ We still get convergence to the Marchenko-Pastur distribution when we "zoom in" to these intervals but now greater care is needed for our mechanism process, especially near the singularity.

\underline{Circular law.} Here the zooming process is different since the eigenvalues are complex. We use a mechanism for creating peaks arising from small circles around each individual eigenvalue. Here the typical separation of the eigenvalues is of order $\mathcal{O}(\frac{1}{\sqrt{N}})$ everywhere inside the support of the unit disk. This means that we can create a random distribution which gives more mass to circles of diameter $\mathcal{O}(\frac{1}{\sqrt{N}})$ inside the unit disk while also following the random eigenvalue distribution. This would still converge to the the circular law for large $N.$

\underline{Elliptic law.} We use the same zooming process as we do with the local circular law to create small circles inside the domain of the ellipse of diameter $\mathcal{O}(\frac{1}{\sqrt{N}}),$ which is the typical separation of the random eigenvalues inside the ellipse. The eigenvalue distribution, aligned with this zooming mechanism still converges to the elliptic law.
\end{section}

\begin{section}{Discussion about the new results}
This section will be technical and will provide a description about the past results and the new original contributions that have been made during my work as a PhD student at Queen Mary university.

\underline{1. Local Marchenko-Pastur law.}

Let $X_N = X/\sqrt{N} $, where $X$ is a $N \times N$ matrix with entries $x_{ij} = \mathrm{Re}(x_{ij}) + i \mathrm{Im}(x_{ij})$, where $\mathrm{Re}(x_{ij}), \mathrm{Im}(x_{ij}) $ are i.i.d. real random variables with mean $0$ and variance $\frac{1}{2},$ so that $\mathbb{E} (x_{ij}) = 0$ and $\mathbb{E}|x_{ij}|^2=1 \ \text{for} \ i,j =1,...,N.$ We assume that the random variables have finite fourth moments.

Let $\lambda_i, \ i=1,...,N$ denote the eigenvalues of the positive semi-definite matrix $X_N^*X_N$ with $0\leq \lambda_1 \leq ... \leq \lambda_N .$ If we denote by $\rho_N$ the empirical spectral density of the eigenvalues:
$$ \rho_N(E) = \frac{1}{N} \# \left \{ i \leq N | s_i \leq E \right \}, $$
and by $\rho(E)$ the density of the Marchenko-Pastur distribution:
$$ \rho(E) = \frac{1}{2\pi} \sqrt{ \frac{4}{E} - 1}, $$
we have that $\rho_N \to P $ weakly in probability, where $P$ is the cumulative density of $\rho$. We show that this convergence holds locally at an optimal scale. Let's comment here about the previous results so far.

In \cite{erdos2012local}, Erd\"os-Schlein-Yau-Yin established the convergence of the empirical spectral density for general covariance matrices to the Marchenko-Pastur law in the bulk on small intervals for $X_N$ being a $M \times N$ matrix with $M<N$. They used a decomposition by minors for the diagonal elements of the resolvent to establish a self-consistent equation for the Stieltjes transform $S_N$ of $\rho_N$. Large deviation estimates and a continuity argument were then used to show the convergence of the spectral measure on small intervals (involving polynomial corrections) in the bulk  distribution. These methods have been extended to the "hard" edge for $N \times N$ matrices and logarithmic rather than polynomial corrections by Cacciapuoti-Maltsev-Schlein in \cite{cacciapuoti2013local}. More precisely, the authors showed that the fluctuation of the Stieltjes transform $\sqrt E S_N$ away from $\sqrt E S_{\rho}$ is on the order of $\sqrt{\frac{\sqrt{E}}{N\eta}}$ and they obtained convergence of the counting function of the eigenvalues everywhere including close to the "hard" edge. Eigenvalue rigidity with polynomial corrections for the "bulk" and "soft" edges for entries with sub-exponential decay can be found in Pillai-Yin \cite{pillai2014universality}.

In this work, we obtain optimal bounds on the expectations of high moments of the fluctuation $\Lambda = S_N - S_{\rho}$ on the optimal scale and the previous bound $\sqrt{\frac{\sqrt{E}}{N\eta}}$ gets improved to ${\frac{\sqrt{E}}{N\eta}}$ without any logarithmic corrections. Our methods and results apply to the "bulk" as well as to the "soft" and "hard" edges. The main objective of this work is to extend the results and the methodology of \cite{cacciapuoti2015bounds} to a "hard edge setup". This is the case when the limiting measure has a square root singularity near $0$ with typical distance between the eigenvalues of the order of $\frac{1}{N^2}.$ We were able to simplify the proof of Theorem \ref{t:sti} in \cite{cacciapuoti2015bounds} by avoiding different cases for the "bulk" and the "edges". Unlike in the Wigner case, where both edges are "soft", the presence of the "hard" edge at $0$ allows us to extend the bounds on the real part of the Stieltjes transform to the negative real line, thus also yielding a fluctuation for the individual eigenvalue near the "hard" edge that is decreasing with the eigenvalue number. In summary, this paper improves on \cite{cacciapuoti2013local} by removing the logarithmic corrections and improving the fluctuation bounds while keeping the same optimal scale of convergence. We also extended the proofs in \cite{gotze2016optimal,gotze2014rate} on fluctuations of quadratic forms to a "soft edge setup" by improving a factor of $|S_{\rho}|$ to a factor of $\mathrm{Im} (S_{\rho})$.

To show our results, we used the Stieltjes transform method which involves a parameter ${z \in \mathbb{C}} $ that depends on $N$. We define:
$$ S_N(z) = \int_{\mathbb{R}} \frac{1}{x-z} \rho_N(x) dx $$ and
$$ S(z) = \int_{\mathbb{R}} \frac{1}{x-z} \rho(x) dx, $$
where $z = E +i\eta$, with $E \in \mathbb{R}$ representing the position of an eigenvalue and $0<\eta(N)\leq 1$ representing the local scale of the convergence for the two distributions. It is proven that for:
\begin{equation} \label{Neta}
\frac{N\eta}{|\sqrt{z}|} \geq M,
\end{equation}
for a suitable large  constant $M$ and for each power $q$ with $1 \leq q \leq c_0 \left (\frac{N\eta}{|\sqrt{z}|} \right )^{1/8} $ and for each $z \in Z_{E,\eta} = \left \{E,\eta \in \mathbb{R} : E^2 + \eta^2 -4 |\eta | \leq 4E \right \} $ or $E<0:$
\begin{equation} \label{thm2}
\mathbb{P} \left ( |\sqrt{z} ( S_N (z) - S_{\rho}(z) ) | \geq K \frac{| \sqrt{z}|}{N\eta} \right ) \leq \frac{(Cq)^{cq^2}}{K^q} ,
\end{equation}
for some constants $C,c_0$ and for each $K>0.$

Relation \eqref{Neta} gives the optimal scaling according to the position of the eigenvalues, namely we have that for $ E $ close to $0:$  

$\eta \gtrsim \frac{1}{N^2} $ while for $E$ away from the "hard" edge: $ \eta \gtrsim \frac{1}{N} $.

Relation \eqref{thm2} gives the optimal bound on the convergence of the Stieltjes transforms. \\ In the "hard" edge the term $\sqrt{|z|}$ corrects this convergence, as we have that $|\sqrt{z} ( S_N (z) - S_{\rho}(z)) | \to 0$ in probability in the optimal scaling.

This theorem has as a consequence a bound on the convergence rate of the empirical spectral distribution to the Marchenko-Pastur distribution. Letting:
\[
P(E) = \int_0^E \rho(x) dx,
\]
we compare it to $\rho_N.$ With the same assumption as before, we have that:
\begin{equation} \label{thm3}
\mathbb{P} \left ( \left |\rho_N(E) - P(E)  \right | \geq K \min\left \{\sqrt{E},\frac{\log N}{N}\right \} \right ) \leq \frac{(Cq)^{cq^2}}{K^q},
\end{equation}
for all $E \in \mathbb{R},$ $K>0,$ $N>N_0$ and for each $q \in \mathbb{N}$ with $1 \leq q \leq c_0 \left (\frac{N\eta}{|\sqrt{z}|} \right )^{1/8}.$

Relation \eqref{thm3} is not optimal since it is known that the fluctuations between the number of eigenvalues predicted by the counting measure and the M-P distribution are of the order of $\frac{\sqrt{\log N}}{N}$ but the technique used here couldn't provide such a result.

Finally the rigidity of the eigenvalues was proven. There exist constants $C,c,N_0, \epsilon>0$ such that:
\begin{equation} \label{rigiditybulk}\mathbb{P} \left ( | \lambda_i  - \gamma_i | \geq K \frac{ \log N}{N} \left ( \frac{i}{N}  \right )\right) \leq \frac{(Cq)^{cq^2}}{K^q},\end{equation}
for $i=1,...,\lceil N/2 \rceil$, $N>N_0$, $K>0$, and $q \in \mathbb{N}$ with $q\leq N^\epsilon.$

Furthermore, for $i \leq \log N$ we have that:
\begin{equation} \label{rigidityedge}\mathbb{P} \left ( | \lambda_i  - \gamma_i | \geq K \frac{ i^2}{N^2}\right) \leq \frac{(Cq)^{cq^2}}{K^{q/2}}. \end{equation}
In this theorem the factor $\frac{i}{N}$ accounts for the higher density at the "hard" edge. Here we focus on hard-edge rigidity, since proofs of soft-edge rigidity require control of the largest eigenvalue which, to our knowledge, is not currently available in the case of truncated entries with four moments, in either the Wigner or the Sample Covariance case.

The terms $\gamma_i$ are defined as the classical locations of the eigenvalues predicted by the M-P distribution, that is:
$$ \int_0^{\gamma_i} \rho(E) dE = \frac{i}{N} .$$
In particular, we obtain the fluctuation of the eigenvalues near the "hard" edge to be of the order of $\frac{\log N}{N^2}$. The fluctuations of eigenvalues in the "bulk" and "soft" edges of both the Gaussian Unitary Ensemble and the Wishart Ensemble are known to be respectively of the order $\frac{\sqrt{\log N}}{N}$ in the bulk and $\frac{\sqrt{\log k}}{k^{1/3}N^{2/3}}$ for the $k$-th eigenvalue from the edge, $k \rightarrow \infty$ (see \cite{gustavsson2005gaussian, su2006gaussian}). To our knowledge similar results are not yet available for the "hard" edge.

\underline{2. Universality of correlated covariance matrices.}

In the second part of this thesis, we study the spectrum of  matrices of the form $X_{1}X_{2}^*,$ where the matrices $X_{1}$ and $X_{2}$ are correlated according to a parameter $\tau \in [0,1].$ This topic has a particularly wide application in Quantum Chromodynamics.

For $\tau=0$ we have two completely uncorrelated matrices whereas for $\tau=1$ we have two completely correlated matrices $(X_1=X_2)$ and we can recover the Marchenko-Pastur distribution in the limit $N \to \infty$.

The goal of this analysis is to determine the distribution (law and support) of the complex eigenvalues of $X_1X_2^*$ when the entries of each matrix are general i.i.d. random variables. This means that we are proving universality results, so that we can extend a recent work by Akemann, Byun and Kang in \cite{akemann2021non}, where the entries there were assumed to be Gaussian. Due to this assumption, the techniques that were used there could not be adapted to our analysis, so we followed the techniques used in \cite{alt2022local}.

Specifically, the model we are examining are matrices of the form $X_1X_2^*$ where
$$ X_1 = \sqrt{1+\tau} P + \sqrt{1-\tau}Q, \quad X_2 = \sqrt{1+\tau} P - \sqrt{1-\tau} Q $$
and $P,Q$ are $N \times N $ random matrices with i.i.d. entries and $\tau$ is a parameter that models the correlation between the matrices $X_1$ and $X_2$.

In \cite{akemann2021non} it was shown that the eigenvalue counting measure of the matrix $X=X_1 X_2^*$ (for random Gaussian entries of the matrices $P,Q$) converges to the measure $\hat{\mu}$, where
\begin{equation} \label{hatmeasure}
d\hat{\mu}(\zeta) = \frac{1}{1-\tau^2} \frac{1}{ 2 \pi |\zeta|} \mathbf{1}_{\hat{S}_{\tau}}(\zeta) d^2\zeta
\end{equation}
and $d^2 \zeta$ is the Lebesgue measure on the complex plane. The support $\widehat{S}_{\tau}$ of this measure is given by:
\begin{equation} \label{hatsupport}
\hat{S}_{\tau} := \left \{ \zeta = x+iy : \left ( \frac{ x - 2\tau}{1+\tau^2  } \right )^2 + \left ( \frac{y}{1-\tau^2 } \right )^2 \leq 1 \right \}.
\end{equation}
It represents a transition from the circular law to an ellipsoid and finally to the Marchenko-Pastur law according to the values of the correlation parameter $\tau$.

The universal result we prove here is the convergence of the empirical spectral density to the measure $d\hat{\mu}$ independently of the choice of the distributions for $P,Q$ and extend it on local scales.

The techniques used are based on the recent publication "Local elliptic law" \cite{alt2022local}, where the Dyson equation method is used for non-Hermitian matrices with correlated entries after they get "hermitized" through the Hermitization technique. The Dyson equation is a powerful technique and except for universality, it can also provide local laws for the limiting spectral measure.

The main idea in our analysis is to firstly prove a universal uniform limiting local elliptic law for the distribution of the eigenvalues of the auxiliary \textit{Dirac} matrix:
\beq \label{1:dirac} D:= \begin{pmatrix} 0 & X_1 \\ X_2^* & 0 \end{pmatrix} \eeq
and then recover the eigenvalues of the correlated covariance matrix $X_1 X_2^*$. The Hermitization matrix, resolvent matrix and then the Dyson equation are defined for this auxiliary matrix to provide the uniform elliptic distribution for its eigenvalues depending on the non-hermiticity parameter $\tau$, in a similar way to \cite{alt2022local}. Specifically, for the Dirac matrix, we have a local law given by the distribution:
\begin{equation}
d{\mu}(\zeta) := \frac{1}{1-\tau^2} \frac{1}{\pi} \mathbf{1}_{{S}_{\tau}}(\zeta) d^2 \zeta.
\end{equation}
with its support ${S}_{\tau}$ given by:
\begin{equation}
{S}_{\tau}:= \left \{ \zeta=x+iy : \left ( \frac{x}{1+\tau}\right )^2 + \left ( \frac{y}{1-\tau} \right )^2 \leq 1 \right \}.
\end{equation}
Much different analysis for the stability of the Dyson equation was now used to establish this local law as the equation is now given by $4\times 4$ matrices instead of $2\times 2$ as in \cite{alt2022local}. The least singular value is controlled after a bounded density assumption for the entries of $P,Q$ as in \cite{alt2021inhomogeneous}.

After establishing the local law for the Dirac matrix, the linearization technique as used for example in \cite{o2015products}, together with a careful change of variables is then enough to provide us with the limiting local spectral law for the initial correlated covariance matrix.

\end{section}
\chapter{Preliminaries}
Before going to the analysis of our results, we present here some preliminaries about random matrices and their spectral properties.

\begin{section}{Random matrices}
In this section, the notion of a random matrix is defined rigorously. We will be working with matrices which have their entries in the complex plane $\mathbb{C}.$ We denote by $\mathbb{C}^{N \times N}$ the space of $N \times N$ matrices $X$ with complex entries $(X_{ij})_{i,j=1}^N,$ which can be regarded as vectors in $\mathbb{C}^{N^2}.$ We can equip this space with a norm such as the operator norm $\| \ \cdot \ \|_{op}$ or the Hilbert-Schmidt norm $\| \ \cdot \ \|_{HS}.$ The operator norm is defined with respect to the Euclidean norm $\| \ \cdot \ \|_2$ in $\mathbb{C}^N$, as follows:
\[
\|X\|_{op} := \sup \left \{ \|Xu\|_2 \ :\ u \in \mathbb{C}^N \ , \ \|u\|_2=1 \right \},
\]
while the Hilbert-Schmidt norm $\| \ \cdot \ \|_{HS} $ is defined as:
\[
\|X\|_{HS} := \sqrt{\mathrm{Tr} (XX^*)},
\]
where $X^*$ is the conjugate transpose of $X$ and we also defined the \textit{trace} functional $\mathrm{Tr}: \mathbb{C}^{N \times N} \mapsto \mathbb{C}$ of a matrix $X$ as:
\[
\mathrm{Tr}(X) := \sum \limits_{i=1}^N X_{ii}.
\]
Some remarks about these norms are that they are equivalent to each other for any finite dimension $N \in \mathbb{N}$ and that the operator norm is in some sense similar to the maximum norm $\| \ \cdot \ \|_\infty$ in $\mathbb{C}^{N^2}$ while the Hilbert-Schmidt norm similar to the Euclidean norm $\| \ \cdot \ \|_2$ in $\mathbb{C}^{N^2}.$ The operator norm transforms $(\mathbb{C}^{N \times N},\|\ \cdot \ \|_{op})$ into a Banach space while the Hilbert-Schmidt norm transforms $(\mathbb{C}^{N \times N},\|\ \cdot \ \|_{HS})$ into a Hilbert space with corresponding inner product:
\[ 
\langle X,Y \rangle := \mathrm{Tr} (XY^*). 
\]
Another remark is that the convergence of a sequence of matrices $\left \{ X^{(k)} \right \}_{k=1}^\infty $, for $k \to \infty$ under any of these two norms in $\mathbb{C}^{N \times N}$ is equivalent to the convergence of each matrix entry $\left \{ X_{ij}^{(k)} \right \}_{k=1}^\infty $ in $\mathbb{C}$ for $i,j=1,...,N,$ under the usual complex norm.

A matrix $X \in \mathbb{C}^{N \times N}$ is called \textit{Hermitian} if the equality $X=X^*$ holds. An \textit{eigenvalue} $\lambda \in \mathbb{C}$ of a matrix $X \in \mathbb{C}^{N \times N}$ is a complex number such that $Xu=\lambda u$ for some $u\in \mathbb{C}^N$ different than the zero vector. Counting multiplicities, we can say that each complex matrix in $\mathbb{C}^{N \times N}$ has exactly $N$ eigenvalues. The set $\mathrm{Spec}(X)=\{ \lambda_1,...,\lambda_N \} $ of the eigenvalues of a matrix $X$ is called its \textit{spectrum.} We say that a matrix $X\in \mathbb{C}^{N\times N}$ is \textit{invertible} if there exists a matrix $Y\in \mathbb{C}^{N \times N}$ such that $XY=YX=I$ and we write then $Y=X^{-1}.$ We also define the \textit{determinant} functional $\det : \mathbb{C}^{N \times N} \mapsto \mathbb{C}$ \ as \ $\det(X) = \prod \limits_{i=1}^N \lambda_i,$ where $\lambda_1,...,\lambda_N$ are the $N$ eigenvalues of the matrix $X.$

Some properties of the trace functional are the following:
\begin{itemize}
    \item It is a continuous linear functional from $(\mathbb{C}^{N \times N},\|\ \cdot \ \|_{op})$ to $(\mathbb{C},\|\ \cdot \ \|).$
    \item For each matrix $X\in \mathbb{C}^{N \times N}$ with corresponding eigenvalues $\lambda_1,...,\lambda_N$ the equality $\mathrm{Tr}(X) = \sum \limits_{i=1}^N \lambda_i$ holds.
    \item For each $S,M \in \mathbb{C}^{N \times N}$ we have that $\mathrm{Tr}(SM) = \mathrm{Tr} (MS) .$
\end{itemize}
In the following we present two important theorems about the eigenvalue structure of a matrix. We note that if a matrix $X$ is Hermitian then all of its eigenvalues are real. We denote by $\mathrm{diag}(a_1,...,a_N)$ the \textit{diagonal} matrix $D \in \mathbb{C}^{N \times N}$ such that $D_{ii}=a_i$ for $i=1,...,N$ and $D_{ij} = 0$ for $i\neq j \in \{1,...,N\}.$ A \textit{triangular} matrix is a matrix $T\in \mathbb{C}^{N \times N}$ such that $T_{ij}=0$ for all $i>j.$ A \textit{unitary} matrix $U \in \mathbb{C}^{N \times N}$ is an invertible complex matrix such that $U^{-1}=U^*.$ 
\begin{thm}[Spectral theorem]
For any Hermitian matrix $X\in \mathbb{C}^{N \times N}$ there exists a unitary matrix $U\in \mathbb{C}^{N \times N}$ such that $U^{-1} X U = \mathrm{diag}(\lambda_1,...,\lambda_N),$ where $\lambda_1,...,\lambda_N$ are the $N$ real eigenvalues of the Hermitian matrix $X.$
\end{thm}
\begin{thm}[Schur decomposition] For any matrix $X\in \mathbb{C}^{N \times N}$ there exists a unitary matrix $U \in \mathbb{C}^{N \times N}$ such that $U^{-1} X U = T,$ where $T$ is a triangular matrix, satisfying $T_{ii}=\lambda_i,$ for $i=1,...,N$ where $\lambda_i$ are the eigenvalues of $X.$
\end{thm}
For a specific $i \in \{1,...,N\}$, the $i$-th eigenvalue $\lambda_X^i$ of a matrix $X \in \mathbb{C}^{N \times N}$ is a continuous function in the topological sense. This means that if $X_k \to X$ for a sequence of matrices $\{X_k\}_{k=1}^\infty$ under the operator or Hilbert-Schmidt norm, then the eigenvalues $\{\lambda_{X_k}^i\}_{i=1}^N$ can be numbered in such a way that $\lambda_{X_k}^i \to \lambda_X^i$. See \cite{Artin:1998}, pages 138-139. 

This is a consequence of the fact that the roots of a polynomial are continuous functions of the polynomial coefficients under a suitable norm, and this fact is applied to the \textit{characteristic polynomial} of a matrix, $p(x):= \det(xI - X)$, see again \cite{Artin:1998}.

We now turn our attention to the definition of a random matrix.
\begin{definition}[Random matrix]
Let $(\Omega, \mathcal{A}, \mathbb{P})$ be a probability space and $N \in \mathbb{N}.$ An $N \times N$ random matrix is a measurable map $X: (\Omega,\mathcal{A}, \mathbb{P}) \mapsto (\mathbb{C}^{N \times N}, \mathcal{B}^{N^2}),$ where $\mathcal{B}^{N^2}$ denotes the Borel $\sigma$-algebra on $\mathbb{C}^{N \times N}.$
\end{definition}
It is clear that $X$ is a measurable function if and only if all entries $X(i,j)$ are measurable functions in $(\mathbb{C},\mathcal{B}_{\mathbb{C}})$ where $\mathcal{B}_{\mathbb{C}}$ denotes the Borel $\sigma$-algebra in $\mathbb{C}.$
We show next that the eigenvalues of a random matrix are also measurable functions. Specifically, we show that the eigenvalue function:
\[
\omega \mapsto \lambda_{X(\omega)}^i: (\Omega, \mathcal{A}, \mathbb{P}) \mapsto \mathbb{C},
\]
where $\lambda_{X(\omega)}^i$ denotes the $i$-th eigenvalue of the random matrix $X(\omega)$ is measurable, which makes it a complex-valued random variable.

We know that the eigenvalue function:
\begin{align*}
\mathrm{eig}^i: \ &\mathbb{C}^{N \times N} \mapsto \mathbb{C} , \\
&X \mapsto \lambda_X^i
\end{align*}
is continuous so it is measurable. The map $X: \Omega \mapsto \mathbb{C}^{N \times N}$ is also measurable by definition, hence the composition $\lambda_X^i:= \mathrm{eig}^i \circ X$ is also measurable.

This property of the eigenvalue function allows us to study the eigenvalues of a random matrix in the context of probability theory. For this, we firstly need to define the concept of a random measure.

\end{section}

\begin{section}{Random probability measures}
We denote the set of all measures on $(\mathbb{C},\mathcal{B}_{\mathbb{C}})$ by $\mathcal{M}(\mathbb{C})$, the set of probability measures by $\mathcal{M}_1(\mathbb{C})$ and the set of sub-probability measures by $\mathcal{M}_{\leq 1} (\mathbb{C})$. If $f:\mathbb{C} \to \mathbb{C}$ is a measurable function, we write:
\[
\langle \mu, f \rangle := \int_{\mathbb{C}} f d\mu
\]
We are interested in studying the convergence behaviour of sequences of measures in $\mathcal{M}_{1}(\mathbb{C})$ where the limit may generally lie in $\mathcal{M}_{\leq 1}(\mathbb{C})$. For this, we define the sets of test functions as follows:
\begin{enumerate}
    \item $C_b(\mathbb{C}) := \{ f: \mathbb{C}\mapsto \mathbb{C} \ | \ f \ \text{continuous and bounded}\}$
    \item $C_0(\mathbb{C}) :=  \{ f: \mathbb{C}\mapsto \mathbb{C} \ | \ f \ \text{continuous and vanishes at infinity}\}$
    \item $C_c(\mathbb{C}) :=  \{ f: \mathbb{C}\mapsto \mathbb{C} \ | \ f \ \text{continuous with compact support} \}$
\end{enumerate}
\begin{definition}
Let $\{\mu_N\}_{N=1}^\infty$ be a sequence of probability measures in $\mathcal{M}_1(\mathbb{C}).$
\begin{itemize}
    \item The sequence $\{\mu_N\}_{N=1}^\infty$ is said to converge weakly to an element $\mu \in \mathcal{M}_1(\mathbb{C}),$ if
    \[
    \forall f \in C_b(\mathbb{C}): \lim \limits_{N \to \infty} \langle \mu_N,f \rangle = \langle \mu,f \rangle
    \]
    \item The sequence $\{\mu_N\}_{N=1}^\infty$ is said to converge vaguely to an element $\mu \in \mathcal{M}_{\leq 1}(\mathbb{C}),$ if
    \[
    \forall f \in C_c(\mathbb{C}): \lim \limits_{N \to \infty} \langle \mu_N,f \rangle = \langle \mu,f \rangle
    \]
\end{itemize}
\end{definition}
Some remarks about the definition are that weak convergence of measures implies vague convergence, what's more weak and vague limits are unique. The target measures of weak convergence of probability measures are again probability measures. This is because $1_{\mathbb{C}} \in C_b(\mathbb{C})$ and so we must have that $\mu(\mathbb{C})=1.$ The target measures of vague convergence are exactly sub-probability measures, see for example \cite{kallenberg1997foundations}, page 143.

We will transform $\mathcal{M}_1(\mathbb{C})$ into a metric space by defining an appropriate distance function on it which metrizes the weak convergence of measures. We will define it according to test functions $f\in \mathbb{C}_b$. We then have that
\begin{itemize}
    \item The following metric:
    \[
    d(\mu,\nu) := \sup \left \{ \left | \int_{\mathbb{C}} f d\mu  - \int_{\mathbb{C}} f d \nu \right |: \|f\|_\infty \leq 1 \right \}
    \]
    metrizes weak convergence of measures, that is $d(\mu_N,\mu) \to 0 \Leftrightarrow \mu_N \to \mu$ {weakly}. 
    \item Equipped with this metric $(\mathcal{M}_1(\mathbb{C}),d)$ becomes a separable but not complete metric space.
\end{itemize}
The metric above is called the \textit{Prokhorov metric} which can be found for example in \cite{dudley2002real}, pages 394-395. 

Since $\mathcal{M}_1(\mathbb{C})$ is a metric space with a corresponding Borel algebra it is also a measurable space, so we can then study $\mathcal{M}_1(\mathbb{C})-$valued random variables which we will call random measures.

\begin{definition}[Random probability measure]
Let $(\Omega, \mathcal{A}, \mathbb{P})$ be a probability space. A random probability measure on $(\mathbb{C},\mathcal{B}_{\mathbb{C}})$ is a measurable map $\mu : \Omega \mapsto \mathcal{M}_1({\mathbb{C}}).$
\end{definition}

We remark that if $\mu$ is a random probability measure $\mu: \Omega \mapsto \mathcal{M}_1(\mathbb{C})$ and $f:\mathbb{C} \mapsto \mathbb{C}$ is a measurable bounded function, then the map $\omega \mapsto \langle \mu(\omega), f \rangle$ is also measurable and bounded by $\|f\|_\infty.$

Based on a random probability measure, we can also define its  "expected" deterministic measure. If $\mu$ is a random probability measure $\mu : \Omega \mapsto \mathcal{M}_1(\mathbb{C})$ and $B \in \mathcal{B}$ then $\mu(\omega,B)$ is a random variable, which means that we can consider its expectation $\mathbb{E} \mu (B)$ as the expected mass that $\mu$ would prescribe to the set $B$. The map $B \mapsto \mathbb{E}\mu(B)$ is then the deterministic "expected" measure on $(\mathbb{C},\mathcal{B}_{\mathbb{C}}).$

\begin{definition}[Expected measure] \label{exp-m} Let $(\Omega,\mathcal{B},\mathbb{P})$ be a probability space and $\mu$ a random probability measure on $(\mathbb{C},\mathcal{B}).$ Then the map:
\[
\overline{\mu} : \mathcal{B} \mapsto [0,1]
\]
with 
\[
\overline{\mu}(B)= \int_\Omega \mu(\omega,B) d\mathbb{P}(\omega) = \mathbb{E}\mu(B)
\]
is an element of $\mathcal{M}_1(\mathbb{C})$ and is called the expected measure of $\mu.$
\end{definition}

We will define three notions of convergence of random probability measures on $(\mathbb{C},\mathcal{B}_{\mathbb{C}})$, weak convergence in expectation, weak convergence in probability and weak convergence almost surely.

\begin{definition}[Weak convergence in expectation] Let $(\mu_n)_{n=1}^\infty$ be a sequence of random probability measures on $(\mathbb{C},\mathcal{B}_{\mathbb{C}})$ and $\mu$ a random probability measure on the same space. We say that $(\mu_N)_{n=1}^\infty$ converges weakly in expectation to $\mu,$ if the sequence of expected measures $(\mathbb{E}\mu_n)_{n=1}^\infty$ converges weakly to the expected measure $\mathbb{E}\mu.$
\end{definition}

\begin{definition}[Weak convergence in probability] We say that a sequence of random measures $(\mu_N)_{N=1}^\infty$
converges weakly in probability to $\mu$, if $\langle \mu_N,f \rangle$ converges to $\langle \mu, f \rangle$ in probability, for all $f \in C_b.$
\end{definition}

\begin{definition}[Almost sure weak convergence] We say that a sequence of random measures $(\mu_N)_{N=1}^\infty$ converges weakly to $\mu$ almost surely, if $\langle \mu_N, f \rangle$ converges to $\langle \mu , f \rangle $ almost surely, for all $f \in C_b.$
\end{definition}

We recall here that the definition of $\langle \cdot , \cdot \rangle$ contains a random measure integral which makes it a random variable. So we can use the definitions of almost sure convergence and convergence of probability for random variables. See for example \cite{dudley2002real}, pages 516-517.
\end{section}

\begin{section}{The empirical spectral distribution} \label{ESD}
Let $X$ be an $N \times N$ random matrix on $(\Omega,\mathcal{A},\mathbb{P}),$ then the \textit{empirical spectral distribution} (\textit{ESD}) $\sigma_N$ of $X$ is the random probability measure on $(\mathbb{C},\mathcal{B}_{\mathbb{C}})$ given by:
\begin{align*}
\sigma_N : \ &\Omega \times \mathcal{B}_{\mathbb{C}} \mapsto [0,1] \\
&(\omega,B) \mapsto \sigma_N(\omega,B) = \frac{1}{N} \sum \limits_{k=1}^N \delta_{\lambda_{X(\omega)}^k}(B),
\end{align*}
where $\delta_{\lambda_X^k}(\cdot)$ is the Dirac measure for the $k-$th eigenvalue of the random matrix $X,$ i.e.
\[
\delta_{\lambda_X^k}(B) = \begin{cases}
    1 , \ \text{if} \ \lambda_X^k \in B \\
    0 , \ \text{if} \ \lambda_X^k \notin B
\end{cases} 
\]

Note that $\sigma_N$ is indeed a random probability measure on $(\mathbb{C},\mathcal{B}_{\mathbb{C}})$. This is because
$\sigma_N(\omega)$ is a convex combination of probability measures and thus again a probability measure. On the other hand, if $B \in \mathcal{B}_{\mathbb{C}}$ is arbitrary, then $\sigma_N(B)$ is clearly measurable, making it a random measure.

For any measurable set $B\subset \mathbb{C}$, $\sigma_N(B)$ gives the proportion of the $N$ eigenvalues that fall into the set $B.$ Thus $\sigma_N$ carries information about the location of the eigenvalues and we are particularly interested in the case $N \to \infty.$

\end{section}
\chapter{The Stieltjes transform method}

In order to analyze properties of the empirical spectral distribution, it is useful to use a tool that relates this distribution with entries of the random matrix or another matrix related to it (see section \ref{s:Resol}). For this, we can use a suitable transform of the measure.

\begin{definition}[Stieltjes transform]
Let $\mu$ be a finite measure on some measurable linear space $Y$. The Stieltjes transform $S_\mu$ of $\mu$ is the map:
\begin{align*}
S_\mu : \ &Y \setminus \mathrm{supp}(\mu) \mapsto Y \\ \\
&y \mapsto \int_{\mathrm{supp}{(\mu)}} \frac{1}{w-y} d\mu(w)
\end{align*}
\end{definition}

The Stieltjes transform is defined via an integral involving our measure of interest.

\begin{section}{Distributions on the real line} \label{S-Re}
In the case that the measure is supported on the real line, we define the Stieltjes transform map as:
\begin{align*}
&S_\mu: \mathbb{C} \setminus \mathrm{supp}(\mu) \mapsto \mathbb{C} \\
&z \mapsto \int_{\mathbb{R}} \frac{1}{x-z} d\mu(x)
\end{align*}
The Stieltjes transform is now defined via an integral of a complex function with respect to our measure of interest. We remark some properties of the Stieltjes transform:
\begin{itemize}
    \item $\mathrm{Im}(z)\geq 0 \Leftrightarrow \mathrm{Im} S_\mu (z) \geq 0.$
    \item $S_\mu(\overline{z}) = \overline{S_\mu (z) }.$
    \item $S_\mu$ is uniquely determined by its restriction $S_\mu : \mathbb{C}_+ \to \mathbb{C}_+.$
    \item $|S_\mu(z)|\leq \frac{1}{|\mathrm{Im}(z)|}$.
    \item $S_\mu$ is holomorphic and in particular, $S_\mu$ is continuous and can be represented by a power series around $z_0 \in \mathbb{C}\setminus \mathrm{supp}(\mu)$ and is infinitely differentiable.
    \item The derivatives of the Stieltjes transform are given by
\[
S_\mu^{(k)}(z) = \int_{\mathrm{supp}(\mu)} \frac{k!}{(w-z)^{k+1}} d\mu(w),
\]
where $S_\mu^{(k)}$ denotes the $k$-th derivative of $S_\mu.$
\end{itemize}
These properties can be found for example in \cite{tao2023topics}, pages 169-171. We now remark an important measure-retrieval formula.
\begin{prop}[Retrieval of measure] For any bounded interval $I \subset \mathbb{R}$ with endpoints $a<b,$ we have the following formula:
\beq \label{retr1}
\mu(a,b) + \frac{1}{2} ( \mu(\{a\}) + \mu(\{b\} ) = \lim \limits_{\eta \to 0^+} \frac{1}{\pi} \int_I \mathrm{Im} S_\mu (E+i\eta) dE,
\eeq
\end{prop}
\begin{proof}
Let $I$ be an interval with endpoints $a<b$ and $\eta>0.$ By Fubini's theorem we get that:
\[
\begin{aligned}
\frac{1}{\pi} \int_I \operatorname{Im} S_\mu(E+i \eta) \mathrm{d} E &=\frac{1}{\pi} \int_I \int_{\mathbb{R}} \frac{\eta}{(x-E)^2+\eta^2} \mathrm{d} \mu (x) \mathrm{d} E \\
&=\frac{1}{\pi} \int_{\mathbb{R}} \int_I \frac{\eta}{(x-E)^2+\eta^2} \mathrm{d} E \mathrm{d}\mu (x) \\
&=\frac{1}{\pi} \int_{\mathbb{R}} \int_a^b \frac{\eta}{(x-E)^2+\eta^2} \mathrm{~d} E \mathrm{d} \mu (x)
\end{aligned}
\]
Now since
\[
\begin{aligned}
\int_a^b \frac{\eta}{(x-E)^2+\eta^2} \mathrm{~d} E &=\frac{1}{\eta} \int_a^b \frac{1}{\left(\frac{E-x}{\eta}\right)^2+1} \mathrm{~d} E \\
&=\int_{\frac{a-x}{\eta}}^{\frac{b-x}{\eta}} \frac{1}{E^2+1} \mathrm{~d} E \\
&=\arctan \left(\frac{b-x}{\eta}\right)-\arctan \left(\frac{a-x}{\eta}\right),
\end{aligned}
\]
and $\arctan : \mathbb{R} \mapsto \left(-\frac{\pi}{2},\frac{\pi}{2}\right)$ is strictly increasing with $\lim \limits_{x \rightarrow \pm \infty} \arctan (x)=\pm \frac{\pi}{2}$, we obtain
\[
\lim _{\eta \to 0^+}\left[\arctan \left(\frac{b-x}{\eta}\right)-\arctan \left(\frac{a-x}{\eta}\right)\right]= \begin{cases}\pi , & \text { if } x \in(a, b) \\ 0, & \text { if } x \notin[a, b] \\ \frac{\pi}{2}, & \text { if } x=a \ \text{or} \ x=b .\end{cases}
\]
Thus, by dominated convergence we find
\[
\begin{aligned}
\lim _{\eta \to 0^+} \frac{1}{\pi} \int_I \operatorname{Im} S_\mu(E+i \eta) \mathrm{d} E &=\lim _{\eta \to 0^+} \frac{1}{\pi} \int_{\mathbb{R}} \arctan \left(\frac{b-x}{\eta}\right)-\arctan \left(\frac{a-x}{\eta}\right) \mathrm{d} \mu (x) \\
&=\int_{\mathbb{R}} \mathbf{1}_{(a, b)}(x)+\frac{1}{2} \mathbf{1}_{\{a, b\}}(x) \mathrm{d} \mu (x) \\
&=\mu((a, b))+\frac{1}{2}(\mu(\{a\})+\mu(\{b\})).
\end{aligned}
\]
\end{proof}
\begin{cor} \label{re-cor} For any bounded interval $I \subset \mathbb{R}$ with $\mu(\partial I)=0,$ we find that:
\[
\mu(I) = \lim \limits_{\eta \to 0^+} \frac{1}{\pi} \int_I  \mathrm{Im} S_\mu (E+i\eta) dE
\]
\end{cor}
The previous proposition and corollary suggest that any finite measure $\mu$ on $(\mathbb{R},\mathcal{B}_{\mathbb{R}})$ is uniquely determined by $S_\mu$ or equivalently the map $\mu \mapsto S_\mu$ is injective.
\end{section}

\begin{section}{Distributions on the complex plane} \label{S-Co}
In the case that the measure is supported on the complex plane, we define the Stieltjes transform map as:
\begin{align*}
&S_\mu: \mathbb{C} \setminus \mathrm{supp}(\mu) \mapsto \mathbb{C} \\
&z \mapsto \int_{\mathbb{C}} \frac{1}{w-z} d\mu(w)
\end{align*}
We remark that all the previous properties of the Stieltjes transform in the real line still hold in the complex plane. Specifically, we have that:
\begin{itemize}
    \item $\mathrm{Im}(z)\geq 0 \Leftrightarrow \mathrm{Im} S_\mu (z) \geq 0.$
    \item $S_\mu(\overline{z}) = \overline{S_\mu (z) }.$
    \item $S_\mu$ is uniquely determined by its restriction $S_\mu : \mathbb{C}_+ \to \mathbb{C}_+.$
    \item $|S_\mu(z)|\leq \frac{1}{|\mathrm{Im}(z)|}$.
    \item $S_\mu$ is holomorphic and in particular, $S_\mu$ is continuous and can be represented by a power series around $z_0 \in \mathbb{C}\setminus \mathrm{supp}(\mu)$ and is infinitely differentiable.
    \item The derivatives of the Stieltjes transform are given by
\[
S_\mu^{(k)}(z) = \int_{\mathrm{supp}(\mu)} \frac{k!}{(w-z)^{k+1}} d\mu(w),
\]
where $S_\mu^{(k)}$ denotes the $k$-th derivative of $S_\mu.$
\end{itemize}
The proofs are similar to the ones for the real-case distributions, see \cite{tao2023topics}, pages 169-171.

The measure-retrieval formula is different in the complex case. We have the following proposition:
\pagebreak

\begin{prop}[Retrieval of measure] \label{re-comp} Let $\mu$ be a finite measure on $(\mathbb{C},\mathcal{B}_{\mathbb{C}})$ with corresponding Stieltjes transform $S_\mu,$ such that $d\mu = f(z)d^2 z,$ we then get that:
\beq \label{retr2}
f(z) = \frac{1}{\pi} \partial_{\overline{z}} S_\mu (z).
\eeq
\end{prop}
Some remarks about this proposition are the following. 

Here, $d^2 z$ denotes the Lebesgue measure applied on the complex domain, while $f:\mathbb{C} \mapsto [0,1]$ is a complex density function, i.e. $f(z) = f_{\mathrm{Re}(Z),\mathrm{Im}(Z)} (\mathrm{Re}(z),\mathrm{Im}(z))$ is the joint density of the real and imaginary part of a complex random variable $Z$ evaluated at the point $z \in \mathbb{C}.$ 

The symbol $\partial_{\overline{z}}$ denotes the \textit{Wirtinger} derivative of a complex function $g(z)$, i.e.
\[
\partial_{\overline{z}} g = \frac{\partial g}{\partial \overline{z}} := \frac{1}{2} \left ( \frac{\partial g}{\partial x} + i \frac{\partial g}{\partial y} \right ), \ \ \text{whenever} \ z=x+iy.
\]
For the proof, see section \ref{wirtinger}.
\end{section}

\begin{section}{Stieltjes transform and weak convergence of measures}
Since $S_\mu$ carries all the information about $\mu,$ $S_\mu$ can also be used to analyze weak convergence of probability measures.

\begin{thm}[Convergence theorem] \label{converge}
Let $(Y,\mathcal{B}_Y)$ be a measurable linear normed topological space (either $\mathbb{C}$ or $\mathbb{R})$ with corresponding Borel algebra $\mathcal{B}_Y$ and let $\mu$ be a probability measure on that space. If $Z \subset Y\setminus \mathrm{supp}(\mu)$ is a set that has at least one accumulation point in $Y\setminus \mathrm{supp}(\mu),$ then we get the following statement:

For a sequences of measures $( \mu_n )_{n=1}^\infty \in \mathcal{M}_1(Y)$ and a measure $\mu \in \mathcal{M}_1(Y),$ we have that
\[
\mu_n \to \mu \ \textit{weakly} \Leftrightarrow \forall y \in Z: S_{\mu_n}(y) \to S_\mu(y),
\]
where we remind that $S_\mu(y) = \int_{\mathrm{supp}(\mu)} \frac{1}{w-y}d\mu(w).$
\end{thm}

In order to prove theorem \ref{converge} we need the following three Lemmas:
\begin{lem}[Helly's selection theorem] \label{helly}
Let $(\mu_n)_{n=1}^\infty$ be a sequence in $\mathcal{M}_1(Y)$, then there exists a subsequence $(\mu_{n_k})_{k=1}^\infty$ and a sub-probability measure $\mu \in \mathcal{M}_{\leq 1}(Y)$ such that $\mu_{n_k} \to \mu$ vaguely.
\end{lem}

\begin{proof} 
The Lemma and its proof can be found for example in \cite{durrett2019probability}, page 108.
\end{proof}

\begin{lem} \label{subseq}
Let $(\mu_n)_{n=1}^\infty \in \mathcal{M}_1(Y)$ and $\mu \in \mathcal{M}_{\leq 1}(Y).$ Then $(\mu_n)_{n=1}^\infty$ converges weakly (respectively vaguely) to $\mu$ if every subsequence $(\mu_n)_{n \in J}, J \subset \mathbb{N}$ has a further subsequence $(\mu_n)_{n \in I}, I \subset J$ that converges weakly (respectively vaguely) to $\mu$.
\end{lem}

\begin{proof}
We prove this by contradiction. Suppose that $(\mu_n)_{n=1}^\infty$ doesn't converge weakly (respectively vaguely) to $\mu$. Then we can find a continuous and bounded (respectively with compact support) function $f: Y \mapsto Y$ and an $\epsilon > 0$ such that
$\left | \langle \mu_n , f \rangle -  \langle \mu , f \rangle \right | \geq \epsilon,$ for all $n$ in an infinite set $J \subset \mathbb{N}.$ But now we can find a subsequence $(\mu_n)_{n \in I}, I \subset J$ that converges weakly (respectively vaguely) to $\mu$. This means that we can find $n \in I \subset J$ such that $\left | \langle \mu_n , f \rangle -  \langle \mu , f \rangle \right | < \epsilon,$ which is a contradiction.
\end{proof}

\begin{lem} \label{vague-weak} If we have a sequence of probability measures that converges vaguely to a probability measure on $\mathcal{M}_1(Y),$ then vague convergence can be strengthened to weak convergence.
\end{lem}

\begin{proof}
For this Lemma and its proof, check for example \cite{chung2001course}, page 93.
\end{proof}

We now continue with the proof of our main theorem in this section.

\begin{proof}[Proof of Theorem \ref{converge}] The $"\Rightarrow"$ is obvious since the function $y \mapsto \frac{1}{w-y}$ is continuous and bounded. To show $"\Leftarrow"$, we will use Lemma \ref{subseq}. Let $(\mu_n)_{n \in J}, J \subset \mathbb{N}$  be a subsequence of $(\mu_n)_{n=1}^\infty$. Then, by Lemma \ref{helly} there exists a further subsequence $(\mu_n)_{n \in I}, I \subset J$ such that $(\mu_{n})_{n \in I} \to \nu$ vaguely for some $\nu \in \mathcal{M}_{\leq 1}(Y).$ Since $y \mapsto \frac{1}{w-y}$ vanishes at infinity, it follows that $S_{(\mu_n)_{n \in I}}(y) \to S_{\nu}(y)$ and therefore $S_{\nu}(y) = S_{\mu}(y),$ for all $y \in Z.$ Since the functions are holomorphic, we establish that $S_\mu = S_{\nu}.$ By the retrieval of measure identities, \eqref{retr1} and \eqref{retr2} we conclude that $\mu = \nu$. For any other subsequence $(\mu_n)_{n \in J}, J \subset \mathbb{N}$ we get by the same arguments, that every further subsequence should converge to $\mu$ vaguely. Therefore, $\mu_n \to \mu,$ vaguely. Because all measures involved are probability measures, by using Lemma \ref{vague-weak}, we conclude that $\mu_n \to \mu$ weakly, as we want.
\end{proof}

\begin{prop}[Convergence theorem for random measures] Let $(\mu_n)_{n=1}^\infty$ be a sequence of random probability measures on $\mathcal{M}_1(Y)$ and $\mu$ a deterministic probability measure on the same space. Then we have the following equivalences:
\begin{itemize}
    \item $\mu_n \to \mu$ \textit{weakly in expectation} $\Leftrightarrow \mathbb{E}S_{\mu_n}(y) \to S_{\mu}(y)$ \textit{for all} $y \in Z .$
    \item $\mu_n \to \mu$ \textit{weakly in probability} $\Leftrightarrow S_{\mu_n}(y) \to S_\mu(y)$  \textit{in probability for all} $y \in Z .$
    \item $\mu_n \to \mu$ \textit{weakly almost surely} $\Leftrightarrow S_{\mu_n}(y) \to S_{\mu}(y)$ \textit{almost surely for all} $y \in Z.$
\end{itemize}
\end{prop}

\begin{proof}
The first statement comes from our main theorem in this section, considering that:
\[
\mathbb{E} S_{\mu_n}(y) = \mathbb{E} \int_{\mathrm{supp}(\mu)}\frac{1}{w-y}d\mu(w) \overset{(*)}{=} \int_{\mathrm{supp}(\mu)}\frac{1}{w-y}d \mathbb{E} \mu(w) = S_{\E \mu_n}(y),
\]
where $d \mathbb{E} \mu(w)$ denotes the "expected measure" as defined in \ref{exp-m}. The $(*)$ equality is justified because it holds generally that:
\[
\E \int f d\mu = \int f d\E \mu,
\]
where $\mu$ is a random probability measure, $\E d \mu$ its expected measure and $f: Y \mapsto Y$ any bounded measurable function. This can be found in \cite{kallenberg2017random}, page 53.

For the second statement, the $"\Rightarrow"$ direction is obvious since the function $y \mapsto \frac{1}{y-w}$ is continuous and bounded. The proof of the $"\Leftarrow"$ direction can be found in 
\cite{anderson2010introduction}, page 45.

For the third statement we work as follows. If $\mu_n \to \mu$ weakly in a measurable set $A \subset \Omega$ such that $\mathbb{P}(A)=1,$ then on $A$ we also have that $S_{\mu_n}(y) \to S_{\mu}(y)$ for all $y \in Z,$ by our main theorem. This proves the $"\Rightarrow"$ direction. For the $"\Leftarrow"$ direction, we fix a sequence $(z_k)_{k=1}^\infty$ in $Z$ that converges to some $z \in Y \setminus \mathrm{supp}(\mu).$ For each $k \in \mathbb{N}$ we find a measurable set $A_k$ with $\mathbb{P}(A_k)=1$ and $S_{\mu_n}(z_k) \to S_{\mu}(z_k).$ Then $A:= \bigcap \limits_{k \in \mathbb{N}} A_k$ is measurable with $\mathbb{P}(A)=1$ and for all $y \in \{z_k : k \in \mathbb{N} \}$ we have $S_{\mu_n}(y) \to S_{\mu}(y).$ Since this set has an accumulation point, we find by our main theorem that $\mu_n \to \mu$ weakly on $A$, as we want.
\end{proof}

\end{section}

\begin{section}{The imaginary part of the Stieltjes transform} \label{s:ImS}
In this section we analyze further the retrieval of measure identity from Corollary \ref{re-cor} for measures on the real line. According to this identity, the function $E \mapsto \frac{1}{\pi} \mathrm{Im} S_\mu (E+i\eta)$ should represent a density that approximates $\mu$ well as $\eta \to 0^+.$

To analyze this property we introduce the concept of the convolution of probability measures as well as the kernel density estimators.

\begin{definition}[Convolution] Let $\mu$ and $\nu$ be two probability measures on $(\mathbb{R},\mathcal{B}_{\mathbb{R}})$ and $f,g$ be two probability density functions.
\begin{itemize}
    \item (convolution of measures). The convolution of the two measures is defined as $\mu \ast \nu := (\mu \otimes \nu)^+,$ where $\mu \otimes \nu$ is the product measure and the addition in the exponent represents the push-forward of the product measure over the addition map.
    
    \item (convolution of density with measure). The convolution of the density $f$ and the probability measure $\nu$ is defined as the function $f \ast \nu: \mathbb{R} \to \mathbb{R}$ with:
    \[
    (f \ast \nu)(x) := \int_{\mathbb{R}} f(x-y) d\nu(y)
    \]
    
    \item (convolution of densities). The convolution of the densities $f$ and $g$ is the function $f \ast g: \mathbb{R} \to \mathbb{R}$ with:
    \[
    (f \ast g)(x) := \int_{\mathbb{R}} f(x-y) g(y) dy
    \]
\end{itemize}
\end{definition}

Some remarks about the definitions are the following. For the first definition, for an arbitrary set $B \in \mathcal{B}_{\mathbb{R}},$ we have that:
\[
(\mu \ast \nu)(B) = (\mu \otimes \nu)(\{(x,y) \in \mathbb{R}^2: x+y \in B \}).
\]
If $f: \mathbb{R} \to \mathbb{R}$ is $\mu \otimes \nu$-integrable, we then obtain that:
\[
\int_{\mathbb{R}} f d(\mu \ast \nu) = \int_{\mathbb{R}^2} (f \circ +) d (\mu \otimes \nu) = \int_{\mathbb{R}^2} f(x+y) d(\mu \otimes \nu) (x,y) = \int_{\mathbb{R}} \int_{\mathbb{R}} f(x+y) d\mu(x) d\nu(y).
\]
Particularly, for an indicator function $f= \mathbf{1}_B$ for some $B \in \mathcal{B}_{\mathbb{R}}$ we get that:
\[
(\mu \ast \nu)(B) = \int_{\mathbb{R}} \mathbf{1}_{B} d(\mu \ast \nu) = \int_{\mathbb{R}} \int_{\mathbb{R}} \mathbf{1}_B (x+y) d\mu(x) d\nu(y) = \int_{\mathbb{R}} \mu(B - y) d\nu(y).
\]
The equality of the first with the third term shows that the convolution function is commutative in the space of probability measures. We denote by $f dx$ the probability measure with density $f,$ where $f$ is a probability density function. We then have the following properties of the convolution:
\begin{itemize}
    \item The convolution is a commutative binary operation in the space of probability measures where the neutral element is given by the Dirac probability measure $\delta_0.$
    \item The function $f \ast \nu$ is a probability density for the convolution measure $(f dx \ast \nu),$ that is $fdx \ast \nu = (f \ast \nu ) dx .$
    \item The function $f \ast g$ is a probability density for the convolution measure $(f dx \ast g dx),$ that is $f dx \ast g dx = (f \ast g) dx,$ where $g$ is another probability density function.
    \item The convolution of probability measures on $(\mathbb{R},\mathcal{B}_{\mathbb{R}})$ is continuous with respect to weak convergence of measures. That is, if $\mu$ , $\nu$ , $(\mu_n)_{n=1}^\infty$ , $(\nu_n)_{n=1}^\infty $ are probability measures on $(\mathbb{R},\mathcal{B}_{\mathbb{R}})$ with $\mu_n \to \mu$ weakly and $\nu_n \to \nu$ weakly, then $\mu_n \ast \nu_n \to \mu \ast \nu $ weakly.
\end{itemize}
We introduce here the kernel density estimators and particularly focus on the Cauchy kernel.
\begin{definition}[Cauchy kernel]
For all $\eta>0,$ we define the Cauchy kernel as the function $P_\eta : \mathbb{R} \to \mathbb{R}$ with
\[
P_\eta (x) = \frac{1}{\pi} \frac{\eta}{x^2+\eta^2},
\]
which is the density function of the Cauchy distribution with scale parameter $\eta.$
\end{definition}
For the Cauchy kernel, the following important identity holds:
\[
\lim \limits_{\eta \to 0^+} (P_\eta dx) = \delta_0,
\]
where $\delta_0$ denotes the Dirac measure centered at 0. That is, for every $B \in \mathcal{B}_{\mathbb{R}},$
\[
\delta_0(B) = \begin{cases}
    1 \ , \ \text{if} \ 0 \in B \\
    0 \ , \ \text{if} \ 0 \notin B.
\end{cases}
\]
This is because the characteristic function of the Cauchy distribution with parameter $\eta$ is given by $t \mapsto e^{-\eta |t|}$ and so letting $\eta \to 0^+$ yields the identity.

For the imaginary part of the Stieltjes transform we have the following fundamental property for any real bounded measure $\mu$:
\[
\frac{1}{\pi} \mathrm{Im} S_\mu(E+i\eta) = \int_{\mathbb{R}} \frac{1}{\pi} \frac{\eta}{(E-x)^2 + \eta^2} d\mu(x) = (P_\eta \ast \mu)(E)
\]
This property means that the function $\frac{1}{\pi}\mathrm{Im}S_\mu(\cdot + i \eta)$ is the convolution of the kernel density $P_\eta$ with $\mu$ and thus a density for the measure $(P_\eta dx ) \ast \mu.$ As $\eta \to 0^+,$ we have that:
\[
\frac{1}{\pi} \mathrm{Im} S_\mu(\cdot+i\eta)dx = (P_\eta dx) \ast \mu \to \delta_0 \ast \mu = \mu .
\]

Assume now that $\{ \sigma_n \}_{n=1}^\infty$ is a sequence of empirical spectral distributions of $\textit{Hermitian}$ random matrices so that $\sigma_n$ converges to a \textit{real} deterministic measure $\sigma$. We now take the convolution of $\sigma_n$ with the Cauchy kernel $P_\eta.$ Since $\sigma_n \to \sigma,$ we get by the previous analysis that $(P_{\eta_n} \ast \sigma_n)dx \to \sigma ,$ for any sequence $\eta_n \to 0.$ However, if $\sigma$ has some density $\sigma = f_\sigma dx,$ this is not enough to deduce that $P_{\eta_n} \ast \sigma_n \to f_\sigma$ for example in the supremum norm. This would allow \textit{local estimation} by $\sigma_n$ through the Cauchy kernel about the density $f_\sigma .$ If $\eta = \eta_n$ drops too quickly to zero as $n \to \infty$ then $(P_{\eta_n} \ast \sigma_n)$ will have steep peaks at each eigenvalue and thus will not approximate the density of $\sigma.$ This phenomenon is typical for kernel density estimators. We introduce them next in their generality.

\begin{definition}[Kernel] A kernel $K$ is a real probability density function. If $K$ is a kernel and $h>0$ a real parameter, we define $K_h$ as the kernel with $K_h(x) = \frac{1}{h} K(\frac{x}{h})$ for all $x \in \mathbb{R}$ and call $K_h $ "the kernel $K$ at bandwidth $h."$
\end{definition}

It is clear that $K_h$ is a kernel if $h>0$ and $K$ is a kernel. An example of a kernel is the previously defined Cauchy kernel $P_{\eta}$. We have for all $x \in \mathbb{R}$ and $\eta>0:$
\[
P(x) = \frac{1}{\pi} \frac{1}{x^2+1} \ \ \text{and} \ \  P_\eta(x) = \frac{1}{\pi \eta} \frac{1}{ \left (\frac{x}{\eta} \right )^2+1} = \frac{1}{\pi} \frac{\eta}{x^2 + \eta^2}.
\]
We are interested in constructing a density function $f$ that describes the experiment of drawing at random from the real-valued observations $u = (u_1,...,u_N)$, in other words that approximates the empirical probability measure:
\[ 
f dx \approx \rho_N = \frac{1}{N} \sum \limits_{i=1}^N \delta_{u_i}
\]
This can be done with the help of a kernel $K_h,$ which is usually chosen to be unimodal and symmetric around zero, see for example \cite{parzen1962estimation}.

\begin{definition}[Kernel density estimator] The kernel density estimator with kernel $K$ and bandwidth $h>0$ for an empirical measure $\rho_N$ is the density function given by the convolution $K_h \ast \rho_N$.
\end{definition}

We therefore have the density function estimator $K_h \ast \rho_N: \mathbb{R} \to \mathbb{R}$ with
\[
x \mapsto (K_h \ast \rho_N)(x) = \frac{1}{N} \sum \limits_{i=1}^N K_h(x-u_i) = \frac{1}{N h} \sum \limits_{i=1}^N K \left ( \frac{x - u_i}{h} \right ).
\]
The center of the kernel is placed upon each observation, whose influence is smoothed out over its neighbourhood. The size of this neighborhood is governed by the bandwidth. A small bandwidth will retain the probability mass of $1/N$ to be closer to its observation, whereas a larger bandwidth will result in a wider spread of probability mass. Therefore, a smaller bandwidth will result in a peaky density function (with steep peaks at each observation), whereas a larger one will result in a smoother density function.

Assume now we are given an empirical spectral distribution $\rho_N$ which comes from a Hermitian $N \times N$ matrix $X_N.$ The kernel density estimator at the location $E \in \mathbb{R}$ for $\rho_N$ with Cauchy kernel $P$ at bandwidth $\eta > 0$ is then given by:
\begin{align*}
(P_\eta \ast \rho_N)(E) &= \frac{1}{N\eta} \sum \limits_{i=1}^N \frac{1}{\pi} \frac{1}{\left ( \frac{E - \lambda_i^{X_N}}{\eta} \right )^2 + 1} \\
&= \frac{1}{\pi N} \sum \limits_{i=1}^N \frac{\eta}{(E - \lambda_i^{X_N})^2 + \eta^2} \\
&= \frac{1}{\pi} \mathrm{Im} S_{\rho_N} (E+i\eta) .
\end{align*}
This gives the imaginary part of the Stieltjes transform the role of the kernel density estimator, see for example \cite{benaych2016lectures} pages 7-8. Choosing different values of $\eta$ we get different estimations of the empirical spectral distribution. For a very small parameter of $\eta$, we may not obtain a useful approximation of the density, whereas for a bigger one the estimation would be better. In the next chapter, we will find the optimal value of $\eta$ so that the estimation is still useful for the approximation of the empirical eigenvalue density of the Marchenko-Pastur law for covariance matrices.
\end{section}

\pagebreak

\begin{section}{The Wirtinger derivative of the Stieltjes transform} \label{wirtinger}
In this section we analyze further the retrieval of measure identity from Corollary \ref{re-comp} for measures on the complex plane. According to this identity, we have a density given by the function $z \mapsto \frac{1}{\pi} \partial_{\overline{z}} S_\mu(z)$ which matches the density of $\mu$ on the complex plane.

We will analyze the generalized Cauchy theorem as well as convolutions with the Cauchy kernel, see for example \cite{lebl2019tasty}, pages 109-110.
\begin{thm}[Generalized Cauchy theorem]
Let $D$ be a disk on $\mathbb{C}$ and $f$ a complex-valued $C^1$ function on the closure of $D.$ Then:
\[
\forall \zeta \in D: f(\zeta) = \frac{1}{2\pi i} \int_{\partial D} \frac{f(z)}{z - \zeta}dz - \frac{1}{\pi} \iint_D \frac{\partial }{\partial \overline{z}} \frac{f(z)}{z - \zeta} d^2 z .
\]
\end{thm}
Notice that if $f$ is complex-differentiable (holomorphic) then the identity reduces to the well-known Cauchy formula. By saying that $f$ is a $C^1$ function we just mean that the derivatives $\frac{\partial f}{\partial x }$ and $\frac{\partial f}{\partial y }$ exist and they are continuous functions.

Notice also that if we pick a function $f \in C_c(\mathbb{C})$ with compact support, then the identity reduces to:
\[
\forall \zeta \in \mathrm{supp}(f): f(\zeta) = - \frac{1}{\pi} \iint_{\mathbb{C}} \frac{\partial }{\partial \overline{z}} \frac{f(z)}{z - \zeta} d^2 z = \frac{1}{2\pi i} \iint_{\mathbb{C}} \frac{\partial }{\partial \overline{z}} \frac{f(z)}{z - \zeta} dz \wedge d\overline{z},
\]
where we defined the \textit{wedge product} as $dz \wedge d\overline{z} := -2i \ d^2 z .$ This is a \textit{complex measure}, the space of which is a linear space over the complex numbers and also a Banach space, see \cite{rudin86real}, pages 116-119.

We now turn our attention to convolutions of complex functions and measures on the complex plane.
\begin{definition}[Convolution of complex measures and densities] Let $\mu$ and $\nu$ be two probability measures on $(\mathbb{C},\mathcal{B}_{\mathbb{C}})$ and $f,g$ two complex probability density functions. We can define again the following convolutions:
\begin{itemize}
    \item (\textit{convolution of measures}). The convolution of the two complex measures is defined as in the real case in which we recall that we defined them in a general setting.
    \item (\textit{convolution of density with measure}). The convolution of the complex density $f$ and probability measure $\nu$ is defined as the complex function $f \ast \nu: \mathbb{C} \to \mathbb{C}$ with:
    \[
    \forall z \in \mathbb{C}: (f \ast \nu)(z) := \int_{\mathbb{C}} f(z-w)d\nu(w).
    \]
    \item (\textit{convolution of densities).} The convolution of the complex densities $f$ and $g$ is the function $f \ast g: \mathbb{C} \to \mathbb{C}$ with:
    \[
    \forall z \in \mathbb{C}: ( f \ast g )(z) := \int_{\mathbb{C}} f(z-w) g(w)dw.
    \]
\end{itemize}
\end{definition}
The same remarks that hold for the real measures and densities also hold for the complex ones.

We introduce now the Cauchy kernel in complex analysis, see for example \cite{lebl2019tasty}, page 10.
\begin{definition}[Cauchy kernel] We define the complex Cauchy kernel as the function $k: \mathbb{C} \setminus \{0\} \to \mathbb{C}$ with:
\[
k(z) := \frac{1}{2\pi i} \ \frac{1}{z}.
\]
\end{definition}
For the Stieltjes transform of a complex measure $\mu$ we now have the following expression for it as a convolution with the complex Cauchy kernel:
\[
S_\mu (z) = \int_{\mathbb{C}} \frac{1}{w-z} d\mu(w) = - 2\pi i \int_{\mathbb{C}} k(z-w) d\mu(w) = - 2\pi i \cdot (k \ast \mu)(z).
\]
Suppose now that the measure $\mu$ has a complex density $f$ such that $d\mu = f(z) \ d^2 z .$ Then through the Stieltjes transform of $\mu$ we can retrieve the density $f$ by the Cauchy-Riemman equation:
\[
\frac{1}{\pi} \frac{\partial}{\partial \overline{z}} ( S_\mu) = f .
\]
In a more general setting, we have that the convolution of a measure $\mu$ with complex density $\phi (z, \overline{z})$ such that $d\mu = \phi(z,\overline{z}) \ dz \wedge d \overline{z}$ with the Cauchy kernel solves the equation:
\[
\frac{\partial}{\partial \overline{z}} ( k \ast \mu) = \phi .
\]
This means that we have the following retrieval of measure identity which we prove through the generalized Cauchy theorem:
\[
d\mu = \left ( \frac{1}{\pi} \frac{\partial}{\partial \overline{z}} S_\mu \right ) d^2 z .
\]
\begin{proof}
We pick a complex test function $f$ with compact support. By the analysis of the generalized Cauchy formula, we recall that:
\[
\forall \zeta \in \mathrm{supp}(f): f(\zeta) = \frac{1}{2\pi i} \iint_{\mathbb{C}} \frac{\partial }{\partial \overline{z}} \frac{f(z)}{z - \zeta} dz \wedge d\overline{z} .
\]
Suppose now we have the C-R equation:
\[
 \frac{\partial}{\partial \overline{z}} ( v ) = g ,
\]
where $g$ is a complex continuous function with compact support and $v$ is an unknown complex function. Suppose that there exists a complex measure $\mu$ such that $d\mu = g(z,\overline{z}) \ dz \wedge d \overline{z}.$ The existence of a solution $v$ for the C-R problem is established for example in \cite{voisin_2002}, pages 35-36. We show that if such a solution exists then we must have $v = k \ast \mu$. Indeed, by the generalized Cauchy formula, we get that:
\begin{align*}
(k \ast \mu ) (z, \overline{z}) &= \frac{1}{2\pi i}\int_{\mathbb{C}} \frac{1}{w-z}d\mu(w) =  \frac{1}{2\pi i }\int_{\mathbb{C}} \frac{g (w,\overline{w})}{w-z}dw \wedge d \overline{w} \\ \\
&= \frac{1}{2 \pi i}  \int_{\mathbb{C}} \frac{\partial }{\partial \overline{w}} \frac{v(w,\overline{w})}{w-z}dw \wedge d \overline{w} = v(z,\overline{z}) . \qedhere
\end{align*}
\end{proof}
In particular, for the Stieltjes transform, we get that:
\[
\frac{1}{\pi} \frac{\partial}{\partial \overline{z}} S_\mu = -2i \frac{\partial}{\partial \overline{z}} ( k \ast \mu ) = -2i \frac{d\mu}{dz \wedge d \overline{z}} =  \frac{d\mu}{d^2 z} .
\]
Assume now that $\{ \sigma_n \}_{n=1}^\infty$ is a sequence of empirical spectral distributions of \textit{non-Hermitian} random matrices, so that $\sigma_n$ converges to a \textit{complex} deterministic measure $\sigma.$ If we take the convolution of $\sigma_n$ with the complex Cauchy kernel $k$, then to show the convergence $\sigma_n \to \sigma$, it is enough to establish that $\sigma_n \ast k \to \sigma \ast k.$

\end{section}
\chapter{Local Marchenko-Pastur law}
We let $X$ be an $ M\times N$ matrix with complex components $x_{ij} = \mathrm{Re} (x_{ij}) + i \mathrm{Im} (x_{ij})$, for $i=1,...,M$ and $j=1,...,N.$ Assume that $\mathrm{Re} (x_{ij})$ and $\mathrm{Im} (x_{ij})$ are independent and identically distributed real random variables with mean zero and variance $\frac12$ so that:
\beq \label{e:XN1}
\E (x_{ij}) =0 \quad \textrm{and} \quad \E|x_{ij}|^2 =1, \ \  \textrm{for} \ \ i=1,...,M, \, j = 1, \dots, N.
\eeq
We denote by $X_N$ the scaled matrix:
\beq
\label{e:XN}
X_N :=X/\sqrt{N}.
\eeq
In this chapter, we will analyze asymptotics of the empirical spectral measure (see section \ref{ESD}) of the matrix $X_N^* X_N$ for $N\to \infty$, when $M=N.$ We define:
\beq \label{e:Kxx}
K_{XX} := X_N^* X_N \in \mathbb{C}^{N \times N}
\eeq
For general $M$ and $N$ we can consider the collection of the complex random variables $\{x_{ij}\}$ as $M$ \textit{observations} of $\mathbb{C}^N$-valued normalized random variables, where we can consider $N$ as the number of \textit{features} of each observation. Hence we have the name \textit{Sample Coviariance Matrix} for the random matrix $X_N^*X_N,$ as its matrix-expectation gives the covariances of the sample of $M$ observations of the $\mathbb{C}^N$-valued random variables.

The Sample Covariance Matrix $K_{XX}$ is always Hermitian and {positive semi-definite}. This is because:
\[
K_{XX}^* = (X_N^* X_N)^* = X_N^* X_N = K_{XX},
\]
where we used the property $(AB)^* = B^* A^*,$ valid for any complex matrices $A,B$ such that their multiplication makes sense. We also have that:
\[
\mathbf{x}^* K_{XX} \mathbf{x} = \mathbf{x}^* X_N^* X_N \mathbf{x} = ( X_N \mathbf{x} )^* X_N \mathbf{x} = \| X_N \mathbf{x} \|^2 \geq 0,
\]
where we used the usual Euclidean norm for the vector $X_N \mathbf{x}$ and $\mathbf{x}$ is any vector in $\mathbb{C}^N.$

We remind that a matrix $A \in \mathbb{C}^{N \times N}$ is called \textit{positive semi-definite} when
\[
\forall \mathbf{x} \in \mathbb{C}^N : \mathbf{x}^* A \mathbf{x} \geq 0.
\]
For a Hermitian matrix $H \in \mathbb{C}^{N \times N}$ the following relation holds:
\[
H \textrm{ is positive semi-definite} \Leftrightarrow \textrm{All eigenvalues of } H \textrm{ are non-negative}.
\]
We denote by $\lambda_i,$ $i=1,...N$ the eigenvalues of the matrix $K_{XX}.$ Since $K_{XX}$ is Hermitian and positive semi-definite, we can assume that:
\[
0\leq \lambda_1 \leq \lambda_2 \leq ... \leq \lambda_N .
\]
We denote by $\rho_N$ the empirical spectral cumulative density of these eigenvalues, according to the definition in section \ref{ESD}, applied to the Borel sets $B_x := (-\infty,x] \subset \mathbb{R}:$
\beq \label{rhoN}
\rho_N (x) := \frac{1}{N} \# \{ i \leq N : \lambda_i \leq x \}
\eeq
It was first established in 1967 in \cite{marchenko1967distribution} that the limiting measure of $\rho_N dx$ is the following measure $\rho^d dx$ which is defined according to the parameter:
\beq \label{d-param}
d:= \lim \limits_{M,N \to \infty} \frac{M}{N} \in (0,1].
\eeq
Specifically, we have that $\rho_N dx \to \rho^d dx$ weakly almost surely, where:
\beq \label{M-P}
\rho^d(x) := \frac{1}{2\pi d } \frac{\sqrt{(\lambda_+-x)(x-\lambda_-)}}{ x}, \quad x \in [\lambda_-,\lambda_+]
\eeq
and $\lambda_- := (1-\sqrt{d})^2 , \quad \lambda_+ := (1+\sqrt{d})^2 .$

This has been called the Marchenko-Pastur distribution. In the case of a square matrix $X$ in which the number of observations is equal to the number of features, we have that $M=N \Rightarrow d=1$ and so the limiting density becomes:
\beq \label{M-P1}
\rho(x) = \left\{
\begin{aligned}
&\frac{1}{2\pi} \sqrt{\frac{4}{x}-1} , & \qquad 0<x\leq 4 \\
&0 , &\textrm{otherwise}
\end{aligned} \right.
\eeq
This is the case when the limiting measure has a square root singularity near $0$. We call this point the \textit{hard edge} of the density, while the point $4$ on the real line will be called the \textit{soft edge} of the limiting density. The space between $0$ and $4$ will be called the \textit{bulk} of the density, as there will be a "bulk amount" of eigenvalues there.

\begin{section}{Resolvent identities for Sample Covariance Matrices} \label{s:Resol}

The matrix transformation that is closely related to the Stieltjes transform of the empirical spectral measure is called the \textit{resolvent matrix}.

For a matrix $X$ we define its resolvent transformation as the matrix:
\[
G_z^X := (X - zI)^{-1},
\]
where $z \in \mathbb{C}$ is a complex parameter related to the complex parameter of the Stieltjes transformation.

We have the following identity relating the Stieltjes transform of the empirical spectral measure with the resolvent matrix:
\beq \label{TrG}
S_{N}(z) := S_{\mu_N}(z) = \int_{\mathbb{R}} \frac{1}{x-z}d\mu_N(x) = \frac{1}{N} \sum \limits_{i=1}^N \frac{1}{\lambda_i - z} = \frac{1}{N} \mathrm{Tr} (G_z^X),
\eeq
for any $z \in \mathbb{C} \setminus \mathbb{R},$ where $\mu_N$ is the empirical spectral measure of the matrix $X\in \mathbb{C}^{N \times N}.$

We recall from section \ref{s:ImS} that the imaginary part of the Stieltjes transform of $\rho_N$ can provide local information about the convergence $\rho_N dx \to \rho dx.$ If we set $z=E+i\eta,$ then $\frac{1}{\pi} \mathrm{Im} S_{\rho_N} (E)$ will approximate the cumulative density $\rho_N(E)$ while the parameter $\eta$ depending on $N$ will model the scale of this convergence around each individual eigenvalue.

In this chapter we will prove this convergence in the optimal scale for $\eta$ while also improving bounds for the rate of convergence of the corresponding Stieltjes transforms of the empirical spectral measure and the limiting Marchenko-Pastur measure:
\[
S_N \to S_{\rho} .
\]
This will provide us with some new results about the rate of convergence of the cumulative density $\rho_N$ to the cumulative density provided by the measure $\rho.$ This means that we will quantify the convergence:
\beq \label{counting}
\rho_N(E) \to \int_{0}^E \rho(x) dx ,
\eeq
in a new way. We will also quantify the rigidity of each individual eigenvalue of the Sample Covariance Matrix in a new way, that is:
\beq
\lambda_i \approx \gamma_i ,
\eeq
where $\gamma_i$ is the $i-$th quantile of the Marchenko-Pastur distribution, i.e.
\beq \label{quantile}
\int_0^{\gamma_i} \rho(x) dx = \frac{i}{N}.
\eeq
We collect now for the sake of the proofs some useful resolvent identities for the Sample Covariance Matrix. The resolvent transformation of the matrix $K_{XX}$ will be denoted by $G$, omitting the dependence on $z$ and $K_{XX}.$

Firstly, we remark some properties of the Stieltjes transform of the Marchenko-Pastur measure, see also \cite{cacciapuoti2013local} section 2.2.

\begin{lem}
Let $S_\rho$ be the Stieltjes transfrom of the measure $\rho dx.$ We then have that:
\beq \label{S-rho1}
S_{\rho}(z) = -\frac{1}{2} + \frac{1}{2} \sqrt{1 - \frac{4}{z}}, \quad z \in \mathbb{C}\setminus [0,4]
\eeq
and the following algebraic identity:
\beq \label{S-rho}
S_{\rho}(z) = - \frac{1}{z(S_\rho(z)+1)} .
\eeq
\end{lem}
We note that the square root is chosen so that $\mathrm{Im} \left ( \sqrt{1- 4/z} \right ) \geq 0.$ It is a simple calculus exercise to compute the integral in the definition of $S_\rho$ to derive \eqref{S-rho1} and then derive the algebraic equation \eqref{S-rho}.

The algebraic identity \eqref{S-rho} will be the basis of our proof. We will show that $S_{N}$ satisfies a similar algebraic identity with high probability so that it has to be "close" to $S_{\rho}.$

\underline{Resolvent identities}.

From now on $X$ will denote the scaled version of our matrix. We define the two resolvent matrices $G$ and $\mathcal{G}$ related to $K_{XX}$ as:
\beq \label{res0}
G := (X X^* - zI_N)^{-1} \quad \textrm{and} \quad \mathcal{G} := (X^*X - zI_N)^{-1}.
\eeq
For $j_1,j_2 \in \{0,1,...,N-1 \}$ we denote by $X^{(j_1)}$ the submatrix of the $N \times N$ matrix $X$ with the first $j_1$ columns removed. We denote by $X_{(j_2)}$ the submatrix of $X$ with the first $j_2$ rows removed.

We now define the resolvents of the covariance of the "stripped" matrix $X_{(j_2)}^{(j_1)}$ as follows:
\beq \label{res1}
    G^{(j_1)}_{(j_2)}:=\left((X^{(j_1)}_{(j_2)})^* X^{(j_1)}_{(j_2)}- z \right)^{-1} \quad \text{and} \quad \mathcal{G}^{(j_1)}_{(j_2)}:=\left(X^{(j_1)}_{(j_2)}(X^{(j_1)}_{(j_2)})^*  - z \right)^{-1}
\eeq
We denote by $G_{(j_2),kl}^{(j_1)}$ the $(k,l)$ element of the resolvent matrix $G_{(j_2)}^{(j_1)}.$ We remark that $(X^{(m)})^*X^{(m)}$ is the minor matrix of $K_{XX}$ with the first $m$ rows and the first $m$ columns removed.
\begin{lem}
The following trace equality holds:
\begin{equation} \label{e:traces}
    \Tr \left [ \mathcal{G}^{(j_1)}_{(j_2)} \right ] =  \frac{j_1-j_2}{z} + \Tr \left [ G^{(j_1)}_{(j_2)} \right ].
\end{equation}
\end{lem}
\begin{proof}
    We have that $(X_{(j_2)}^{(j_1)})^*X_{(j_2)}^{(j_1)}$ and $X_{(j_2)}^{(j_1)}(X_{(j_2)}^{(j_1)})^*$ have the same non-zero eigenvalues and a total of $N-j_1$ eigenvalues for the first and a total of $N-j_2$ eigenvalues for the second. If $j_1 < j_2$ then the first matrix has $j_2 - j_1$ extra zero eigenvalues, while if $j_1 > j_2$ the second one has $j_1 - j_2$ extra zero eigenvalues. Hence, we have the identity.
\end{proof}
We now define the following quantities
\begin{equation} \label{d:S-Lambda}
    S_N (z)^{(j_1)}_{(j_2)} := \frac{1}{N} \Tr \left [G^{(j_1)}_{(j_2)} \right ] \quad \text{and} \quad \Lambda (z)^{(j_1)}_{(j_2)} := S_N (z)^{(j_1)}_{(j_2)} - S_{\rho}(z)
\end{equation}
and we use the notation $S_N$ and $\Lambda$ when $j_1 = j_2 = 0$. We will generally omit denoting the dependence on the $z$-variable.

We now state some identities for resolvent entries. 

We denote by $\left [ G^{(j_1)}_{(j_2)} \right ]^k$ the resolvent which occurs after an extra removal of the $k-$th column of the scaled matrix $X,$ where $k\in \{j_1+1,...,N\}.$
\begin{lem}[Stripping lemma] \label{l:resol1}
 With $G^{(j_1)}_{(j_2)}$ and $k$ as before, we have that for $i,j \neq k :$
\begin{equation} \label{e:resol1}
 G^{(j_1)}_{(j_2), ij} = \left [ G^{(j_1)}_{(j_2)} \right ]^k_{ij} + \frac{G^{(j_1)}_{(j_2), ik}G^{(j_1)}_{(j_2), kj}}{G^{(j_1)}_{(j_2), kk}}.
 \end{equation}
\end{lem}
\begin{proof}
    For the proof we will use Woodbury's matrix identity:
    \[
       (A + UBV)^{-1} = A^{-1} - A^{-1} U ( B^{-1} + VA^{-1}U)^{-1}VA^{-1}
    \]
    This identity can be found for example in \cite{higham2002accuracy} page 258.

    We can assume that $j_1=j_2=0,$ otherwise the proof will be similar. We denote by $M$ the matrix $X^*X - zI$. Let $\mathbf{e_k}$ be the standard $k-$th column unit basis vector and define the matrices:
    \[
M_{ij}^{[k]}:=M_{ij} \mathbf{1}(i \neq k) \mathbf{1}(j \neq k), \quad U:= \begin{pmatrix}
M \mathbf{e}_k & \mathbf{e}_k
\end{pmatrix} , \quad V:= \begin{pmatrix}
\mathbf{e}_k^* \\
\mathbf{e}_k^* M
\end{pmatrix}
\]
so that we have
\[
M=M^{[k]}+U V-M_{k k} \mathbf{e}_k \mathbf{e}_k^* .
\]
Then by the Woodbury matrix identity we have
\[
\left(M^{[k]}-M_{k k} \mathbf{e}_k \mathbf{e}_k^* \right)^{-1}=G+G U(I-V G U)^{-1} V G .
\]
A straightforward calculation yields
\[
I-V G U=-\left(\begin{array}{cc}
0 & G_{k k} \\
M_{k k} & 0
\end{array}\right), \quad(I-V G U)^{-1}=-\frac{1}{M_{k k} G_{k k}}\left(\begin{array}{cc}
0 & G_{k k} \\
M_{k k} & 0
\end{array}\right),
\]
from which we deduce, after a short calculation,
\[
\left(M^{[k]}-M_{k k} \mathbf{e}_k \mathbf{e}_k\right)^{-1}=G-\frac{1}{M_{k k}} \mathbf{e}_k \mathbf{e}_k^*-\frac{1}{G_{k k}} G \mathbf{e}_k \mathbf{e}_k^* G .
\]
Using the formula for the block inversion of a matrix, we conclude
\[
[G]^k_{ij} = G_{ij} -\frac{1}{G_{k k}} G \mathbf{e}_k \mathbf{e}_k^* G,
\]
from which \eqref{e:resol1} follows.
\end{proof}
Next, we have the following relationship between the $(i,i)$ element of $|G|^2 := GG^*$ and $(\mathrm{Im} G)_{ii}$, and the same holds for $G^{(j_1)}_{(j_2)}, \mathcal{G}^{(j_1)}_{(j_2)}$ instead of $G$:
\begin{lem}[Ward identity] \label{ward} If $z=E+i\eta,$ then
\beq
|G(z)|_{ii}^2 = \frac{(\mathrm{Im} G(z))_{ii}}{\eta}.
\eeq
\end{lem}
\begin{proof}
    We apply spectral decomposition to $G$ as follows:
    \[
    G = \sum \limits_{j=1}^N \frac{u_j u_j^*}{\lambda_j - z},
    \]
    where $u_j$ is the $j-$th eigenvector of $K_{XX}$ and $\lambda_j$ is the $j-$th eigenvalue of $K_{XX}.$ We then have that:
    \begin{equation*}
    (GG^*)_{ii} = \sum \limits_{j=1}^N \frac{|u_j(i)|^2}{|\lambda_j-z|^2} = \sum \limits_{j=1}^N \frac{1}{\eta} \mathrm{Im} \frac{|u_j(i)|^2}{\lambda_j - z} = \frac{(\mathrm{Im} G_{ii})}{\eta} . \qedhere
    \end{equation*}
\end{proof}
The Lemma above yields the following inequality:
\begin{equation}\label{e:GsImG2}
|(G^2)_{ii}|\leq \frac{\mathrm{Im} G_{ii}}{\eta}.
\end{equation}
This is because
\[
|(G^2)_{ii}| = |\langle e_i , G^2 e_i \rangle| \leq \|G^* e_i \| \ \| G e_i \| = |G|^2_{ii}.
\]

Furthermore, we obtain the following result for the resolvent of the Sample Covariance ensemble, and the proof works also for $G_{(j_2)}^{(j_1)}$ and $ \mathcal{G}_{(j_2)}^{(j_1)}$ instead of $G$, see \cite{cacciapuoti2015bounds}, equation (3.10):
\begin{lem}\label{G11eta/s}
With $G$ and $\mathcal{G}$ as before and $s \in \mathbb{R}$, we have that
\begin{equation}
G_{ii}(E+i\eta/s) \leq sG_{ii}(E+i\eta).
\end{equation}
\end{lem}
\begin{proof}
  The proof can be found in \cite{cacciapuoti2015bounds}, page 18 and is based on the Ward property of the resolvent of a Hermitian matrix.
\end{proof}
The following identity holds exclusively for Sample Covariance matrices and it will be the starting point of our analysis. It can be found in \cite{cacciapuoti2013local}, equation 2.1. 

We denote by $[X_{(j_2)}^{(j_1)}]^k$ the matrix that occurs after the extra removal of the $k-$th column, where $k \in \{j_1+1,...,N\}.$ We denote by $[X_{(j_2)}^{(j_1)}]_l$ the matrix that occurs after the extra removal of the $l-$th row, where $l \in \{j_2+1,...,N\}.$
\begin{lem}[Sample Covariance stripping lemma] We let $\mathbf{x}_{(j)}^k$ denote the $k-$th column of the scaled matrix $X_{(j)}$. We then have the following identity regarding the resolvent diagonal elements, for $j_1,j_2 \in \{0,1,...,N-1\}$ and $k$ as before:
    \begin{align}
G_{(j_2), kk}^{(j_1)} &= \frac{1}{\| \x_{(j_2)}^k \|^2 - z - (\x_{(j_2)}^k)^*[X_{(j_2)}^{(j_1)}]^k\left(([X_{(j_2)}^{(j_1)}]^k)^*[X_{(j_2)}^{(j_1)}]^k - z \right)^{-1}([X_{(j_2)}^{(j_1)}]^k)^*\x_{(j_2)}^k} \label{G1}  \\
&= -\frac{1}{z \left(1+  (\x_{(j_2)}^k)^* \mathcal{G}_{(j_2)}^{(k)}
\x_{(j_2)}^k \right)} \label{G2}
    \end{align}
\end{lem}
\begin{proof}
    The proof is based on the block inversion formula (Schur's inversion). To ease the notation, we can assume that $j_1=j_2=0,$ otherwise the proof is similar. We can also asume that $k=1$. We strip the $X$ matrix by the column $\x^1,$ giving that:
    \[
    G = ( X^* X - z)^{-1} = M^{-1},
    \]
    where $M$ is the matrix:
    \begin{align*}
        M &= X^*X - zI = \begin{pmatrix}
        \x^1 & X^{(1)}
        \end{pmatrix}^*
        \begin{pmatrix}
         \x^1 & X^{(1)}
        \end{pmatrix} - zI \\
        &= \begin{pmatrix}
            (\x^1)^* \\
            (X^{(1)})^*
        \end{pmatrix}
        \begin{pmatrix}
         \x^1 & X^{(1)}
        \end{pmatrix} - zI \\
        &= \begin{pmatrix}
            \|\x^1\|^2 & (\x^1)^* X^{(1)} \\
            (X^{(1)})^* \x^1 & (X^{(1)})^* X^{(1)}
        \end{pmatrix} - zI \\
        &= \begin{pmatrix}
            \|\x^1\|^2 - z & (\x^1)^* X^{(1)} \\
            (X^{(1)})^* \x^1 & (X^{(1)})^* X^{(1)} - zI
        \end{pmatrix}.
    \end{align*}
    By the matrix inversion formula, we get that:
    \[
    G_{11} = ( M_{11} - M_{12} M_{22}^{-1} M_{21} )^{-1},
    \]
    for the corresponding sub-blocks of $M.$ This gives \eqref{G1}. 
    
    To derive \eqref{G2}, we use the matrix identity:
    \[
    X^{(1)} \left [ (X^{(1)})^* X^{(1)} - zI \right ]^{-1} (X^{(1)})^* =  X^{(1)} (X^{(1)})^*  \left [ X^{(1)} (X^{(1)})^*  - zI \right ]^{-1} ,
    \]
    which can be proved as follows:
    \begin{gather*}
        LHS = RHS \Leftrightarrow \\
        X^{(1)} \left [ (X^{(1)})^* X^{(1)} - zI \right ]^{-1} (X^{(1)})^* \left [ X^{(1)} (X^{(1)})^* - zI \right ] = X^{(1)} (X^{(1)})^* \Leftrightarrow \\
        X^{(1)} G^{(1)} (X^{(1)})^* X^{(1)} (X^{(1)})^* - zX^{(1)} G^{(1)} (X^{(1)})^* = X^{(1)} (X^{(1)})^* \Leftrightarrow \\
        X^{(1)} \left [ G^{(1)} (X^{(1)})^* X^{(1)} - z G^{(1)} \right ]  (X^{(1)})^* = X^{(1)} (X^{(1)})^* \Leftrightarrow \\
        X^{(1)} (X^{(1)})^* = X^{(1)} (X^{(1)})^* ,
    \end{gather*}
    after which \eqref{G2} is straight-forward.    
\end{proof}
Furthermore we have the following asymptotic behaviour for the Stieltjes transform of the Marchenko-Pastur distribution near the "soft" edge, $E \to 4$ and $\eta \to 0^+.$ Notice that near the "soft" edge we have that $S_{\rho} + \frac{1}{2} \approx 0$ and $\mathrm{Im} ( S_{\rho} ) \approx 0,$ according to \eqref{S-rho1}.
\begin{lem}
For $E>0$ we set $\kappa := |E-4|.$ Then for any fixed $E_0,E_1>0$ and $\eta_0>0$ there exist constants $c,C>0$ such that
\beq \label{asymp1} 
\left | S_{\rho} + \frac{1}{2} \right | \geq C ( \kappa^2   +\eta^2)^{\frac{1}{4}} \geq C\sqrt{\kappa+\eta},
\eeq
and
\beq \label{asymp2}
c \frac{\eta}{\sqrt{\kappa+\eta}} \leq \mathrm{Im} ( S_{\rho} ) \leq C \frac{\eta}{\sqrt{\kappa+\eta}} , 
\eeq
for all \ $E \in [E_0, E_1],$ and $\eta \in (0,\eta_0],$ with $\kappa \geq \eta .$
\end{lem}
\begin{proof}
    For \eqref{asymp1} we notice that
    \[
    S_{\rho} + \frac{1}{2} = \sqrt{1- \frac{4}{z}} = \sqrt{\frac{z-4}{z}} = \sqrt{\frac{E-4+i\eta}{E+i\eta}},
    \]
    and so
    \[
    \left | S_{\rho} + \frac{1}{2} \right | = \sqrt{\frac{|\kappa+i\eta|}{|E+i\eta|}} \geq C \sqrt{|\kappa+i\eta|} = C ( \kappa^2   +\eta^2)^{\frac{1}{4}}.
    \]
    The result now follows from the elementary inequality:
    \[
    \sqrt{\frac{a^2+b^2}{2}} \geq \frac{a+b}{2}.
    \]
    For \eqref{asymp2}, we notice that
    \[
    \mathrm{Im} ( S_{\rho} ) = \frac{1}{2} \mathrm{Im} \sqrt{1 - \frac{4}{z}} \overset{(z \to 4)}{\sim} \mathrm{Im} \sqrt{z-4}.
    \]
    If we set $w:= z-4,$ with $|w|=|\kappa+i\eta|,$ we then get that:
    \[
    \mathrm{Im} ( S_{\rho} ) \overset{(z \to 4)}{\sim} \mathrm{Im} ( \sqrt{w} ) = \sqrt{|w|} \sin \left (\frac{\theta}{2} \right ),
    \]
    where \[
    \theta := \tan^{-1} \left ( \frac{\eta}{\kappa} \right ).
    \]
    By the asymptotics $\tan^{-1}(x) \overset{x \to 0}{\sim} x$ and $\sin(x) \overset{x \to 0}{\sim} x,$ we get that:
    \begin{align*}
    \mathrm{Im} ( S_{\rho} ) &\overset{(z \to 4)}{\sim} (\kappa^2 + \eta^2)^{1/4} \frac{\eta}{\kappa} \\ 
    &\overset{(\kappa, \eta \to 0)}{\sim} \sqrt{\kappa + \eta} \frac{\eta}{\kappa} \\
    &= \frac{\eta}{\sqrt{\kappa+\eta}} \left ( 1 + \frac{\eta}{\kappa} \right ) \\
    &\overset{(\kappa, \eta \to 0)}{\sim} \frac{\eta}{\sqrt{\kappa+\eta}},
     \end{align*}
     where we used the fact that $\eta = \mathcal{O}(\kappa).$ This concludes the proof of \eqref{asymp2}.
\end{proof}
\end{section}

\begin{section}{Main theorems}

To state our theorems we define the domain for our parameters where we obtain our results:
\begin{equation}\label{e:etae}
Z_{E, \eta}:= \left \{ 4|\eta| \geq c(E^2 + \eta^2  - 4E ) \ \middle | \ E,\eta \in \mathbb{R} \right \}
\end{equation}
for some $c >0$. This domain is chosen so that $ | \mathrm{Im} (S_{\rho} + 1/2)^2 | \geq c |\mathrm{Re} (S_{\rho} + 1/2)^2 |$ which we need for the proof of Proposition \ref{b-R}. While all the proofs work for all $c>0$ not dependent on $N$, we will specifically work with $c=1$ to allow us the opportunity to illustrate it in the following picture, Figure \ref{f:domains}.
\begin{figure}[ht] 
\centering{\includegraphics[scale=0.5]{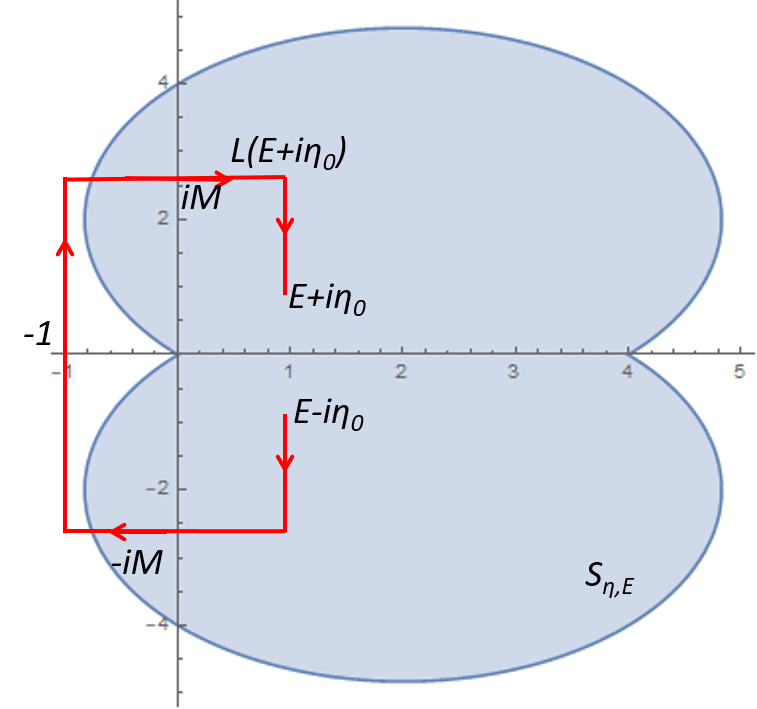}
\caption{The set $Z_{E,\eta}$ shaded in blue and the integration contour $L(z_0)$ from \eqref{e:pleijel} in red.} \label{f:domains}
} 
\end{figure}
 
We assume that
\begin{equation} \label{4moment}
\E|x_{11}|^4 =: \mu_4 < \infty,\end{equation} 
and that there exists a constant $D>0$ such that for all $N \in \mathbb{N}$ and $q \geq 1$:
\begin{equation}\label{truncation} \E |x_{11}|^{4q} \leq D^{4(q-1)} N^{q-1} \mu_4. 
\end{equation} 
These assumptions are inspired by the papers of G\"otze-Tikhomirov \cite{gotze2014rate, gotze2016optimal}, where they assume four-moment bounds \textit{just for independent random matrix entries} and not necessarily identically distributed. Notice that by the identical-distribution assumption, $x_{11}$ can be replaced by any $x_{ij}, \ i,j=1,...,N.$
%

We are now ready to state our first theorem.
\begin{thm}
\label{t:sti}
Let $X_N$ be a $N\times N $ matrix as described in equations \eqref{e:XN1} and \eqref{e:XN} and assume \eqref{4moment} and \eqref{truncation} for its matrix entries. Let $S_N$ and $S_{\rho}$ be the Stieltjes transforms as defined in equations \eqref{TrG} and \eqref{S-rho1}. Moreover set $z = E + i \eta$, with $\frac{N\eta}{|\sqrt{z}|} \geq M$  for some suitably large $M$. Then there exist positive constants $c_0$, $C$ such that for each $K>0$ and $ 1\leq q \leq c_0\left ( \frac{N\eta}{|\sqrt{z}|} \right )^{1/8} $ and $z \in Z_{E,\eta}$ \ or $E < 0$:
\beq
\mathbb{P} \left( \left | S_N(z)-S_{\rho}(z) \right | \geq  \frac{K}{N\eta} \right) \leq \frac{(Cq)^{cq^2}}{K^q}.
\eeq
Furthermore, for any $E \in \mathbb{R}$ and $\eta>0$ such that $\frac{N\eta}{|\sqrt{z}|} \geq M$ we have that:
\beq
 \mathbb{P} \left( \left | \mathrm{Im} S_N(z) - \mathrm{Im} S_{\rho} (z) \right) | \geq \frac{K}{N\eta} \right) \leq \frac{(Cq)^{cq^2}}{K^q}.
\eeq
\end{thm}
Some remarks about this theorem are the following. Firstly, if $\eta \sim 1,$ we have the so-called "global" law. Notice that in this case we also have that $|\sqrt{z}| \sim 1.$ Then, for a fixed and large enough $N$ we can pick a large enough $K(N) > 0$ such that $\frac{K}{N} \sim N^{-c},$ for some $c>0$ and also choose $q(N)$ such that $\frac{(Cq)^{cq^2}}{K^q} \sim N^{-d}$ for some $d>0.$ This gives the standard convergence $S_N \to S_{\rho}$ on "global" scales $\eta \sim 1$.

For "local" scales, $\eta \to 0^+ ,$ we have to consider two cases, whether we are near the "hard" edge, $E \to 0^+$, or somewhere else, $E \sim 1.$ Notice that in the second case, we once more have that $|\sqrt{z}| \sim 1$ and we must have $N\eta$ to be large enough, so that we can choose $K(N)$ such that $\frac{K}{N\eta} \sim N^{-c}$ for some $c>0.$ This means that we can optimally have $\eta \sim \frac{1}{N^{1-\epsilon}}, $ for any small $\epsilon > 0.$ Choosing also a suitable $q(N)$ gives that $\frac{(Cq)^{cq^2}}{K^q} \sim N^{-d}$, for some $d>0.$ Thus, we have achieved the "local" convergence $S_N \to S_{\rho}$ up to the scale $\eta \sim \frac{1}{N}$ away from the "hard" edge.

What is more important in this theorem, is the case when $E,\eta \to 0^+ ,$ which means that $|\sqrt{z}| \to 0^+$ and we get close to the "hard" edge of the Marchenko-Pastur distribution. Notice that in this case, the quantity $\frac{N\eta}{|\sqrt{z}|}$ needs to be large enough. Now, because of the singularity on the "hard" edge, the convergence $S_N \to S_{\rho}$ is still valid up to scales $\eta \sim \frac{1}{N}$ but if we want to go further down to scales $\eta \sim \frac{1}{N^2},$ we have to correct it as $|\sqrt{z}| \left | S_N - S_{\rho} \right | \to 0.$ Notice that:
\[
\left | S_N(z)-S_{\rho}(z) \right | \geq  \frac{K}{N\eta} \Leftrightarrow |\sqrt{z}| \left ( S_N - S_{\rho} \right ) \geq K \frac{|\sqrt{z|}}{N\eta}.
\]
We now want $\frac{N\eta}{|\sqrt{z}|} \geq M ,$ for some large enough $M.$ This means that:
\begin{align*}
     &\frac{N\eta}{(E^2+\eta^2)^{1/4}} \geq M \Leftrightarrow \\ 
     &\frac{\eta}{\sqrt{E+\eta}} \gtrsim \frac{1}{N} \Leftrightarrow \\
     &E+\eta \lesssim (N\eta)^2.
\end{align*}
If now $\eta \sim N^{-a},$ for some $a > 0,$ we must then have that:
\[
 -a \leq 2 - 2a \Leftrightarrow a \leq 2 ,
\]
 so that we can have this convergence up to the scale $\eta \sim \frac{1}{N^2}.$ Notice that we must also have that $E = \mathcal{O} \left (\frac{1}{N^2} \right )$ for these $\eta$ - scales. Thus, we have proved the convergence $|\sqrt{z}| \left | S_N - S_{\rho} \right | \to 0,$ optimally for $\eta \sim \frac{1}{N^{2-\epsilon}}$ for any small $\epsilon > 0$ since we can choose $K(N)>0$ large enough such that $\frac{K}{\frac{N\eta}{|\sqrt{z}|}} \sim N^{-c},$ for some $c>0$ and choose $q(N)$ such that $\frac{(Cq)^{cq^2}}{K^q} \sim N^{-d},$ for some $d>0.$ To summarize, we have proved the convergence $|\sqrt{z}| \left | S_N - S_{\rho} \right | \to 0,$ up to the scale $\eta \sim \frac{1}{N^2}$ near the "hard" edge for $E = \mathcal{O} \left (\frac{1}{N^2} \right ).$

We now state our second theorem which is a result of theorem \ref{t:sti}. We obtain fluctuation estimates on the counting cumulative function as stated next. Letting:
\begin{equation}
P (E) := \int_0^E \rho(x)dx,
\end{equation}
we compare it to $\rho_N(E)$, as defined in \eqref{rhoN}.

\pagebreak

\begin{thm}\label{t:counting}
With assumptions as in Theorem \ref{t:sti}, there exist constants $M_0, N_0, C, c >0$ such that for any $K > 0$ and $E \geq \frac{M_0}{N^2}$:
\begin{equation}\label{e:counting}
\mathbb{P}\left(|\rho_N(E) - P(E)| \geq K \min \left\{\sqrt {E}, \frac{ \log N}{N}\right\}\right) \leq \frac{(Cq)^{cq^2}}{K^q},
\end{equation}
for all $E\in \mathbb{R}$, $K>0, N>N_0, q\in \mathbb{N}$.
\end{thm}
We use the above theorem \ref{t:counting} to obtain rigidity estimates, that is how far each eigenvalue can fluctuate away from its "classical" location.

We define the "classical" locations of the eigenvalues, predicted by the Marchenko-Pastur distribution, as the quantiles $\gamma_i$, for $i=1,...,N$, such that:
\[
\int_0^{\gamma_i} \rho(x) dx = \frac{i}{N}.
\]
\begin{thm}\label{t:rigidity}
With assumptions as in Theorem \ref{t:sti}, there exist constants $C,c,N_0,\epsilon>0$ such that:
\begin{equation} \label{e:rigiditybulk}\mathbb{P} \left ( | \lambda_i  - \gamma_i | \geq K \frac{ \log N}{N} \left ( \frac{i}{N}  \right )\right) \leq \frac{(Cq)^{cq^2}}{K^q} \end{equation}
for $i=1,...,\lceil N/2 \rceil$, $N>N_0$, $K>0$, and $q \in \mathbb{N}$ with $q\leq N^\epsilon.$
Furthermore, for $i \leq \log N$ we have that
\begin{equation} \label{e:rigidityedge}\mathbb{P} \left ( | \lambda_i  - \gamma_i | \geq K \left ( \frac{ i}{N}\right)^2 \right ) \leq \frac{(Cq)^{cq^2}}{K^{q/2}}. 
\end{equation}
\end{thm}
\end{section}
\begin{section}{A quadratic formula for the difference of the Stieltjes transforms}

We have defined the difference of the Stieltjes tranform $S_{\rho}(z)$ from its approximation $S_N(z) $ as follows:
\begin{equation} \label{lambda}
    \Lambda_{(j_2)}^{(j_1)} := S_N(z)^{(j_1)}_{ (j_2)} - S_{\rho}(z).
\end{equation}
The fixed-point relation for $S_{\rho}(z)$ is:
\begin{equation} \label{equation-S}
  S_{\rho}(z) = - \frac{1}{z (S_{\rho}(z) +1)}  .
\end{equation}
We will try to deduce something similar for $S_N(z)$. We denote by $\mathbf{x}_k^{(j)}$ the $k$-th row of the scaled matrix $X^{(j)}$ and by $\mathbf{x}_{(j)}^k$ the $k$-th column of the scaled matrix $X_{(j)}.$ We start with the following lemma, after we define the following quantities:
\begin{equation} \label{T-T}
\left [\mathcal{T}_{(j_2)}^{(j_1)} \right ]^k := \frac{1}{N} \Tr \left [ \mathcal{G}_{(j_2)}^{(j_1)} \right ]^k - \frac{1}{N}\Tr \left [ G_{(j_2)}^{(j_1)} \right ] \quad \text{,} \quad \left [T_{(j_2)}^{(j_1)} \right ]_l := \frac{1}{N} \Tr \left [ G_{(j_2)}^{(j_1)} \right ]_l - \frac{1}{N}\Tr \left [ \mathcal{G}_{(j_2)}^{(j_1)} \right ]
\end{equation}
\begin{equation} \label{Y-Y}
\left [\Upsilon_{(j_2)}^{(j_1)} \right ]^k := ( \mathbb{I} - \E_{\textbf{x}_{(j_2)}^k}) (\textbf{x}_{(j_2)}^k)^* \left [\mathcal{G}_{(j_2)}^{(j_1)} \right ]^k \textbf{x}_{(j_2)}^k \quad \text{,} \quad \left [Y_{(j_2)}^{(j_1)} \right ]_l := ( \mathbb{I} - \E_{\textbf{x}_l^{(j_1)}}) (\textbf{x}_l^{(j_1)})^* \left [G_{(j_2)}^{(j_1)} \right ]_l \textbf{x}_l^{(j_1)}
\end{equation}
where $\E_{\textbf{x}^k_{(j)}}$ denotes the expectation with respect only to the randomness of the $k-$th column of $X_{(j)}$ while $\E_{\textbf{x}_l^{(j)}}$ denotes the expectation with respect only to the randomness of the $l-$th row of $X^{(j)}$. The term $\mathbb{I}$ denotes the identity operator applied to a scalar random variable, i.e $\mathbb{I}(a) := a$. The superscipt $k$ in the resolvents denotes an extra removal of the $k-$th column of the scaled matrix $X,$ where $k\in \{j_1+1,...,N \}$. The subscript $l$ in the resolvents denotes an extra removal of the $l-$th row of the scaled matrix $X,$ where $l \in \{j_2+1,...N\}.$
\begin{lem}[Probabilistic fixed-point relation]
For any $z \in Z_{E,\eta}$ and $N \in \mathbb{N}$ we have the following almost fixed point equation for $S_N:$
\begin{align}
{S_N}^{(j_1)}_{(j_2)} &= - \frac{1}{N}\sum \limits_{k=j_1+1}^N
\frac{1}{ z \left(1+{S_N}^{(j_1)}_{(j_2)} + [\mathcal{T}_{(j_2)}^{(j_1)}]^k + [\Upsilon_{(j_2)}^{(j_1)}]^k \right)}  \label{equation-SN1} \\ 
&= - \frac{1}{N} \sum \limits_{l=j_2+1}^N \frac{1}{ z \left(1+ {S_N}^{(j_1)}_{(j_2)} + [T_{(j_2)}^{(j_1)}]_l + [Y_{(j_2)}^{(j_1)}]_l \right)}, \label{equation-SN2}
\end{align}
where the quantities $[T_{(j_2)}^{(j_1)}]_l, [\Upsilon_{(j_2)}^{(j_1)}]^k, [\mathcal{T}_{(j_2)}^{(j_1)}]^k, [Y_{(j_2)}^{(j_1)}]_l,$ are the previous probabilistic error terms.
\end{lem}
\begin{proof} To derive the first equality \eqref{equation-SN1} we use \eqref{G2}:
\begin{equation*}
G_{(j_2), kk}^{(j_1)} = -\frac{1}{z \left(1+  (\x_{(j_2)}^k)^* [\mathcal{G}_{(j_2)}^{(j_1)}]^k
\x_{(j_2)}^k \right)}.
\end{equation*}
It is now enough to show that (for $j_1=j_2=0$):
\beq \label{exp-c-r}
\mathbb{E}_{\x^{k}} \left [ (\textbf{x}^k)^* \mathcal{G}^{k} \textbf{x}^k \right ] = \frac{1}{N} \Tr \left [ \mathcal{G}^{k} \right ].
\eeq
The first identity will then follow, because we can take summation in \eqref{G2} and then divide by $N$.

To prove \eqref{exp-c-r}, we notice that:
\[
\mathbb{E}_{\x^{k}} \left [ (\textbf{x}^k)^* \mathcal{G}^{k} \textbf{x}^k \right ] = \mathbb{E}_{\x^{k}} \left [ \sum \limits_{i,j=1}^N \overline{x_{ki}} \mathcal{G}_{ij}^{k} x_{jk} \right ] = \frac{1}{N} \sum_{i=1}^N \mathcal{G}_{ii}^{k},
\]
where we used the fact that the entries of $X$ are independent with expectation $0$ and variance $\frac{1}{N}$. We also used the fact that the matrix $\mathcal{G}^{k}$ is independent of the column $\x^{k}.$

The second equality $\eqref{equation-SN2}$ comes from the following modified stripping identity, the proof of which is similar to the original:

\begin{equation}\label{e:calG}
\mathcal{G}_{(j_2), ll}^{(j_1)}
= -\frac{1}{z \left(1+  (\textbf{x}_l^{(j_1)}) \left(([X_{(j_2)}^{(j_1)}]_l)^*[X_{(j_2)}^{(j_1)}]_l- z \right)^{-1} (\textbf{x}_l^{(j_1)})^* \right)},
\end{equation}
where $l$ is defined as before. It similarly yields the second part \eqref{equation-SN2}, recalling \eqref{e:traces}.
\end{proof}
We have an easy deterministic bound for the quantities $[\mathcal{T}_{(j_2)}^{(j_1)}]^k$ and $[T_{(j_2)}^{(j_1)}]_l$ in the following lemma:
\begin{lem}
\begin{equation}
\label{e:boundsT12}
\left |[T_{(j_2)}^{(j_1)}]_l \right |, \left |[\mathcal{T}_{(j_2)}^{(j_1)}]^k \right | \leq \frac{|j_1 - j_2|+1}{N\eta} .
\end{equation}
\end{lem}
\begin{proof}
We prove the bound for $[\mathcal{T}_{(j_2)}^{(j_1)}]^k.$ We use the stripping lemma \ref{l:resol1} for the proof. We use the notation $G = {G}_{(j_2)}^{(j_1)}.$
\begin{equation*}
\begin{split}
&[\mathcal{T}_{(j_2)}^{(j_1)}]^k
= -\frac{j_1 - j_2 }{N z} +\frac{1}{N} \left ( \sum \limits_{i \neq k} G_{ii}^{k} - \sum \limits_{i=1}^N G_{ii} \right ) \\
&= -\frac{j_1- j_2}{N z} +\frac{1}{N} \left( \sum \limits_{i \neq k} \left ( G_{ii} - \frac{G_{ik}G_{ki}}{G_{kk}} \right) - \sum \limits_{i=1}^N G_{ii} \right )
\\
&= -\frac{j_1 - j_2}{N z} -\frac{1}{N} \left( \sum \limits_{i\neq k} \frac{G_{ik}G_{ki}}{G_{kk}} -G_{kk} \right)
= -\frac{ j_1 - j_2 }{N z} -\frac{1}{N} \frac{1}{G_{kk}}\sum \limits_{i=1}^N G_{ik}G_{ki} \\
&= -\frac{j_1 - j_2 }{N z} -\frac{(G^2)_{kk}}{NG_{kk}}.
\end{split}
\end{equation*}
We now use \eqref{e:GsImG2} to obtain
\begin{equation*}
\left |[T_{(j_2)}^{(j_1)}]^k \right | \leq \frac{|j_1-j_2|}{N|z|} + \frac{\mathrm{Im} G_{kk}}{|G_{kk}| N \eta}
\end{equation*}
yielding that
\begin{equation}\label{e:minor}
 \left | [T_{(j_2)}^{(j_1)}]^k \right | \leq \frac{|j_1-j_2|+1}{N \eta}.
\end{equation}
A similar argument also works for the quantity $[{T}_{(j_2)}^{(j_1)}]_l,$ replacing  ${G}$ with $\mathcal{G}$ and using the modified stripping lemma for $\mathcal{G}:$
\begin{equation} \label{e:resol2}
    \mathcal{G}^{(j_1)}_{(j_2), ij} = \left [ \mathcal{G}^{(j_1)}_{(j_2)} \right ]_{l,ij} + \frac{\mathcal{G}^{(j_1)}_{(j_2), il}\mathcal{G}^{(j_1)}_{(j_2), li}}{\mathcal{G}^{(j_1)}_{(j_2), ll}},
\end{equation}
which holds for $l \in \{j_2+1,...,N\}$ and $i,j \neq l$.
\end{proof}
It is highly non-trivial to bound the terms $[\Upsilon_{(j_2)}^{(j_1)}]^k$ and $[Y_{(j_2)}^{(j_1)}]_l$ if we only have a four moment bound as an assumption for the matrix entries of $X$. We will bound them in section \ref{Qforms}. In the case of a more powerful assumption like a sub-gaussian decay of the entries, bounds on these quantities would be almost trivial.

We now connect the two similar fixed-point equations yielding a quadratic formula for the difference $\Lambda = S_N - S_{\rho}$ involving a new probabilistic error term $R$ which combines the previous ones and prevents the equation from being trivial with a zero solution.

We start with the equation \eqref{equation-SN1} and modify it by including in it the information provided by equation \eqref{S-rho} to deduce an equation for $\Lambda.$

We will use the identity
\[
\frac{1}{A+\epsilon} = \frac{1}{A} - \frac{\epsilon}{A(A+\epsilon)},
\]
where $\epsilon$ will represent the sum of the two probabilistic error terms plus the quantity $\Lambda$, all multiplied by $z$.

We obtain that (also for any $j_1, j_2$ and so we suppress here the notation):
\begin{equation}
\begin{split}
S_N
&
= - \frac{1}{N} \sum_{k=1}^N
\frac{1}{z (1+ S_{\rho}) + z \Lambda  + z (T^k + \Upsilon^{k})
}
\\
&
=-  \frac{1}{z (1+S_{\rho})} + \frac{1}{N} \sum_{k=1}^N
\frac{1}{z (1+S_{\rho})}\,\frac{ z \Lambda  + {z}(T^k + \Upsilon^{k})}{z \left(1+S_N + T^k + \Upsilon^{k} \right)}
\\
&
= S_{\rho} + \frac{S_{\rho}}{N} \sum_{k=1}^N z \Lambda G_{kk}  - \frac{S_{\rho}}{N} \sum_{k=1}^N G_{kk} z (T^k + \Upsilon^{k})
\\
&
= S_{\rho} + S_{\rho} z \Lambda S_N  - \frac{S_{\rho}}{N} \sum_{k=1}^N G_{kk} z (T^k + \Upsilon^{k}) .
\end{split}\end{equation}
This yields that
\begin{equation} \label{l-eq}
\Lambda =
 z S_{\rho} \Lambda (S_{\rho} + \Lambda)  -  \frac{S_{\rho}}{N} \sum_{k=1}^N G_{kk} z (T^k + \Upsilon^{k} ) .
\end{equation}
We can now define the error term $R$ we were talking about, as:
\begin{align}
R &:= \frac{1}{N} \sum_{k=1}^N G_{kk} (T^k + \Upsilon^{k}).
\end{align}
Of course we can similarly define the error term $R^{(j_1)}_{(j_2)}$.
Now \eqref{l-eq} together with the definition of $R$ yields the following quadratic equation for $\Lambda$:
\begin{equation} \label{e:quadlambda}
z S_{\rho} \Lambda^2 + (z S_{\rho}^2 - 1)\Lambda + z S_{\rho} R = 0.
\end{equation}
Dividing by $z S_{\rho}$, using that $z S_{\rho} = -\frac{1}{1+S_{\rho}}$ and the quadratic formula, yields
\begin{equation}
-(S_{\rho} + 1/2) \pm \sqrt{(S_{\rho} + 1/2)^2 - R}
\end{equation}
as two solutions for our estimated difference $ S_N - S_{\rho}.$ From the definition of $\Lambda$ in \eqref{d:S-Lambda}, it follows that $\mathrm{Im} ( \Lambda ) \geq -\mathrm{Im}(S_{\rho})$. Thus, if we take the branch cut of the square root to be on the positive reals so that the imaginary part of the square root is always positive, we obtain that:
\begin{equation}
\Lambda =  -(S_{\rho} + 1/2) + \sqrt{(S_{\rho} + 1/2)^2 - R}
\end{equation}
We also notice that the second solution, call it $\tilde{\Lambda}$, to  \eqref{e:quadlambda} is given by:
\beq\label{d:tildeLambda}
\tilde{\Lambda} = -\Lambda - 2S_{\rho} - 1.
\eeq
\end{section}

\begin{section}{A first look at the error term \texorpdfstring{$R$}{R} }

We explain here how exactly the probabilistic error term $R$ quantifies the distance of $\Lambda$ from $0$ according to our quadratic equation. We get the following estimates:

\begin{prop} \label{b-R}
Let $z = E + i\eta$. There exists a constant $C>0$, such that:
\begin{equation} \label{e:Lambdabound}
|\Lambda| \leq C \min \left \{ \frac{|R|}{|S_{\rho} + \frac{1}{2}|} , \sqrt{|R|} \right \} ,
\end{equation}
for all $(E, \eta) \in Z_{E, \eta} $ or $E<0$. Furthermore, for any $E\in \mathbb{R}$ and $\eta > 0$ we have that:
\begin{equation}\label{e:imLambdabound}
|\mathrm{Im} \Lambda| \leq C\min \left\{ \frac{|R|}{|S_{\rho} + \frac{1}{2}|} , \sqrt{|R|} \right \}
\end{equation}
and
\begin{equation}\label{e:twosolutionsbound}
\min\{|\Lambda|, |\tilde{\Lambda}|\} \leq C\sqrt{|R|}.
\end{equation}
Analogous statements hold for $\Lambda_{(j_2)}^{(j_1)}$ with $R_{(j_2)}^{(j_1)}$.
\end{prop}

\pagebreak

We will use the first bound near the "hard" edge, because $|S_{\rho} + \frac{1}{2}| \to \infty$ when $z \to 0$ and so it will be most useful there. The $\sqrt{|R|}-$ bound will be most useful in the "bulk" and near the "soft" edge, where $z \sim 1.$
\begin{proof}
We apply the following Lemma which can be found in \cite{cacciapuoti2015bounds}, page 12:
\begin{lem}  \label{sqrt} We denote by $\sqrt{w}$ the square root of $w$ with $\mathrm{Im} \sqrt{w} \geq 0. $
    \begin{itemize}
        \item For any fixed $c>0,$ there exists a constant $C>0$ such that:
        \beq \label{sqrt1}  
            \left |\sqrt{a+b} - \sqrt{a} \right | \leq C \frac{|b|}{\sqrt{|a|+|b|}},
        \eeq
        for all $a,b \in \mathbb{C}$ with $|\mathrm{Im} (a)| \geq c \mathrm{Re} (a)$.
        \item There exists a constant $C>0$ such that:
        \beq \label{sqrt2}
            \left |\mathrm{Im} \left ( \sqrt{a+b} \right ) - \mathrm{Im} ( \sqrt{a} )  \right | \leq C \frac{|b|}{\sqrt{|a|+|b|}},
        \eeq
        for all $a,b \in \mathbb{C}$.
    \end{itemize}
\end{lem}
We apply this Lemma with $a= (S_{\rho} + \frac{1}{2})^2$ and b = $ - R $. 

Since $\mathrm{Im}(S_{\rho} + \frac{1}{2}) \geq 0$, with our choice of branch cut we have that $\sqrt{(S_{\rho} +1/2)^2} = S_{\rho}+1/2$, and we recall that we defined $Z_{E, \eta}$ in \eqref{e:etae} to be exactly the set where $\left | \mathrm{Im} \left [(S_{\rho} + 1/2)^2 \right ] \right | \geq c \left |\mathrm{Re} \left [(S_{\rho} + 1/2)^2 \right ] \right |$ for some $c>0$. These observations are enough for the proof of \eqref{e:Lambdabound} according to the proof of Lemma \ref{sqrt} in \cite{cacciapuoti2015bounds} on page 13.

Analogously, \eqref{e:imLambdabound} follows directly from \eqref{sqrt2} and again the fact that $\sqrt{(S_{\rho} +1/2)^2} = S_{\rho}+1/2$. For the $\sqrt{R}-$ bound, it is again enough the follow the proof of \ref{sqrt} in \cite{cacciapuoti2015bounds} on page 13.

The proof of \eqref{e:twosolutionsbound} follows from Vieta's formula:
\[
|\Lambda | | \tilde{\Lambda} | = | R |,
\]
so that we must have:
\[
\min\{|\Lambda|, |\tilde{\Lambda}|\} \leq \sqrt{|R|}.
\]
We now extend \eqref{e:Lambdabound} for $E<0$. Recalling that $S_N(z) = \frac{1}{N}\sum \limits_{i=1}^N \frac{1}{\lambda_i - E - i\eta}$ and noting that for $E < 0$ the real part of each summand is positive we conclude that $\mathrm{Re}( S_N ) > 0$ for $E < 0$, and similarly to our argument about the imaginary part of $\Lambda$, we see from \eqref{d:S-Lambda} that $\mathrm{Re} (\Lambda) > -\mathrm{Re} (S_{\rho})$ while from \eqref{d:tildeLambda} we see that $\mathrm{Re}( \tilde{\Lambda}) < -\mathrm{Re}(S_{\rho}) - 1$. Since we have that:
\begin{align*}
\mathrm{Re} (\Lambda) &= -\mathrm{Re} (S_{\rho}+1/2) + \mathrm{Re}\left (\sqrt{(S_{\rho} + 1/2)^2 - R} \right) \\
\mathrm{Re} (\tilde\Lambda) &= -\mathrm{Re} (S_{\rho}+1/2) -\mathrm{Re} \left( \sqrt{(S_{\rho} + 1/2)^2 - R} \right),
\end{align*}
we see that $\mathrm{Re}\left(\sqrt{(S_{\rho} + 1/2)^2 - R}\right) > 0$ and thus $|\mathrm{Re} (\Lambda) | < |\mathrm{Re} (\tilde \Lambda) |$ and thus one part of \eqref{e:Lambdabound} follows from \eqref{e:twosolutionsbound}. For the other part of \eqref{e:Lambdabound}, we estimate that:
\begin{equation}
|\Lambda| = \left|\frac{R}{\sqrt{(S_{\rho} + 1/2)^2 - R} + (S_{\rho} + 1/2)}\right| \leq \left|\frac{R}{S_{\rho} + 1/2}\right|,
\end{equation}
where the last inequality follows since both the real and imaginary parts of both summands in the denominator are positive. The fact that $\mathrm{Re} (S_{\rho}+\frac{1}{2})\geq 0$ comes from the definition of the Stieltjes transform which integrates a measure that has a non-negative support.
\end{proof}
What remains now is to focus on the error term
\begin{equation} \label{main-error}
R := \frac{1}{N} \sum_{k=1}^N G_{kk} (T^k + \Upsilon^{k})
\end{equation}
and deduce some bounds for it (also for $R_{(j_2)}^{(j_1)}$). To obtain our optimal bounds, the deterministic bounds on $T^k$ are enough, $\Upsilon^{k}$ will be regarded as a quadratic form and will get bounded in section \ref{Qforms}. The resolvent entry $G_{kk}$ will get bounded in section \ref{RBounds}.
\end{section}

\begin{section}{The four-moments condition and bounds for quadratic forms} \label{Qforms}

Here we obtain bounds for the quadratic form  $ \Upsilon := \frac{1}{N}(\mathbb{I} - \E_{\mathbf{x}} ) \mathbf{x}^* \mathcal{G} \mathbf{x}, $ where $\x$ is a column vector of $X$.

We remark that we only have a four-moment condition as an assumption for the entries of $X$. If instead, we had a sub-gaussian decay assumption for the entries, a bound for $\Upsilon$ would be much easier to deduce by using the Hanson-Wright inequality for quadratic forms as for example in \cite{cacciapuoti2013local}, page 5. The proof here heavily relies on the Ward identity \eqref{ward} which is true both for $G$ and $\mathcal{G}$.

We simplify the notation by taking $j_1=j_2=0,$ but everything would work out for a different case as well.

We prove the following Lemma:
\begin{lem} \label{l:Upsilonbound}
Let $\mathcal{G} = \mathcal{G}$ or $G$. Let
$ \Upsilon := \frac{1}{N}(\mathbb{I} - \E_{\mathbf{x}} ) \mathbf{x}^* \mathcal{G} \mathbf{x} $, assuming \eqref{4moment} and \eqref{truncation} for the elements of $\mathbf{x}$. Then, for any $q \geq 1,$ we have that
\begin{equation}\label{e:upsilonbound1} \E|\Upsilon|^{2q} \leq (Cq)^{cq}\left(\frac{\E \left (\mathrm{Im} \Tr \mathcal{G} \right )^{q}}{N^q(N\eta)^q} +   \frac{\E|\mathcal{G}_{11}|^{2q}}{N^q}+    \frac{\E|\mathcal{G}_{11}|^{q}}{(N\eta)^q}\right). 
\end{equation}
Moreover, we have the more precise inequality:
\begin{equation}\label{e:upsilonbound2} \E|\Upsilon|^{2q} \leq \frac{(Cq)^{cq}}{(N\eta)^q} \left(\E \left (\frac{\mathrm{Im} \Tr \mathcal{G} }{N}\right )^{q} +   \E\left| \frac{ \mathcal{G}_{11} }{\sqrt N}\right|^{q}\right) + (Cq)^{cq} \frac{\E |\mathcal{G}_{11}|^{2q}}{N^q}. \end{equation}
\end{lem}
\begin{proof}
We start by the decomposition:
\begin{equation*}
\Upsilon = \frac{1}{N} \sum \limits_{j\neq l} \overline{x_{j}} x_{l} \mathcal{G}_{jl}  + \frac{1}{N} \sum \limits_{j} (|x_{j}|^2-1) \mathcal{G} _{jj} =\epsilon_{2} + \epsilon_{1},
\end{equation*}
where \begin{equation}\label{d:epsilonk12}
\epsilon_{2}  := \frac{1}{N} \sum \limits_{j\neq l} \overline{x_{j}} x_{l} \mathcal{G}_{jl}  \quad \text{ and } \quad
\epsilon_{1} := \frac{1}{N} \sum \limits_{j} (|x_{j}|^2-1) \mathcal{G} _{jj}.\end{equation}
We use Rosenthal's inequality (Lemma \ref{l:rosenthal}) to obtain for any $q\geq 1,$ that:
\begin{equation}\label{e:Rosenthal} \E | \epsilon_{1} |^{2q} \leq (Cq)^{2q} N^{-2q}  \left [ \left ( \mu_4 \sum \limits_{j=1}^N  \E |\mathcal{G}_{jj}  |^2 \right )^q + \sum \limits_{j=1}^N \E | x_{j} |^{4q}\E | \mathcal{G}_{jj}  |^{2q} \  \right ] .\end{equation}
We use our four-moment assumption \eqref{truncation}:
\begin{equation*}
\E |x_{j}|^{4q} \leq D^{4q-4} N^{q-1}\mu_4,
\end{equation*}
which yields that
\beq \label{e:epsilon1bound}
\E | \epsilon_{1} |^{2q} \leq (Cq)^{2q} N^{-q} \E|\mathcal{G}_{11} |^{2q} .
\eeq
For $\epsilon_{2}$ we will systematically use both Burkholder's and Rosenthal's inequalities, expressed for complex random variables in the Lemmas \ref{l:rosenthal} and \ref{l:burkholder}. Using Burkholder's inequality we obtain:
\begin{multline} \label{burk}
\E | \epsilon_{2} |^{2q}
\leq
 N^{-2q} (C_1q)^{2q} \left [ \E \left ( \sum \limits_{j=2}^N \left | \sum_{k=1}^{j-1}  {x_{k}} \mathcal{G}_{jk} \right |^2 \right )^{q} + \max \limits_{k} \E | x_{k}|^{2q} \sum \limits_{j=2}^N \E \left | \sum_{k=1}^{j-1} x_{k} \mathcal{G}_{jk}  \right |^{2q} \right ] \\
 + N^{-2q} (C_1q)^{2q} \left [ \E \left ( \sum \limits_{j=2}^N \left | \sum_{1\leq k\leq j-1}  {x_{k}} \mathcal{G}_{kj} \right |^2 \right )^{q} + \max \limits_{k} \E | x_{k}|^{2q} \sum \limits_{j=2}^N \E \left | \sum_{k=1}^{j-1} x_{k} \mathcal{G}_{kj}  \right |^{2q} \right ]
\end{multline}
We define the quantities:
\begin{equation}\label{d:Q0}
Q_0  := \sum \limits_{j=2}^N \left | \sum_{k=1}^{ j-1}  {x_{k}} \frac{\mathcal{G}_{jk}}{\sqrt{N}} \right |^2 \quad\text{and}\quad \widehat{Q}_0  := \sum \limits_{j=2}^N \left | \sum_{k=1}^{j-1}  {x_{k}} \frac{\mathcal{G}_{kj}}{\sqrt{N}}  \right |^2 
\end{equation}
The difficult part of the proof will be to bound expectations of powers of these quantities. For the other terms we apply Rosenthal's inequality and \eqref{4moment} getting that:
\begin{align} \label{e:firststep}
&\E | \epsilon_{2}|^{2q} \leq (Cq)^{2q} N^{-q} (\E |Q_0 |^q+\E |\widehat{Q}_0 |^q) +
(Cq)^{4q} N^{-\frac{3q}{2}-1} \times \\
&\left [ \sum \limits_{j=2}^N \E\left \{ \left ( \sum_{k=1}^{j-1}  | \mathcal{G}_{jk}  |^2  \right )^q + \left ( \sum_{k=1}^{j-1}  | \mathcal{G}_{kj}  |^2  \right )^q \right \} + \sum \limits_{j=2}^N  \sum_{k=1}^{ j-1} \E | x_{k} |^{2q} (\E | \mathcal{G}_{jk}  |^{2q} +
\E | \mathcal{G}_{kj}  |^{2q}) \right ] \nonumber
\end{align}
For the last terms we observe that
\begin{equation} \label{e:boundoffdiag}
|\mathcal{G}_{jk} | \leq \frac{1}{2} \sqrt{ \frac{\mathrm{Im}\mathcal{G}_{jj} }{\eta} } + \frac{1}{2} \sqrt{\frac{\mathrm{Im}\mathcal{G}_{kk} }{\eta}},
\end{equation}
which can be obtained as follows. Let $(u_i)_{i=1}^N$ be the normalized eigenvectors of $\mathcal{G} $ and $\lambda_0= 0$. Then
\begin{align*}
|\mathcal{G}_{lj} |= \left|\sum_{q=0}^N \frac{u_{lq}u_{qj}}{\lambda_q - z}\right| &\leq \sum_{q=0}^N \frac{|u_{lq}u_{qj}|}{|\lambda_q - z|}
\leq \frac{1}{2}\sum_{q=0}^N \frac{|u_{lq}|^2 + |u_{qj}|^2}{|\lambda_q - z|} \\ 
&\leq \frac{1}{2}\sqrt{\sum_{q=0}^N \frac{|u_{lq}|^2} {|\lambda_q - z|^2}} + \frac{1}{2}\sqrt{\sum_{q=0}^N \frac{|u_{qj}|^2} {|\lambda_q - z|^2}},
\end{align*}
where in the last step we recall that the eigenvectors are normalized and use Jensen's inequality. Then \eqref{e:boundoffdiag} follows. Using for the first terms the Ward identity \eqref{ward}, $\sum \limits_{l=1}^{N} | \mathcal{G}_{jl}  |^2 \leq \eta^{-1} \mathrm{Im} \ \mathcal{G}_{jj}$, we now get that
\begin{equation}\label{e:epsilon2bound}
\E | \epsilon_{2} |^{2q} \leq (Cq)^{2q} N^{-q}  \left [ \E |Q_0 |^q  + \E |\widehat{Q}_0 |^q \right ]+ \frac{(Cq)^{4q}}{ (N\eta)^{q}}  \E | \mathrm{Im}\mathcal{G}_{11} |^{q}.
\end{equation}

We will now bound the quantity $\E|Q_0 |^q$, and note that $\E|\widehat{Q}_0 |^q$ is similar. We will implement an induction scheme on the quantity $\E|Q_0 |^q$ to gradually decrease its exponent $q$ and finally remove it.
The technique is similar to the one in \cite{gotze2016optimal, gotze2014rate} but we extend Rosenthal and Burkholder inequalities to complex entries and improve the bound near the "soft" edge, where $\frac{\mathrm{Im} \Tr \mathcal{G}}{N} = \mathrm{Im} S_N \approx \mathrm{Im} S_{\rho} \to 0,$ see \eqref{e:upsilonbound1} and \eqref{e:upsilonbound2}. Notice that this improvement is captured only on \eqref{e:upsilonbound2}.

Similarly to \cite{gotze2016optimal} we define the following quantities. This analysis will let us exploit the resolvent identities for the resolvent entries that appear on the quadratic forms $Q_0$ and $\widehat{Q}_0$. Also, notice the decomposition \eqref{e:decomp}.

\pagebreak

\begin{multline}\label{d:Qs} Q_\nu  := \sum \limits_{j=2}^N \left | \sum_{ k=1}^{j-1} x_{j} a_{jk}^{(\nu)} \right |^2,
\quad
 Q_{\nu1}  := \sum \limits_{l=1}^N a_{ll}^{(\nu+1)} ,
 \\
 Q_{\nu2}  := \sum \limits_{l=1}^N \left ( |x_{l}|^2 -1 \right ) a_{ll}^{(\nu+1)}, \quad \text{and} \quad
 Q_{\nu3}  := \sum \limits_{l \neq j} x_{l}\overline{x_{j}} a_{lj}^{(\nu+1)}, 
\end{multline}
where $a_{jk}^{(\nu)}$ are defined recursively via
\begin{equation}\label{d:as}
a_{jk}^{(0)}:=  \frac{\mathcal{G}_{jk}}{\sqrt{N}}  \quad
\text{and}
\quad a_{jk}^{(\nu+1)} := \sum \limits_{l = \max\{j,k\}+1}^{N}  a_{jl}^{(\nu)}\overline{a_{kl}^{(\nu)}}\end{equation}
for $\nu = 0,1,...,L-1, L$ where $L$ is an integer such that $2^{L-1} < q \leq 2^{L}$. 

From \cite{gotze2016optimal} Lemma 5.1 and Corollaries 5.2 and 5.3, we have the following bounds:
\begin{equation}
\label{e:arr}\max \left \{|a_{rr}^{(\nu+1)}|,\; \sum_j |a_{jr}^{( \nu)}|^2 \right \}\leq \left(\frac{\mathrm{Im} \Tr \mathcal{G}  }{N\eta}\right)^{2^\nu-1}\, \frac{\mathrm{Im} \mathcal{G}_{rr} }{N\eta}\end{equation}
To set up our induction scheme we expand the absolute value square and interchange the order of summations in $Q_\nu $:
\begin{align}\label{e:decomp}
Q_{\nu}  &= \sum \limits_{j=2}^N  \sum_{ 1 \leq k_1, k_2 \leq j-1}x_{k_1} \overline{x_{k_2}}  a_{k_1j}^{(\nu)}\overline{a_{k_2j}^{(\nu)}}
= \sum_{1 \leq k_1, k_2 \leq N-1}x_{k_1} \overline{x_{k_2}} \sum \limits_{j= \max\{k_1, k_2\}+1}^{N}  a_{k_1j}^{(\nu)}\overline{a_{k_2j}^{(\nu)}} \nonumber
  \\
  &= \sum_{1 \leq j_1, j_2 \leq N}x_{j_1} \overline{x_{j_2}}   a_{j_1j_2}^{(\nu+1)}
   = Q_{\nu 1}  + Q_{\nu 2}  + Q_{\nu 3}.
\end{align}
We now take the power $2^{L-\nu}$ so as to match the case $\E |Q_0|^{2^L},$ for $\nu = 0,$ after we gradually increase the powers under an induction scheme from $\nu = L$ to $\nu = 0.$ The base case will be $\E |Q_L|$.

Taking power $2^{L-\nu}$ and expectation we obtain that:
\begin{equation} \label{QN}
\E | Q_{\nu}  |^{2^{L-\nu}} \leq 3^{2^{L-\nu}} \left ( \E | Q_{\nu1}  |^{2^{L-\nu}} + \E | Q_{\nu2}  |^{2^{L-\nu}} + \E | Q_{\nu3}  |^{2^{L-\nu}} \right ).
\end{equation}
Firstly, by the definition of $Q_{\nu1} $ and the definition of the coefficients $a_{ll}^{(\nu+1)}$ and \eqref{e:arr} we can check that:
\beq \label{Q1}
\E | Q_{\nu1}  |^{2^{L-\nu}} \leq \E \left [\frac{\mathrm{Im} \Tr \mathcal{G}  }{N\eta}  \right ]^{2^L}.
\eeq
We apply Rosenthal's inequality for the quantity
$ Q_{\nu 2} $ getting that:
\begin{multline}
\E | Q_{\nu2}  |^{2^{L-\nu}}
\leq  (Cq)^q \left [\E \left ( \sum \limits_{l=1}^N  | a_{ll}^{(\nu+1)} |^2  \right )^{2^{L-(\nu+1)}} + \sum \limits_{l=1}^N \E | |x_{l}|^2 -1 |^{2^{L-\nu}} |a_{ll}^{(\nu+1)} |^{2^{L-\nu}} \right ]
\\
\leq (Cq)^q \left [ \E \left ( \sum \limits_{l=1}^N | a_{ll}^{(\nu+1)} |^2 \right )^{2^{L-(\nu+1)}} + N^{2^{L-(\nu+1)}} \frac{1}{N} \sum \limits_{l=1}^N  \E |a_{ll}^{(\nu+1)} |^{2^{L-\nu}} \right ] ,
\end{multline}
where we used \eqref{truncation} in the last line. We notice that the first term is bounded above by the second term by Jensen's inequality, so we obtain that:
\begin{equation}
\E | Q_{\nu2}  |^{2^{L-\nu}}\leq (Cq)^q N^{2^{L-(\nu+1)}} \frac{1}{N} \sum \limits_{l=1}^N  \E |a_{ll}^{(\nu+1)} |^{2^{L-\nu}}
\leq (Cq)^q J_\nu,
\end{equation}
where we used \eqref{e:arr} and introduced the notation
\begin{equation}
J_{\nu}: = N^{-2^{L- (\nu+1)}}\E \left [ \left| \frac{\mathrm{Im} \Tr \mathcal{G}  }{N\eta} \right|^{2^L- 2^{L-\nu}}\left|\frac{\mathrm{Im} \mathcal{G}_{11} (z)}{\eta}\right|^{2^{L-\nu}} \right ].
\end{equation}
Now we apply Burkholder's inequality to $\E|Q_{\nu3}|^q$ to obtain a bound which involves $\E|Q_{\nu+1}|^{q/2}$ and $\E|\widehat{Q}_{\nu+1}|^{q/2}$ as in \eqref{burk}. We use Rosenthal's inequality to bound the other term arising from the application of Burkholder's inequality:
\begin{align}
\E|Q_{\nu3}|^{2^{L-\nu}} &\leq \E \left |\sum_{ j_1 \neq j_2}x_{j_1} \overline{x_{j_2}}   a_{j_1j_2}^{(\nu+1)} \right |^{2^{L-\nu}} \nonumber
\\
&\leq
 (Cq)^{q} \left (\E |Q_{\nu+1} |^{2^{L-(\nu+1)}} + \E |\widehat{Q}_{\nu+1} |^{2^{L-(\nu+1)}} \right ) \nonumber \\
 &+ (Cq)^q
 \E|x_1|^{2^{L-\nu}} \sum \limits_{j=2}^n
 \E\left( \left | \sum \limits_{k=1}^{j-1} a_{jk}^{(\nu+1)} x_k \right |^{2^{L-\nu}} + \left | \sum \limits_{k=1}^{j-1} a_{kj}^{(\nu+1)} x_k \right |^{2^{L-\nu}}\right) \nonumber
 \\
&\leq
 (Cq)^{q} \left ( \E |Q_{\nu+1} |^{2^{L-(\nu+1)}} + \E |\widehat{Q}_{\nu+1} |^{2^{L-(\nu+1)}} \right ) \nonumber \\
&+ (Cq)^{2q} N^{2^{L-(\nu+2)}} \frac{1}{N} \sum \limits_{j=2}^N \E\left ( \sum_{1\leq l\leq N, l \neq j} | a_{jl}^{(\nu+1)}|^2 \right )^{2^{L-(\nu+1)}} \nonumber
\\
&+ (Cq)^{q}  N^{2^{L-(\nu+1)}} \frac{1}{N^2} \sum \limits_{j=2}^N \sum_{1\leq l\leq N, l \neq j} \E | a_{jl}^{(\nu+1)} |^{2^{L-\nu}}
\end{align}
The resulting terms are bounded by \eqref{e:arr} and the following argument. By H\"older inequality and the definition \eqref{d:as}, we obtain that:
\begin{align} \label{e:arjhighpower}
&\frac{1}{N^2} \sum \limits_{j=2}^N \sum_{1\leq r\leq N, r\neq j}|a_{rj}^{(\nu+1)}|^{2^{L-\nu}} 
\leq
\frac{1}{N^2} \sum \limits_{j=2}^N  \sum_{1\leq r\leq N, r\neq j} \left|\sum_{l_1} |a_{r{l_1}}^{(\nu+1)}|^2 \sum_{l_2} |a_{{l_2}j}^{(\nu+1)}|^2\right|^{2^{L-(\nu+1)}}  \nonumber \\
 &\leq
  \frac{1}{N^2} \sum \limits_{j=2}^N  \sum_{1\leq r\leq N, r\neq j} \left|a_{rr}^{(\nu)}a_{jj}^{( \nu)}\right|^{2^{L-(\nu+1)}} \leq \left(\frac{1}{N}\sum_j |a_{jj}^{(\nu)}|^{2^{L-(\nu+1)}}\right)^2 \nonumber \\
 &\leq
    \frac{1}{N}\sum_j |a_{jj}^{(\nu)}|^{2^{L-\nu}} ,
\end{align}
where in the last step we used Jensen's inequality. This yields that:
\begin{align}
\E|Q_{\nu3}|^{2^{L-\nu}}
&\leq
(Cq)^{q} \left (\E |Q_{\nu+1} |^{2^{L-(\nu+1)}} + \E |\widehat{Q}_{\nu+1} |^{2^{L-(\nu+1)}} \right ) \nonumber
\\
&+
 (Cq)^{2q} N^{2^{L-(\nu+2)}} \E \left |  \left(\frac{\mathrm{Im} \Tr \mathcal{G}  }{N\eta}\right)^{2^{\nu+1}-1} \frac{\mathrm{Im} \mathcal{G}_{11} }{N\eta} \right |^{2^{L-(\nu+1)}} \nonumber \\
&+ (Cq)^{q}  N^{2^{L-(\nu+1)}} \E   \left|\left(\frac{\mathrm{Im} \Tr \mathcal{G}  }{N\eta}\right)^{2^{\nu}-1}\, \frac{\mathrm{Im} \mathcal{G}_{11} }{N\eta} 
\right|^{2^{L-\nu}} \nonumber
\\
&\leq
(Cq)^{q} \left (\E |Q_{\nu+1} |^{2^{L-(\nu+1)}} + \E |\widehat{Q}_{\nu+1} |^{2^{L-(\nu+1)}} \right )+
 (Cq)^{2q}(J_{\nu+1} + J_{\nu}).
\end{align}
Same bounds hold for $\widehat{Q}_{\nu}$ with $\widehat{Q}_{\nu1}$ , $\widehat{Q}_{\nu2}$ and $\widehat{Q}_{\nu3}$ defined analogously. Thus, by \eqref{QN}, we have proved that:
\begin{align*}
&\E | Q_{\nu}  |^{2^{L-\nu}} \leq 3^{2^{L-\nu}} \left ( \E | Q_{\nu1}  |^{2^{L-\nu}} + \E | Q_{\nu2}  |^{2^{L-\nu}} + \E | Q_{\nu3}  |^{2^{L-\nu}} \right ) \leq \\
&(Cq)^{2q} \left (  \E \left [\frac{\mathrm{Im} \Tr \mathcal{G}  }{N\eta}  \right ]^{2^L} +  J_{\nu} +  \E |Q_{\nu+1} |^{2^{L-(\nu+1)}} + \E |\widehat{Q}_{\nu+1} |^{2^{L-(\nu+1)}} + (J_{\nu+1} + J_{\nu})  \right ),
\end{align*}
for any $\nu = 0,1,...,L.$
By an induction argument we obtain that:
\begin{align} \label{induction}
 &\E(|Q_0|^{2^L}+|\widehat{Q}_0|^{2^L}) \leq \nonumber
 \\
 &\leq (Cq)^{cq}\left(\E(|Q_{L}|+|\widehat{Q}_{L}|) +  \sum_{\nu=0}^{L-1} \left( \E |Q_{\nu1}|^{2^{L-\nu}}+\E| \widehat{Q}_{\nu1}|^{2^{L-\nu}}+ \E|Q_{\nu2}|^{2^{L-\nu}}+\E|\widehat{Q}_{\nu2}|^{2^{L-\nu}}
 + J_\nu + J_{\nu+1}\right)\right) \nonumber
 \\
 &\leq
  (Cq)^{cq}\left(\E|Q_{L}| + \E |\widehat{Q}_{L}| + \sum_{\nu=0}^{L-1} \E \left (|Q_{\nu1}|^{2^{L-\nu}}+|\widehat{Q}_{\nu1}|^{2^{L-\nu}} \right )+ \sum_{\nu=0}^{L}
 J_\nu\right).
\end{align}
Now, we bound the three terms on the RHS. For $0 \leq \nu \leq L$ we have that
\begin{align} \label{JN}
J_\nu &= \eta^{-2^L}\E \left [ \left( \frac{\mathrm{Im} \Tr \mathcal{G}  }{N} \right)^{2^L}\left|\frac{\mathrm{Im} \mathcal{G} _{11}}{\sqrt N}\, \frac{N}{\mathrm{Im} \Tr \mathcal{G}  }\right|^{2^{L-\nu}} \right ] \nonumber
\\
&\leq \eta^{-2^L}\E\left[ \left( \frac{\mathrm{Im} \Tr \mathcal{G}  }{N}\right)^{2^{L}} + \left|\frac{\mathrm{Im} \mathcal{G} _{11}}{\sqrt N}\right|^{2^{L}}\right ] ,
\end{align}
where in the last step we used Young's inequality and $0 \leq \nu \leq L.$

Using \eqref{e:decomp} and \eqref{e:arr} we obtain
\begin{equation} \label{QL}
\E |Q_L| = \E  \left | \sum_{1 \leq j_1, j_2 \leq N}x_{j_1} \overline{x_{j_2}}   a_{j_1j_2}^{(L+1)} \right | = \sum_{j=1}^N \E |a_{jj}^{(L+1)}| \leq \E\left(\frac{\mathrm{Im} \Tr \mathcal{G}  }{N\eta}\right)^{2^L},
\end{equation}
and a similar bound for $\E |\widehat{Q}_L|.$ 

Upon substitution of \eqref{Q1}, \eqref{JN} and \eqref{QL} into \eqref{induction}, we get that:
\begin{equation}\label{e:Q0bound}
\E |Q_0|^q +\E |\widehat{Q}_0|^q\leq (Cq)^{cq}\eta^{-q}\left( \E\left(\frac{\mathrm{Im} \Tr \mathcal{G}  }{N}\right)^{q} + \E\left|\frac{\mathrm{Im} \mathcal{G} _{11}}{\sqrt N}\right|^{q}\right),
\end{equation} 
We can now go back to \eqref{e:epsilon2bound} and replace this bound on the RHS. The desired bounds \eqref{e:upsilonbound1} and \eqref{e:upsilonbound2} on $\E |\Upsilon|^{2q}$  will follow from \eqref{e:epsilon1bound} and \eqref{e:epsilon2bound}.
\end{proof}

\pagebreak

\underline{Rosenthal's and Burkholder's inequalities.}

Here we state Rosenthal's and Burkholder's inequalities adapted to complex variables and non-Hermitian bilinear forms (useful for $\mathcal G$). We are given complex random variables $x_1, ..., x_N$ with i.i.d. real and imaginary parts and $\E \left [ \mathrm{Re} (x_j) \right ] = \E \left [ \mathrm{Im} (x_j) \right ] = 0$ and $\E \left | \mathrm{Re} (x_j) \right |^2 = \E \left | \mathrm{Im} (x_j) \right |^2 = 1/2$, for $j=1,...,N$ as in our setup. We assume that for $j=1,...,N: \E |x_j|^p \leq \mu_{p}$ for any $p\geq 1$, so that all the moments exist and are bounded by these quantities. In our setup, they may also depend on $N.$

The following Lemma is our version of Rosenthal's inequality. Here we have a vector $\mathbf{a}=(a_1,...,a_N)$ of complex scalars. It is easy to prove by separating real and imaginary parts of the random variables and using Lemma 7.1 of \cite{gotze2016optimal}:
\begin{lem}[Rosenthal's inequality] \label{l:rosenthal} For any $p \geq 1,$ there exists a constant $C_1$ such that
\beq
\E \left | \sum_{j=1}^N a_j x_j \right |^{p} \leq (C_1 p)^p \left [ \left(\sum_{j=1}^N |a_j|^2\right)^{p/2}+ \mu_{p}\sum_{j=1}^N|a_j|^p \right ].
\eeq
\end{lem}
\begin{proof}
    We let $x_j = \mathrm{Re}(x_j) + i \mathrm{Im} (x_j)$ for $j=1,..,N$ and then use Lemma 7.1 of \cite{gotze2016optimal}.
\end{proof}
The following Lemma is our version of Burkholder's inequality. Here, we have a family of complex scalars $(a_{ij})_{i,j=1}^N$. This is an extension of Lemma 7.3 of \cite{gotze2016optimal} for complex entries and non-Hermitian quadratic forms. Let $$Q := \sum \limits_{j \neq k}a_{jk}x_j \overline{x_k}.$$
\begin{lem}[Burkholder's Inequality]\label{l:burkholder}
For any $q\geq 1$, there exist absolute constants $C_1, C_2$ such that:
\begin{align*}
\E |Q|^q &\leq (C_1q)^q \left ( \E \left [ \sum \limits_{j=2}^n \left | \sum \limits_{k=1}^{j-1} a_{jk} x_k \right |^2 \right ]^{q/2} + \mu_{q} \sum \limits_{j=2}^n \E \left | \sum \limits_{k=1}^{j-1} a_{jk} x_k \right |^q \right ) \\
&+ (C_2q)^q \left ( \E \left [ \sum \limits_{j=2}^n \left | \sum \limits_{k=1}^{j-1} a_{kj} x_k \right |^2 \right ]^{q/2} + \mu_{q} \sum \limits_{j=2}^n \E \left | \sum \limits_{k=1}^{j-1} a_{kj} x_k \right |^q \right ) .
\end{align*}
\end{lem}
\begin{proof}
We introduce for $j=1,2,...,N$ the random variables
\beq
\xi_j := x_j \sum_{k=1}^{j-1} a_{jk} x_k \ , \ \widehat{\xi_j} := x_j \sum_{k=1}^{j-1} a_{kj} x_k
\eeq
We let $\mathcal{R}_j := \sigma( \xi_1,...,\xi_j )$ be the sigma-algebra generated by the first $j$ random variables $\xi_1,...,\xi_j$. We observe that the $\xi_j$ and $\widehat{\xi}_j $ are $\mathcal{R}_j-$measurable with $\E [ \xi_j | \mathcal{R}_{j-1} ] = 0$ and $\E [ \widehat{\xi_j} | \mathcal{R}_{j-1} ] = 0,$ which means that they form martingale differences. Next, we write $Q$ as
\beq
Q = \sum \limits_{j=2}^{n} \xi_j + \sum \limits_{j=2}^n \widehat{\xi_j}
\eeq
and so,
$$ \E |Q|^q \leq C^q \E \left | \sum \limits_{j=2}^n \xi_j \right |^q + C^q \E \left | \sum \limits_{j=2}^n \widehat{\xi_j} \right |^q .$$
We now apply a general Burkholder-Rosenthal Inequality as seen in \cite{osekowski2012note}, analogous to Lemma 7.2 from \cite{gotze2016optimal},  to the martingale-difference sequences $\xi_1,...,\xi_n$ and $\widehat{\xi_1},..., \widehat{\xi_n}.$ We just have to evaluate $\E[ \xi_j^2 | \mathcal{R}_{j-1} ]$ and $\E |\xi_j |^q ,$ for $j=1,2,...,N.$ The case is similar for the $\widehat{\xi_1},..., \widehat{\xi_n} $ random variables. We therefore observe that:
\begin{align*}
\E \left [ \left | \xi_j|^2 \right | \mathcal{R}_{j-1} \right ] &= \E|x_j|^2  \left | \sum \limits_{k=1}^{j-1} a_{jk} x_k \right |^2 = \left | \sum \limits_{k=1}^{j-1} a_{jk} x_k \right |^2, \\
\E | \xi_j |^q &= \E | \zeta_j|^q \E \left | \sum \limits_{k=1}^{j-1} a_{jk} x_k \right |^q \leq \mu_{q} \ \E \left | \sum \limits_{k=1}^{j-1} a_{jk} x_k \right |^q,
\end{align*}
and the lemma follows after applying the inequality seen in \cite{osekowski2012note}.
\end{proof}
\end{section}

\begin{section}{The resolvent bounds} \label{RBounds}

In this section we will bound the quantities $\mathbb{E} | G_{kk}|^q $ that appear in the main error term $R$ defined in \eqref{main-error}.

We remind that our strategy is to bound high powers of the expected value of this term (i.e. bound $\mathbb{E}|R|^q $) and use Markov's inequality to show that it is small in probability.

Let
\begin{equation}\label{d:lambda}
\lambda^{(j_1)}_{(j_2)} : = \max \left \{ \left |\Lambda^{(j_1)}_{(j_2)} \right |\chi_{S_{E,\eta}}, \min \left \{ \left |\Lambda^{(j_1)}_{(j_2)} \right |, \left |\tilde \Lambda^{(j_1)}_{(j_2)} \right | \right \}, \left | \mathrm{Im} \Lambda^{(j_1)}_{(j_2)} \right | \right \}
\end{equation}
By Proposition \ref{b-R},
$
\E \left | \lambda^{(j_1)}_{(j_2)} \right |^{2q} \leq C^{2q} \E \left | R^{(j_1)}_{(j_2)} \right |^q.
$
Taking expectation of a power $q\geq 1$ of $\left | R^{(j_1)}_{(j_2)} \right |$ we obtain (as in \cite{cacciapuoti2015bounds}), using \eqref{e:minor}, \eqref{e:upsilonbound1}, \eqref{e:traces}, Jensen's and Cauchy-Schwartz inequalities:
\begin{align*}\label{e:calc}
&\E \left | R^{(j_1)}_{(j_2)} \right |^q \leq
 \frac{1}{N} \sum_{k= j_1 +1}^N \E \left |\left ( [{T}_{(j_2)}^{(j_1)}]^k
  + [ \Upsilon^{(j_1)}_{(j_2)} ]^k \right)G^{(j_1)}_{(j_2) kk} \right|^q \leq
  \\
&\E\left|\left( [T_{(j_2)}^{(j_1+1)}]
  + [\Upsilon^{(j_1+1)}_{(j_2)}] \right) G^{(j_1)}_{(j_2), 11}\right|^q
  \\
&\leq \frac{\E \left |C (|j_1- j_2|+1) G^{(j_1)}_{(j_2),11}\right|^q}{(N\eta)^q} + \left|C \right|^q\sqrt{\E|G^{(j_1)}_{(j_2), 11}|^{2q}\E\left|
  [\Upsilon^{(j_1+1)}_{(j_2)}] \right|^{2q}}
  \\
&\leq \frac{\E \left |C (|j_1- j_2|+1) G^{(j_1)}_{(j_2),11}\right|^q}{(N\eta)^q} + \left|Cq\right|^{cq}\sqrt{ \E|G^{(j_1)}_{(j_2),11}|^{2q}\frac{\E|\mathcal{G}^{(j_1+1)}_{(j_2),11}|^{2q}}{N^q}}
\\
&+
 \left|C q\right|^{cq}\sqrt{\E|G^{(j_1)}_{(j_2),11}|^{2q}   \left(   \frac{\E ( \mathrm{Im} \Tr \mathcal{G}^{(j_1+1)}_{(j_2)})^q}{(N\eta)^qN^q} +  \frac{\E|\mathcal{G}^{(j_1+1)}_{(j_2),11}|^{q}}{(N\eta)^q}      \right)}
\\
&\leq \frac{\E \left |C (|j_1- j_2|+1) G^{(j_1)}_{(j_2),11}\right|^q}{(N\eta)^q} + \left|C q\right|^{cq} \sqrt{ \E|G^{(j_1)}_{(j_2),11}|^{2q}\frac{\E|\mathcal{G}^{(j_1+1)}_{(j_2),11}|^{2q}}{N^q}}
\\
&+
 C^q\sqrt{ \E|G^{(j_1)}_{(j_2),11}|^{2q}
\frac{(Cq)^{cq}}{(N\eta)^q} \left( \left ( \frac{\left|j_1+1-j_2 \right|}{N\eta} \right)^{q} +  \E \left(\mathrm{Im} S_{\rho} +  \mathrm{Im} {\Lambda}^{(j_1+1)}_{(j_2)}\right)^q +   \E|\mathcal{G}^{(j_1+1)}_{(j_2),11}|^{q}\right)}
\\
&\leq \frac{ \left |C (|j_1- j_2|+1) \right |^q \sqrt{ \E |G^{(j_1)}_{(j_2),11}|^{2q} } }{(N\eta)^q} + \left|C q\right|^{cq} \sqrt{ \E|G^{(j_1)}_{(j_2),11}|^{2q} } \sqrt{ \frac{\E|\mathcal{G}^{(j_1+1)}_{(j_2),11}|^{2q}}{N^q}} \numberthis
\\
&+
 \left|Cq\right|^{cq}\frac{ \sqrt{ \E|G^{(j_1)}_{(j_2),11}|^{2q}}} {(N\eta)^{\frac{q}{2}}} \left[ \left ( \frac{\left|j_1+1- j_2 \right|}{N\eta} \right )^{q/2} +  \sqrt{\left ( \frac{C}{\sqrt{|z|}} \right )^q +\E|\lambda_{(j_2)}^{(j_1)}|^q} + \sqrt{\E|\mathcal{G}^{(j_1+1)}_{(j_2),11}|^{q}}\right] .
\end{align*}
In the second to last line, the term $\frac{|j_1+1-j_2|}{N|z|}$ arises from equation \eqref{e:traces} and $| \lambda_{(j_2)}^{(j_1)} - \lambda_{(j_2)}^{(j_1+1)}| = \mathcal{O} \left ( \frac{1}{N\eta} \right )$ similarly with the proof of \eqref{e:minor}. We notice here that in order to get finite bounds near the "hard" edge, we have to extinguish the $\frac{1}{\sqrt{|z|}}$ term, that came from the imaginary part of $S_{\rho}.$ This suggests that we should multiply everything with $|z|^q$ and actually get bounds for the quantity $\E |z R|^q$, which in turn suggests bounds for the resolvent entries of the form $\E|\sqrt{z} G_{11}|^q$ and $\E|\sqrt{z} \mathcal{G}_{11} |^q$.

Multiplying by $|z|^q$, \eqref{e:calc} becomes:
\begin{align*}
    &\E \left | z R^{(j_1)}_{(j_2)} \right |^q \leq \\
    &\frac{ \left |C \sqrt{z} (|j_1- j_2|+1) \right |^q \sqrt{ \E |\sqrt{z} G^{(j_1)}_{(j_2),11}|^{2q} } }{(N\eta)^q} + \left|C q\right|^{cq} \sqrt{ \E|\sqrt{z} G^{(j_1)}_{(j_2),11}|^{2q} } \sqrt{ \frac{\E|\sqrt{z} \mathcal{G}^{(j_1+1)}_{(j_2),11}|^{2q}}{N^q}} \\
    &+
    \left|Cq\right|^{cq}\frac{ |z|^{q/4} \sqrt{ \E|\sqrt{z} G^{(j_1)}_{(j_2),11}|^{2q}}} {(N\eta)^{\frac{q}{2}}} \left[ \left|j_1+1- j_2 \right|^{q/2} +  \sqrt{\E|\sqrt{z} \lambda_{(j_2)}^{(j_1)}|^q} + \sqrt{\E| \sqrt{z} \mathcal{G}^{(j_1+1)}_{(j_2),11}|^{q}}\right] \\
    &\leq
    C^q \frac{ |z|^{q/2} \left ( |j_1- j_2|+1 \right )^q  }{(N\eta)^q} \sqrt{ \E |\sqrt{z} G^{(j_1)}_{(j_2),11}|^{2q} } +  \left(C q\right)^{cq} \sqrt{ \E|\sqrt{z} G^{(j_1)}_{(j_2),11}|^{2q}  \frac{\E|\sqrt{z} \mathcal{G}^{(j_1+1)}_{(j_2),11}|^{2q}}{N^q} } \\
    &(Cq)^{cq} \frac{ |z|^{q/4} \sqrt{ \E|\sqrt{z} G^{(j_1)}_{(j_2),11}|^{2q}}} {(N\eta)^{\frac{q}{2}}} \left[ |j_1 - j_2 +1|^{q/2}+\sqrt{\E|\sqrt{z} \lambda_{(j_2)}^{(j_1)}|^q} + \sqrt{\E| \sqrt{z} \mathcal{G}^{(j_1+1)}_{(j_2),11}|^{q}}\right] . \numberthis
\end{align*} 

We now try to deduce a bound for $\E |\sqrt{z}\lambda_{(j_2)}^{(j_1)}|^{2q}$. Since for any $x, \delta>0$, 
$$x^{1/4} \leq \delta x + \delta^{-1/3},$$
for $x = \E|\sqrt{z} \lambda_{(j_2)}^{(j_1)}|^{2q} $ by Cauchy-Schwartz we obtain that:
\begin{align*}
     x &\leq (Cq)^q \frac{ |z|^{q/2} \left ( |j_1- j_2|+1 \right )^q  }{(N\eta)^q} \sqrt{ \E |\sqrt{z} G^{(j_1)}_{(j_2),11}|^{2q} } + \left(C q\right)^{cq} \sqrt{ \E|\sqrt{z} G^{(j_1)}_{(j_2),11}|^{2q} } \sqrt{ \frac{\E|\sqrt{z} \mathcal{G}^{(j_1+1)}_{(j_2),11}|^{2q}}{N^q}} \\
     &+
     (Cq)^{cq} \frac{ |z|^{q/4} \sqrt{ \E|\sqrt{z} G^{(j_1)}_{(j_2),11}|^{2q}}} {(N\eta)^{\frac{q}{2}}} \left[ \delta x + \delta^{-1/3} + \sqrt{\E| \sqrt{z} \mathcal{G}^{(j_1+1)}_{(j_2),11}|^{q}} + |j_1-j_2+1|^{q/2} \right] .
\end{align*}
Writing this as
\begin{align*}
    &x \leq A + B \left [ \delta x + \delta^{-1/3} + \Gamma \right ] \Leftrightarrow \\
    &(1 - B \delta ) x \leq A + B \delta^{-1/3} + B \Gamma,
\end{align*}
we see that we can set 
$$ \delta = \frac{B^{-1}}{2} = \left [ 2(Cq)^{cq} \frac{ \sqrt{ \E|\sqrt{z} G^{(j_1)}_{(j_2),11}|^{2q}}} {(N\eta)^{\frac{q}{2}}} \right ]^{-1} ,$$
to get that:
\begin{align}\label{e:bb}
\E |\sqrt{z}\lambda_{(j_2)}^{(j_1)}|^{2q} 
&\leq (Cq)^{cq} \frac{|z|^{q/2}(|j_1-j_2|+1)^{cq}}{(N\eta)^{q}} \left((\E|\sqrt{z} {G}^{(j_1)}_{(j_2),11}|^{2q})^{1/2} \right) \nonumber \\ 
&+ (Cq)^{cq} \frac{|z|^{q/3} \left( \E|\sqrt{z} {G}^{(j_1)}_{(j_2),11}|^{2q} \right )^{2/3}}{(N\eta)^{2q/3}} + \left (C q\right)^{cq} \sqrt{ \E|\sqrt{z} G^{(j_1)}_{(j_2),11}|^{2q} } \sqrt{ \frac{\E|\sqrt{z} \mathcal{G}^{(j_1+1)}_{(j_2),11}|^{2q}}{N^q}} \nonumber \\
&+\left(C q \right )^{cq} \sqrt{ \frac{ |z|^{q/2} \E|\sqrt{z}G^{(j_1)}_{(j_2),11}|^{2q}}{ (N\eta)^q}} \left [ \sqrt{\E|\sqrt{z}\mathcal{G}^{(j_1+1)}_{(j_2),11}|^{q}} + |j_1-j_2+1|^{q/2} \right ] .
\end{align}
\begin{lem}[Resolvent bounds] \label{l:Gkk}
Let $E \leq 4, \eta \leq \eta_0$ and $q \leq \left (\frac{N \eta}{|\sqrt{z}|} \right )^{1/4}$ with $\frac{N\eta}{|\sqrt{z}|} \geq M$ for some suitable large constant $M>0$. Assume that $j_1, j_2\in \{0,...,N-1\}$ are such that $0 \leq | j_1 - j_2| \leq  C_1 q$ for a uniform constant $C_1$. Then with definitions as before, 
\[
\E|\sqrt z G^{(j_1)}_{(j_2), 11}|^{q} \leq C^{q} \quad \text{and} \quad \E|\sqrt z \mathcal{G}^{(j_1+1)}_{(j_2), 11}|^{q} \leq C^{q},
\]
for some constant $C$.
\end{lem}
\begin{proof}
We will implement an induction argument similar to \cite{cacciapuoti2015bounds, gotze2016optimal}.  The induction hypothesis will be that for $\eta_i = \eta_0/16^i$ for some constant $\eta_0$ and any $j_1,j_2 $ with $|j_1-j_2| \leq L_i := C_1 \left (\frac{N \eta_i}{|\sqrt{E}|} \right )^{1/4}.$
\begin{equation}\label{e:inductionhyp}
\E|\sqrt{z} G^{(j_1)}_{(j_2), 11}(\eta_i)|^q \leq C_0^{q} \quad \text{and} \quad \E|\sqrt{z} \mathcal{G}^{(j_1+1)}_{(j_2), 11}(\eta_i)|^q \leq C_0^{q}
\end{equation}
for $q \leq \left (\frac{N \eta_i}{|\sqrt{E}|} \right )^{1/4}$ for a universal constant $C_0$. 

\underline{Induction basis:} We notice that this holds to initiate our induction for $\eta_0$ as a constant. This can be proved for example by the inequality $|S_N(z)| \leq \frac{1}{\eta}$ as seen in section \ref{S-Re}, point 4.

\underline{Induction step:} Letting $\eta_{i+1} = \eta_i/16$ and $L_{i+1} =  L_i/2 $ we will show that inequality \eqref{e:inductionhyp} taken at $\eta_i$ implies the same inequality with the same constant $C_0$ for $\eta_{i+1}$. 

From the induction hypothesis and Lemma \ref{G11eta/s} we see that
\begin{equation}
\begin{split}\label{e:s}
\E|\sqrt z G^{(j_1)}_{(j_2), 11}(\eta_{i+1})|^{q} \leq (16C_0)^{q} \quad \text{and} \quad \E| \sqrt{z}\mathcal{G}^{(j_1+1)}_{(j_2), 11}(\eta_{i+1})|^{q} \leq (16C_0)^{q}
\end{split}
\end{equation}
for any $j_1, j_2$ with $| j_1 - j_2 | \leq L_i$. This will need to be improved to the bound $C_0^q$ for any $k_1, k_2$ with $|k_1-k_2| \leq 2^{-1} L_i$.

We will use the inequality $|k_1 - k_2| \leq L_{i+1}$ in equation \eqref{e:lambdabootstrap}.

From \eqref{G2}, \eqref{T-T}, \eqref{Y-Y}, \eqref{S-rho} and \eqref{d:tildeLambda} we obtain that (for $\eta_{i+1})$:
\begin{align}
&G^{(k_1)}_{(k_2), 11} = S_{\rho} + zS_{\rho} \left (\Lambda^{(k_1)}_{(k_2)} + [\mathcal{T}_{(k_2)}^{(k_1)}]^N + [\Upsilon^{(k_1)}_{(k_2)}]^N \right ) G^{(k_1)}_{(k_2), 11}  \\
&G^{(k_1)}_{(k_2),11} = S_{\rho} - zS_{\rho} \left (\tilde{\Lambda}^{(k_1)}_{(k_2)}- [\mathcal{T}_{(k_2)}^{(k_1)}]^N - [\Upsilon^{(k_1)}_{(k_2)}]^N \right ) G^{(k_1)}_{(k_2), 11} \nonumber \\ 
&+ zS_{\rho}(2S_{\rho} + 1) G^{(k_1)}_{(k_2), 11}.
\end{align}
The analogous statements for $\mathcal{G}^{(k_1)}_{(k_2),11}$ follow similarly from \eqref{e:calG}, \eqref{T-T}, \eqref{Y-Y}, \eqref{S-rho} and \eqref{d:tildeLambda} (for $\eta_{i+1}$):
\begin{align} \label{e:bootstrapcalG}
&\mathcal{G}^{(k_1)}_{(k_2), 11} = S_{\rho} +z S_{\rho} \left (\Lambda^{(k_1)}_{(k_2)} + [T_{(k_2)}^{(k_1)}]_N + [Y^{(k_1)}_{(k_2)}]_N \right] \mathcal{G}^{(k_1)}_{(k_2), 11} \\
&\mathcal{G}^{(k_1)}_{(k_2), 11} = S_{\rho} - z S_{\rho} \left (\tilde{\Lambda}^{(k_1)}_{(k_2)} - [T_{(k_2)}^{(k_1)}]_N - [Y^{(k_1)}_{(k_2)}]_N \right] \mathcal{G}^{(k_1)}_{(k_2), 11} \nonumber \\
&+ zS_{\rho}(2S_{\rho} + 1) \mathcal{G}^{(k_1)}_{(k_2), 11}.
\end{align}
This yields that:
\begin{align*}
&|G^{(k_1)}_{(k_2),11}|\leq |S_{\rho}| + \left |G^{(k_1)}_{(k_2),11} \right | \left ( \left |\sqrt{z}
\Lambda^{(k_1)}_{(k_2)} \right |+ \left | \sqrt{z}[\mathcal{T}^{(k_1)}_{(k_2)}]^N \right|+\left |\sqrt{z}[\Upsilon^{(k_1)}_{(k_2)}]^N \right | \right )|\sqrt{z} S_{\rho}|
\\
&|G^{(k_1)}_{(k_2),11}|\leq \left|\frac{S_{\rho}}{1-z S_{\rho}(2S_{\rho} +1)}\right| + \left |G^{(k_1)}_{(k_2),11} \right | \left ( \left |\sqrt{z} \tilde{\Lambda}^{(k_1)}_{(k_2)} \right |+ \left |\sqrt{z} [T^{(k_1)}_{(k_2)}]^N \right |+ \left |\sqrt{z}[\Upsilon^{(k_1)}_{(k_2)}]^N \right | \right ) \times \\
&\left | \frac{ \sqrt{z} S_{\rho}}{1-z S_{\rho}(2S_{\rho} +1)} \right |,
\end{align*}
and using \eqref{S-rho} we see that $1-z S_{\rho}(2S_{\rho} +1) = 2 -zS_{\rho}^2$ and thus $\frac{S_{\rho}}{1-z S_{\rho}(2S_{\rho} +1)}= \frac{S_{\rho}}{2-zS_{\rho}^2}$.

We use the bounds $|{z} S_{\rho}^2 - 2|\geq c_1$ and $|\sqrt{z} S_{\rho}| \leq C_2$, valid in our domain, and let $C = \max \{ C_2 , C_2 / c_1 \}  .$

We then have that:
\begin{equation*}
|\sqrt{z}G^{(k_1)}_{(k_2),11}|
\leq C \left [ 1+ \left |\sqrt{z} G^{(k_1)}_{(k_2),11} \right | \left (|\sqrt{z}| \min\ \left \{|\Lambda^{(k_1)}_{(k_2)}|, |\tilde\Lambda^{(k_1)}_{(k_2)}| \right \} + \left |\sqrt{z}[T^{(k_1)}_{(k_2)}]^N \right |+ \left |\sqrt{z}[\Upsilon^{(k_1)}_{(k_2)}]^N \right | \right ) \right ]
\end{equation*}
and taking power $q$, expectation, and using Jensen, Cauchy-Schwarz and \eqref{e:minor} we get at $\eta_{i+1}$:
\begin{multline} \label{3.5}
\mathbb{E}|\sqrt{z}G^{( k_1)}_{( k_2), 11}|^q
\leq C^q \Bigg [ 1+ \sqrt{\mathbb{E}|\sqrt{z}G^{( k_1)}_{( k_2), 11}|^{2q}} \sqrt{\mathbb{E}|\sqrt{z}\lambda^{( k_1)}_{( k_2)}|^{2q}}
\\
+ \frac{|\sqrt{z}|^q (|k_1-k_2|+1)^q}{(N\eta_{i+1})^q} \mathbb{E}|\sqrt{z}G^{( k_1)}_{( k_2), 11}|^{q} +\sqrt{\mathbb{E}|\sqrt{z}G^{( k_1)}_{( k_2), 11}|^{2q}} \sqrt{\mathbb{E}|\sqrt{z}[\Upsilon^{( k_1 )}_{( k_2)}]^N|^{2q}} \Bigg ]
\end{multline}
Using the above, Lemma \ref{l:Upsilonbound} and a calculation similar to \eqref{e:calc} we obtain again at $\eta_{i+1}$:
\begin{align*}\label{e:bootstrap}
 &\mathbb{E}|\sqrt{z}G^{(k_1)}_{(k_2), 11}|^q \leq C^{cq} \Bigg [ 1+ \sqrt{\mathbb{E}|\sqrt{z}G^{( k_1)}_{( k_2), 11}|^{2q}} \sqrt{\mathbb{E}|\sqrt{z}\lambda^{( k_1)}_{( k_2)}|^{2q}} \\
 &+
  \frac{|\sqrt{z}|^q (|k_1-k_2|+1)^q}{(N\eta_{i+1})^q} \mathbb{E}|\sqrt{z}G^{( k_1)}_{( k_2), 11}|^{q} \\
  &+ (Cq)^{cq} \sqrt{\mathbb{E}|\sqrt{z}G^{( k_1)}_{( k_2), 11}|^{2q}} \frac{|\sqrt{z}|^{q/4}}{(N\eta_{i+1})^{q/2}} \sqrt{1 + \E |\sqrt z \lambda^{(k_1)}_{(k_2)}| ^{q} +  \E|\sqrt z \mathcal{G}^{(k_1+1)}_{(k_2), 11}|^{q}} \\
  &+  \sqrt{\E|\sqrt{z}G^{(k_1)}_{(k_2), 11}|^{2q}\frac{\E|\sqrt{z} \mathcal{G}^{(k_1+1)}_{(k_2), 11}|^{2q}}{N^q}} \Bigg ] \numberthis
\end{align*}
We use \eqref{e:s} to bound the terms $\mathbb{E}|\sqrt{z}G^{(k_1)}_{(k_2), 11}|^{2q}$ and $\E| \sqrt{z} \mathcal{G}^{(k_1+1)}_{(k_2), 11}|^{2q}$ in the above inequality, noting that $|k_1 - k_2| \leq L_{i+1} \leq L_i$. To use \eqref{e:s} we need $2q \leq  \left (\frac{N \eta_{i+1}}{|\sqrt{E}|} \right )^{1/4}$, which gives us $q \leq \left(\frac{N \eta_i}{16|\sqrt{E}|} \right )^{1/4} $, which is what we need. Then using \eqref{e:s} on equation \eqref{e:bb} at $\eta_{i+1}$ and recalling that $|k_1-k_2| \leq C_1q$ we obtain:
\begin{align*}\label{e:lambdabootstrap}
\E | \sqrt{z} \lambda^{(k_1)}_{(k_2)} |^q
&\leq  (Cq)^{cq} \frac{|z|^{q/2}}{(N\eta_{i+1})^q} (16C_0)^{q/2} + (Cq)^{cq} \frac{|z|^{q/3}}{(N\eta_{i+1})^{2q/3}} (16C_0)^{2q/3} \\
&+ (Cq)^{cq} \frac{|z|^{q/4}}{(N\eta_{i+1})^{q/2}} (16C_0)^{3q/4} + (Cq)^{cq} \frac{(16C_0)^q}{N^{q/2}} . \numberthis
\end{align*}

\pagebreak

Since we will choose a large enough $M>0$ such that $\frac{N\eta_{i+1}}{|\sqrt{z}|} \geq M$ to work the argument, we can assume that $\frac{N\eta_{i+1}}{|\sqrt{z}|} > 1,$ so that \eqref{e:lambdabootstrap} becomes:
\begin{equation} \label{e:lambda-boot}
    \E | \sqrt{z} \lambda^{(k_1)}_{(k_2)} |^q \leq  (Cq)^{cq} \frac{|z|^{q/4}}{(N\eta_{i+1})^{q/2}} (16C_0)^{3q/4} + (Cq)^{cq} \frac{(16C_0)^q}{N^{q/2}} . 
\end{equation}

Substituting this into \eqref{e:bootstrap}, we obtain that at $\eta_{i+1}$:
\begin{align*}
&\E |\sqrt{z}G^{(k_1)}_{(k_2), 11}|^q \leq C^{cq} \Bigg \{ 1 + (16C_0)^q (Cq)^{cq} \left [ \frac{|\sqrt{z}|^{q/2}} {(N\eta_{i+1})^{q/2}}  (16C_0 )^{q/4} + \frac{(16C_0)^{q}}{N^{q/2}} \right ] +
\\
 &+
 (16C_0)^q  (Cq)^{cq} \frac{|\sqrt{z}|^{q/2}}{(N\eta_{i+1})^{q/2}} \sqrt{1 + \frac{|\sqrt{z}|^{q/2}}{(N\eta_{i+1})^{q/2}} (16C_0)^{q/4} + \frac{(16C_0)^{q}}{N^{q/2}}+ (16C_0)^{q}} + \frac{(16C_0)^{2q}}{N^{q/2}}
\Bigg \} \\
&\leq C^{cq} \left [ 1 + K^q \left( \frac{ |\sqrt{z}| } {N\eta_{i+1}} \right)^{q/4}  \right ] , \end{align*}
where we used the bound for $q$. Here we have a constant $K>0$ depending on $C_0$ and $C$. We can choose $C_0 > 2C^{c}$ and $ \frac{ N \eta_{i+1}}{|\sqrt{z}|} > M > K^{4}$, so that $ K^q \left( \frac{|\sqrt{z}|} {N\eta_{i+1}} \right)^{q/4} <1$ and therefore $\E|\sqrt{z} G^{(k_1)}_{(k_2),11}(\eta_{i+1})]^q \leq C_0^q$ as required. 

We notice that all the steps are identical for $\mathcal{G}^{(k_1+1)}_{(k_2), 11}$ using  \eqref{e:bootstrapcalG} and exactly one row gets stripped as well as exactly one column so that $|k_1+1 -(k_2+1)| = |k_1 - k_2|\leq L_{i+1}.$
\end{proof}
Combining \ref{l:Gkk} with \eqref{e:bb}, we can deduce a first bound for $\Lambda$, which will be optimized in the next section.
\begin{cor}[Weak local law] Let $z \in Z_{E,\eta}$ such that $\frac{N\eta}{|\sqrt{z}|} \geq M,$ for some suitably large constant $M$. Then, there exist $C,C_1>0$ such that:
\[
\E | \Lambda |^{2q} \leq \frac{(Cq)^{cq}}{(N\eta)^{q/2}},
\]
for any $0\leq q \leq C_1 \left (\frac{N\eta}{|\sqrt{z}|} \right )^{1/4}.$
\end{cor}
\end{section}
\begin{section}{Optimal bound for the Stieltjes transform} \label{Expand}

In this section we prove Theorem \ref{t:sti}. We will use the matrix expansion algorithm from \cite{cacciapuoti2015bounds}, which carries over directly as it is based entirely on linear algebra of resolvents. We will make a note of the important modifications. We note, importantly, that as we expand resolvent entries, we will be removing columns of $X_N$ and we never need to remove rows. The expansion algorithm yields results in terms of high moments of the following quantities:
\begin{equation}\label{e:quantities} | \sqrt{z} G_{kk}^{(\J)} | , \left | \frac{1}{\sqrt{z}G_{kk}^{(\J)}} \right | ,\left| (\mathbb{I}-\E_k) \frac{1}{\sqrt{z} G_{kk}^{(\J)}} \right |, |\sqrt{z} G_{kl}^{(\J)}|,  
\end{equation}
where $\J$ is any subset of $\{1,..,N\}$ and $G^{(\J)}$ means that we remove all columns of the scaled matrix $X$ that belong on that subset $\J$ and then take the resolvent of $K_{XX}$. Analogously, $G_{(\J)}$ means that we remove all rows of the scaled matrix $X$ that belong on that subset $\J$ and then take the resolvent of $K_{XX}$.

To obtain optimal bounds on $\Lambda$, we will use the precise inequality in $\eqref{e:upsilonbound2}.$

We begin this section by estimating high moments of the quantities in \eqref{e:quantities}. We will use \eqref{e:upsilonbound1} to obtain bounds on $\E |\sqrt{z} G_{(j_2), kl}^{(j_1)}|$ as well as $|\E\sqrt{z} \mathcal{G}_{(j_2), kl}^{(j_1)}|$.

For convenience of notation, we introduce the control parameter:
\begin{equation}\label{d:controlparam} \mathcal{E}_q := \frac{1}{N^{q} |z|^{q/2}} + \max \left \{\frac{[\mathrm{Im} (|z| S_{\rho} )]^{q} + \E |z \Lambda|^q}{(N\eta)^q} , \frac{|z|^{q}}{(N\eta)^{2q}} \right \}  .
\end{equation}
We now show how to estimate the last quantity in \eqref{e:quantities}, using the following formulas (see for example (2.20) of \cite{pillai2014universality} ) (valid also for any $j_1, j_2$, with $G=G^{(j_1)}_{(j_2)}$ or $\mathcal{G}=\mathcal{G}^{(j_1)}_{(j_2)}$ with $k, l > \max \{ j_1 , j_2 \}$):
\begin{equation} \label{e:form-off-diag}
\begin{split}
&G_{kl} = 	z G_{ll} G_{kk}^{(\{l\})} \left [ (\mathbf{x}^{l})^* \mathcal{G}^{(\{k, l\}) }\mathbf{x}^{k} \right ] =: \sqrt{z} G_{ll} \sqrt{z} G_{kk}^{(\{l\})} K_{kl}
\\
&\mathcal{G}_{kl} = z \mathcal{G}_{ll} \mathcal{G}_{(\{l\}),kk} \left [ \mathbf{x}_{l} G_{(\{k, l\}) }(\mathbf{x}_{k})^* \right ] =: \sqrt{z}\mathcal{G}_{ll} \sqrt{z} \mathcal{G}_{(\{l\}), kk} \mathcal{K}_{kl} .
\end{split}
\end{equation}
We can define $K^{(j_1)}_{(j_2), kl}, \mathcal{K}^{(j_1)}_{(j_2), kl}$ analogously. The following lemma provides the necessary bound on $\E |  K_{kl} |^{2q}$.
\begin{lem} \label{off-diag}
Assume \eqref{4moment} and \eqref{truncation} for the entries of the matrix $X_N$ as before and let $z=E+i\eta$. Then there exist constants $c,C,M_1,M_2>0$ such that
\begin{equation}\label{e:Kkl}
|z|^q \max\{\E |  K_{kl} |^{2q}, \E |  \mathcal{K}_{kl} |^{2q}\} \leq (Cq)^{cq} \mathcal{E}_q \end{equation}
for $E,\eta \in Z_{E,\eta}$, $N>M_1$ , $\frac{N\eta}{|\sqrt{z}|} > M_2$, $k\neq l\in \{1,...,N\}$, $q \in \mathbb{N}$ with $q\leq c \left ( \frac{N\eta}{|\sqrt{z}|} \right )^{1/4}.$ Assuming $|j_1 - j_2| \leq Cq$ for some constant $C$, the same inequality holds for $K^{(j_1)}_{(j_2), kl}, \mathcal{K}^{(j_1)}_{(j_2), kl}.$ with $k,l>\max\{j_1,j_2\}.$
\end{lem}
\begin{proof}
The following argument is identical for $K^{(j_1)}_{(j_2), kl}, \mathcal{K}^{(j_1)}_{(j_2), kl},$ so we work with $K_{kl}$. By the definition of  $K_{kl}$ and using the notation $\epsilon_{k1}, \epsilon_{k_2}$ for $\epsilon_1$ and $\epsilon_2$ as in \eqref{d:epsilonk12} we get that:
\begin{equation} \label{e:Kklbb}
|z|^q \E | K_{kl} |^{2q} \leq \frac{(Cq)^{cq}}{N^{2q}} \left( \E |\epsilon_{k2}|^{2q} + \E \sum_j|\mathcal{G}^{(kl)}_{jj}x_{kj}x_{lj}|^{2q}\right)
\leq
 \frac{(Cq)^{cq}}{(N\eta)^q} |z|^{q/2}
\end{equation}
where $\E |\epsilon_{k2}|$ is bounded using \eqref{e:epsilon2bound}, \eqref{e:Q0bound} and $\E \sum \limits_j |\mathcal{G}^{(kl)}_{jj}x_{kj}x_{lj}|^{2q}$ is bounded by Rosenthal's inequality like $\E|\epsilon_{k1}|^{2q}$ in \eqref{e:Rosenthal}. We also used Lemma \ref{l:Gkk} to bound $\E |\mathcal{G}_{kk}|^{2q}$. In summary, we used the bound \eqref{e:upsilonbound1} to bound this quantity.

We now improve this to \eqref{e:Kkl}. We use the bound \eqref{e:upsilonbound2}, after we modify it using Lemma \ref{l:Gkk} and the fact that: $$ \frac{2}{(N\eta)^qN^{q/2}} \leq \frac{1}{(N\eta)^{2q}}+ \frac{1}{N^{q}},$$ to obtain:
\begin{equation}\label{e:upsilonboundoptimal} \E|\Upsilon|^{2q} \leq \frac{(Cq)^{cq}}{(N\eta)^{2q}}\E|\mathrm{Im}\Tr \mathcal{G}|^q   + (Cq)^{cq}\left(\frac{1}{N^q}+\frac{1}{(N\eta)^{2q}}\right)\end{equation}
and using this, we can improve the bound on \eqref{e:Kklbb}, which yields \eqref{e:Kkl}.
\end{proof}
This gives the desired bound:
\begin{equation} \label{off-diag1}
    \E |\sqrt{z} G_{12}^{(\J)}|^{2q} \leq (Cq)^{cq} \mathcal{E} _{q},
\end{equation}
where $|\mathbb{J} | \leq Cq.$
\begin{lem} \label{reciprocal}
Assume \eqref{4moment} and \eqref{truncation} for the entries of $X_N$ as before and let $z=E+i\eta \in Z_{E,\eta}$. Then, there exist constants $c,C,M>0$ such that:
$$ \E \frac{1}{| \sqrt{z} G_{11}^{(\mathbb{J})}|^{2q } } \leq C^q ,$$
for $ \frac{N\eta}{\sqrt{|z|}} > M$, $q\leq \left (\frac{N\eta}{|\sqrt{z}|} \right )^{1/4} $, with $|\mathbb{J}| \leq cq.$
\end{lem}

\begin{proof}
We can take $|\mathbb{J}|=0$ as the argument is similar in the general case. We use \eqref{G2} to get that:
\begin{multline*} \E \frac{1}{| \sqrt{z} G_{11}|^{2q } } = \E \left |\sqrt{z} (1+ (\mathbf{x}^1)^*\mathcal{G}^{(1)}\mathbf{x}^1 ) \right |^{2q} \leq C^q + C^q|z|^q \E \left | (\mathbf{x}^1)^*\mathcal{G}^{(1)}\mathbf{x}^1 \right |^{2q}
\\
\leq C^q \left (1+ |z|^q\E \left |(\mathbf{x}^1)^*\mathcal{G}^{(1)}\mathbf{x}^1  -\E_{\xx^1} (\mathbf{x}^1)^*\mathcal{G}^{(1)}\mathbf{x}^1 \right |^{2q} +\E \left | \E_{\xx^1} \sqrt{z}(\mathbf{x}^1)^*\mathcal{G}^{(1)}\mathbf{x}^1 \right |^{2q} \right ) .
\end{multline*}
The second term on the RHS is small by Lemma \ref{l:Upsilonbound}. For the third term, we find that:
\begin{equation}
\E \left | \E_{\xx^1} \sqrt{z}(\mathbf{x}^1)^*\mathcal{G}^{(1)}\mathbf{x}^1 \right |^{2q} = \E \left |\frac{1}{N} \sqrt{z} \mathrm{Tr} (\mathcal{G}^{(1)}) \right|^{2q} = \E \left | \frac{1}{N} \sqrt{z} \left ( \frac{1}{z} + \mathrm{Tr} (G^{(1)}) \right) \right |^{2q} \leq C^q ,
\end{equation}
where we used Lemma \ref{l:Gkk} and that $|S_N^{(1)} - S_{N}|\leq \frac{1}{N\eta} $ as in \eqref{e:minor}, as well as the bound $ \frac{N\eta}{\sqrt{|z|}} > M.$
\end{proof}
To estimate the third quantity in \eqref{e:quantities}, we find by \eqref{e:upsilonboundoptimal} that:
\begin{equation}\label{e:oneover}\E \left| (\mathbb{I}-\E_{\xx^k}) \frac{1}{\sqrt{z} G_{kk}}\right|^{2q} = \E \left|-\sqrt{z} \Upsilon^{(\{k\})} \right|^{2q} \leq (Cq)^{cq} \mathcal{E}_q,
\end{equation}
where we used \eqref{G2} and the precise bound \eqref{e:upsilonbound2} along with Lemma \ref{l:Gkk}.

\pagebreak

Lastly, we also need a bound on $\E \left | \frac{1}{\E_{\xx^1} \frac{1}{\sqrt{z} G_{11}}} \right |^q$ which we obtain in the following Lemma:

\begin{lem} \label{lem}
Let $E,\eta \in Z_{E,\eta}$, where $z =E+i\eta$. There exist constants $c,C,M>0$ such that:
$$ \E \left | \frac{1}{\E_{\xx^1} \frac{1}{\sqrt{z} G_{11}}} \right |^q \leq C^q ,$$
for $N \eta \geq |\sqrt{z}| M $ and for $q \in \mathbb{N}$ with $q\leq c \left (\frac{N\eta}{|\sqrt{z}|} \right )^{1/4} .$
\end{lem}

\begin{proof}
The proof is similar to Lemma 5.1 in \cite{cacciapuoti2015bounds}.
We define:
$$\widetilde{G_{11}}=\frac{1}{\E_{\xx^1} \frac{1}{G_{11}}}= - \frac{1}{ z \left (1+\frac{\Tr\mathcal{G}^{(\{1\})}}{N} \right )}. $$
We calculate that: 
$$ \left| \frac{d}{d\eta} \log \widetilde{G_{11}}(E+i\eta) \right|= \left| \frac{d}{d\eta} \log \left(\frac{1}{z} \right) + \frac{d}{d\eta} \log \left ( \frac{1}{1+ \Tr \frac{\mathcal{G}^{(\{1\})}}{N}} \right ) \right|
=
 \left| - \frac{i}{z} - \frac{\frac{d}{d\eta}\Tr\mathcal{G}^{(\{1\}) } }{N+\Tr\mathcal{G}^{(\{1\}) }} \right| .$$
We show that $ \left | \frac{d}{d\eta} \Tr\mathcal{G}^{(\{1\}) } \right | \leq \frac{\mathrm{Im} \Tr\mathcal{G}^{(\{1\}) }}{\eta} $ as follows:
\begin{align*}
\frac{d}{d\eta} \Tr \mathcal{G}^{(1)} &= \sum \limits_{k=1}^N \frac{d}{d\eta} \mathcal{G}^{(\{1\})}_{kk}(\theta) = \sum \limits_{k=1}^N i  ((\mathcal{G}^{(\{1\})})^{2})_{kk}= \sum \limits_{k=1}^N i \langle e_k, (\mathcal{G}^{(\{1\})})^2 e_k \rangle
\\
 &\Rightarrow  \left | \frac{d}{d\eta} \mathcal{G}^{(\{1\})} \right | \leq \sum \limits_{k=1}^N \| (\mathcal{G}^{(\{1\})})^* e_k \| \ \| \mathcal{G}^{(\{1\})} e_k \| \leq \sum \limits_{k=1}^N ( (\mathcal{G}^{(\{1\})})^*\mathcal{G}^{(\{1\})} )_{kk} 
 \\
 &= \sum \limits_{k=1}^N \frac{\mathrm{Im} (\mathcal{G}^{(\{1\})})_{kk}}{\eta} = \frac{\mathrm{Im} \Tr \mathcal{G}^{(\{1\})}}{\eta} . \numberthis
 \end{align*}
We conclude that
\begin{equation}  \left | \frac{d}{d\eta} \log \widetilde{G_{11}} \right | \leq \frac{1}{|z|} + \frac{\mathrm{Im} \Tr\mathcal{G}^{(\{1\})}}{\eta |N+\Tr\mathcal{G}^{(\{1\})}|} \leq \frac{2}{\eta} ,
\end{equation}
yielding that:
\begin{equation} \left |\log \widetilde{G_{11}}(E+i\eta) - \log \widetilde{G_{11}}(E+i\eta/s) \right |= \left | \int_{\eta/s}^{\eta} \frac{d}{d\nu} \log \widetilde{G_{11}}(E+i\nu) d\nu \right | \leq \int_{\eta/s}^{\eta} \frac{2}{\nu} d\nu = \log s^2 \end{equation}
and thus $$| \widetilde{G_{11}}(E+i\eta) | \leq s^2 | \widetilde{G_{11}}(E+i\eta/s ) | .$$
The proof now proceeds with induction on $\eta$ just like in the proof of Lemma \ref{l:Gkk} using the identity
\begin{equation}\label{e:tildeGkk}\sqrt{z}\widetilde{G_{11}} = \sqrt{z} G_{11} + \sqrt{z} G_{11} \sqrt{z} \widetilde{G_{11}} (\mathbb{I}-\E_{\xx^1} ) (\sqrt z G_{11})^{-1} .
\end{equation}
as well as \eqref{e:oneover} and the results of Lemma \ref{l:Gkk}. Specifically, we notice that:
\[
\E |\sqrt{z} \widetilde{G_{11}}|^q \leq C^q + (Cq)^{cq} \E \left ( |\sqrt{z} \widetilde{G_{11}}|^{3q} \right )^{1/3} \frac{|z|^{q/2}}{(N\eta)^q} ,
\]
after which the induction is straight-forward just like Lemma \ref{l:Gkk}.
\end{proof}
Lastly, we use the matrix expansion algorithm to take advantage of the fluctuations. This technique was firstly introduced in the papers \cite{erdos2010universality} and \cite{erdHos2012rigidity} and was later simplified and improved in the papers \cite{erdHos2013averaging} and \cite{erdHos2013local}. We will follow the analysis in \cite{cacciapuoti2015bounds}.

Hence the following proposition, analogous to Lemma 4.1 of \cite{cacciapuoti2015bounds}:
\begin{prop} \label{p:optimalbound}
 Let $\mathcal{E}_q$ be the control parameter as in \eqref{d:controlparam}. There exist constants $C, M, c_0 >0$ such that:
\begin{equation} \label{4.2}
\E \left | \frac{1}{N} \sum \limits_k \sqrt{z} \Upsilon^{(\{k\})}  \sqrt{z} G_{kk} \right |^{2q} \leq (Cq)^{cq^2} \mathcal{E}_{4q}^{1/2} ,
\end{equation}
for $1 \leq q \leq c_0 \left (\frac{N\eta}{|\sqrt{z}|} \right )^{1/8} $ , $\frac{N\eta}{\sqrt{|z|}}\geq M $ , $K>0$ , $z = E + i \eta \in Z_{E, \eta} $.
\end{prop}
\begin{proof}
To match notation in \cite{cacciapuoti2015bounds}, we introduce $W_k= \sqrt{z} \Upsilon_{k} \sqrt{z} G_{kk}$ and we split:
\beq
\label{e:eW2q} \frac{1}{N} \sum \limits_k  W_k = \frac{1}{N} \sum \limits_k (\mathbb{I}-\E_k)W_k + \frac{1}{N}\sum \limits_k \E_k W_k . \eeq
By H\"older's inequality,
\beq \label{W2q} \E \left | \frac{1}{N} \sum \limits_k  W_k \right |^{2q} \leq C^q \E \left | \frac{1}{N} \sum \limits_k  ( \mathbb{I} - \E_k) W_k \right |^{2q} + C^q \E |\E_1 W_1|^{2q} .
\eeq
To bound the second term in \eqref{W2q} above, using that $\sqrt{\theta}\Upsilon^{(\{k\})}=-(\mathbb{I}-\E_k)\frac{1}{\sqrt{z}G_{kk}}$, we obtain
\begin{equation} \label{4.3}
\E_k W_k = \frac{\E_k [\sqrt{z} G_{kk} (\sqrt{z}\Upsilon^{(\{k\})})^2]}{\left (\E_k\frac{1}{\sqrt{z} G_{kk}}\right )}.
\end{equation}
and applying Lemma \ref{lem} to \eqref{4.3}, we get that:
\begin{align*}
\E | \E_1 W_1 |^{2q} &\leq (\E |\sqrt{z} G_{11}|^{8q})^{\frac{1}{4}} \left ( \E \left | \frac{1}{\E_1 \frac{1}{\sqrt{z} G_{11}}} \right |^{8q} \right )^{\frac{1}{4}} ( \E | \sqrt{z} \Upsilon^{(\{1\})} |^{8q} )^{\frac{1}{2}} \\
&\leq (Cq)^{cq} \left ( \mathcal{E}_{4q} \right )^{1/2} ,
\end{align*}
which is what we want.

In order to handle the first term of \eqref{W2q}, we use the matrix expansion algorithm as in Section 5.2 of \cite{cacciapuoti2015bounds}. We notice that equations (5.7), (5.8), and (5.9) are the basis of the expansion algorithm, and they are equivalent to the following (see e.g. (2.18) in \cite{pillai2014universality}):
\beq \label{e:algorithmnew}\begin{split}
&\sqrt z G_{ij}^{(\T)} = \sqrt z G_{ij}^{(\T k)} + \frac{\sqrt z G_{ik}^{(\T)}\sqrt z G_{kj}^{(\T)}}{\sqrt z G_{kk}^{(\T)}} \text{ for } i, j, k \notin \T \text{ and } i,j \neq k,
\\
&\frac{1}{\sqrt z G_{ii}^{(\T)}} = \frac{1}{\sqrt z G_{ii}^{(\T k)}} - \frac{\sqrt z G_{ik}^{(\T)}\sqrt z G_{ki}^{(\T)}}{\sqrt z G_{ii}^{(\T)}\sqrt z G_{ii}^{(\T k)}\sqrt z G_{kk}^{(\T)}} \text{ for } i, k \notin \T \text{ and } i \neq k
\end{split}
\eeq
Using the above equation \eqref{e:algorithmnew}, we see that in our case the steps of the expansion algorithm (5.13), (5.14), (5.15) in \cite{cacciapuoti2015bounds} are the same except that each resolvent entry is multiplied by a factor of $\sqrt z$. Using our definition of $W$, equation (5.6) in \cite{cacciapuoti2015bounds} becomes analogous to
\begin{equation} \label{formexpand}
( \mathbb{I}-\E_{k_s})W_{k_s} = ( \mathbb{I}-\E_{k_s}) \left [ ( \mathbb{I}-\E_{k_s}) \frac{1}{\sqrt{z}G_{k_s k_s}} \right ] \sqrt{z} G_{k_s k_s} , s=1,...,2q ,
\end{equation}
so the initial terms of the algorithm are $A^{r}:=\sqrt z G_{k_r k_r}$ and $B^{r} :=\frac{1}{\sqrt{z}G_{k_r k_r}}$ are the same as (5.16), (5.17) of \cite{cacciapuoti2015bounds} except that each resolvent entry is multiplied by a $\sqrt z$.
Then (5.18), (5.19), and (5.20) of \cite{cacciapuoti2015bounds} carry over directly as well as properties (1) through (5) of relevant strings. We then obtain the desired result
\[
\E \left|\frac{1}{N}\sum_k( \mathbb{I}-\E_{k})W_{k}\right|^{2q} \leq (Cq)^{cq^2} \mathcal{E}_{4q}^{1/2}
\]
using the proof of (5.32) of \cite{cacciapuoti2015bounds}. It relies on counting the types of terms that result from the expansion algorithm. Since our algorithm yields the same type and number of terms in each step, the proof in our case will be identical.
In \cite{cacciapuoti2015bounds}, we notice the use of bounds (3.9) and Lemma 5.2 in (5.44) as well as in Case 2, bounds (5.26) and (3.4) in (5.43) and (5.49). We can replace (3.9), Lemma 5.2, (5.26), and (3.4) of \cite{cacciapuoti2015bounds} by our bounds on the relevant quantities in \eqref{e:quantities} as well as our (5.4).
\end{proof}
\begin{proof}[Proof of Theorem \ref{t:sti}]
 By Proposition \ref{b-R}, in order to control $\Lambda$, we need to control high moments of $R=N^{-1} \sum \limits_{k=1}^N G_{kk}(T_k + \Upsilon^{(\{k\})})$. Taking expectation of $2q$  power we obtain:
\begin{equation}\label{e:Ypsilonpower}\E| z R|^{2q} \leq  C^q \left ( \E \left | \frac{1}{N} \sum \limits_k \sqrt{z} T_k \sqrt{z} G_{kk} \right |^{2q} + \E \left | \frac{1}{N} \sum \limits_k \sqrt{z} \Upsilon^{(\{k\})} \sqrt{z} G_{kk} \right |^{2q} \right )  .\end{equation}
For the first term by \eqref{e:minor}, we obtain
\begin{align} \label{4.1}
\E \left| \frac{1}{N} \sum \limits_{k} \sqrt{z}T_k \sqrt{z}G_{kk} \right|^{2q}\leq C^q \frac{1}{N^{2q} |z|^q} .
\end{align}
while the second term is handled in Proposition \ref{p:optimalbound}, yielding that
$$ \E|z R|^{2q} \leq (Cq)^{cq^2} \mathcal{E}_{4q}^{1/2}. $$
Here we are able to simplify the analysis in \cite{cacciapuoti2015bounds} by only using the bounds proportional to $R$ from Proposition \ref{b-R} to control $E|\Lambda|^{2q}$ in $Z_{E, \eta}$ and $E|\mathrm{Im}\Lambda|^{2q}$. Our simplifications carry over also to the Wigner case. We can assume that
$$[\mathrm{Im}( S_{\rho})]^{2q} + \E| \Lambda|^{2q} \geq \frac{1}{(N\eta)^{2q}} ,$$
(otherwise $\E | \Lambda|^{2q} \leq \frac{1}{(N\eta)^{2q}} ,$ as we want)
and in this case:
$$ \mathcal{E}_{2q} = \frac{1}{N^{2q} |z|^{q}} + \frac{\mathrm{Im} (|z| S_{\rho} )]^{2q} + \E |z \Lambda|^{2q}}{(N\eta)^{2q}} \leq \frac{\eta^q + \mathrm{Im} (|z| S_{\rho} )]^{2q} + \E |   z \Lambda|^{2q}}{(N\eta)^{2q}} ,$$
Using the bound proportional to $|R|$ from Proposition \ref{b-R}, we obtain
\begin{align*}
\E |z \Lambda|^q &\leq \frac{C^q \E| z R|^q}{|S_{\rho}+\frac{1}{2}|^q} \leq \frac{(Cq)^{cq^2}}{|S_{\rho}+\frac{1}{2}|^{q}} \left ( \frac{\eta^q + [\mathrm{Im} (|z|S_{\rho})]^{2q}}{(N\eta)^{2q}} \right)^{1/2} \\
&= \frac{(Cq)^{cq^2}}{|S_{\rho}+\frac{1}{2}|^q} \frac{|z|^q}{(N\eta)^q} \left ( \frac{\eta^q}{|z|^{2q}} + [\mathrm{Im} (S_{\rho})]^{2q} \right )^{1/2} \\
&\leq \frac{(Cq)^{cq^2} |z|^q } {(N\eta)^q} \left [ \left (\frac{\sqrt{\eta}}{|z| |S_{\rho}+\frac{1}{2}|} \right )^q+ \left ( \frac{\mathrm{Im} S_{\rho}}{|S_{\rho}+\frac{1}{2}|} \right )^q \right ].
\end{align*}
To obtain the desired bound we now note that
$\mathrm{Im} S_{\rho} \leq |S_{\rho}+\frac{1}{2}|$ and $\frac{\sqrt{\eta}}{|z| |S_{\rho}+\frac{1}{2}|} \leq C$ in our domain. The first one follows easily and for the second one we argue as follows:
$$ \frac{\sqrt{\eta}}{|z| |S_{\rho}+\frac{1}{2}|} = \frac{2\sqrt{\eta}}{\sqrt{|z|}\sqrt{|z-4|}} ,$$
and by triangle inequality either $|z|\geq 2$ or $|z-4|\geq 2$. Then in the first case, we use the bound $\sqrt{\eta} \leq \sqrt{|z-4|} $ and in the second case the bound $\sqrt{\eta} \leq \sqrt{|z|} $.

Overall, this implies that:
\begin{equation}\label{e:finalLambdabound}\mathbb{P} \left( |S_{N} - S_{\rho}|\geq \frac{K}{ N\eta} \right)\leq \frac{(N\eta)^q}{K^q}\E | \Lambda|^q \leq \frac{(Cq)^{cq^2}}{K^q} ,
\end{equation}
for $1 \leq q \leq c_0 \left (\frac{N\eta}{|\sqrt{z}|} \right )^{1/8} $ , $\frac{N\eta}{\sqrt{|z|}}\geq M $ , $K>0$ , $z = E + i \eta \in Z_{E, \eta} $ .
\end{proof}
\end{section}

\begin{section}{Rate of convergence to the Marchenko-Pastur distribution}

In this section we prove Theorem \ref{t:counting}.

Let
\begin{equation}
\mathcal{N}(I) := \#\{i \leq N\,|\; \lambda_i \in I\},
\end{equation}
denote the number of eigenvalues in the interval $I.$
We denote by $\nu$ the probability distribution of $\mathrm{Re} (x_{ij})$ and $\mathrm{Im} (x_{ij}),$ for $i,j=1,...,N.$
\begin{proof}[Proof of Theorem \ref{t:counting}]
Let $0<E \leq 4$. We will use a Pleijel argument from \cite{pleijel1963theorem}, recently used in obtaining estimates on a measure $\mu$ from estimates on its Stieltjes transform as in \cite{erdHos2016fluctuations}. We start from the following equations (equations (13) and (14) in  \cite{erdHos2016fluctuations}, following from equation (5) of \cite{pleijel1963theorem}). Here $\mu$ is any probability measure supported on an interval $[-K,K],$ for a large $K>0:$
\begin{equation}\label{e:pleijel}
\mu(-K, E) = \frac{1}{2\pi i}\int_{L(z_0)}S_{\mu}(z)dz + \frac{\eta_0}{\pi} \mathrm{Re} S_\mu(z_0) + \mathcal{O} \left (\eta_0 \mathrm{Im} S_\mu(z_0) \right )
\end{equation}
and
\begin{equation}\label{e:pleijel2}
\mu(x, x') = \frac{1}{2\pi i}\int_{\gamma(x, x')} S_{\mu}(z)dz + \mathcal{O} \left (\eta_0 \left | S_\mu(x + i\eta_0) \right | + \eta_0 \left | S_\mu(x' + i\eta_0 \right | \right ),
\end{equation}
where $S_\mu$ is the Stieltjes transform of $\mu$ and $L(z_0)$ is a contour as in Figure \ref{f:domains} (see also \cite{erdHos2016fluctuations} Fig 1A), namely connects with line segments the points $E-i\eta_0, E-iQ, -1-iQ , -1+ iQ, E+ iQ, E+ i\eta_0$ in that order with arbitrarily chosen constants $-1$ and $Q$, and $\gamma(x, x')$ is the chain connecting $x+i\eta_0$ with $x'+i\eta_0$ and $x'- i\eta_0$ with  $x-i\eta_0.$

By Markov inequality, we obtain that
\begin{equation}
\PP\left(|\rho_N(E) - P(E)| \geq \frac{C \log N}{N}\right) \leq \frac{N^q\E \left |\rho_N(E) - P(E) \right |^q}{\left(C \log N\right)^q} .
\end{equation}
Then using \eqref{e:pleijel} and taking $z_0 := E + i\eta_0$ with $\eta_0 := \frac{M\sqrt E}{N}$ with $M$ as in Theorem \ref{t:sti}, we obtain that:
\begin{multline}
 \E \left |\rho_N(E) - P(E) \right |^q = \E \bigg |\frac{1}{2\pi i}\int_{L(z_0)}\Lambda(z)dz + \frac{\eta_0}{\pi} \mathrm{Re} \Lambda(z_0) + \mathcal{O}\bigg ( \eta_0  ( \mathrm{Im} S_N(z_0)+ \mathrm{Im} S_{\rho}(z_0)  ) \bigg ) \bigg |^q
\\
\leq  C^q\left(\E\left|\int_{L(z_0)}\Lambda(z) dz\right|^q + \mathcal{O} \bigg (\eta_0^q \E|\Lambda(z_0)|^q+\eta_0^q \mathrm{Im} S_{\rho}(z_0)^q \bigg )\right),
\end{multline}
noting that the constant in the $\mathcal{O}$ comes from the Pleijel formula and is uniform in the matrix randomness. We study the above expression one term at a time.
For $E \leq 4$ we can bound the second term as follows
\begin{equation}\label{e:etalambda}
 \eta_0^q \E|\Lambda(z_0)|^q \leq  \eta_0^q \frac{(Cq)^{cq^2}}{(N\eta_0 )^q} \leq \frac{(Cq)^{cq^2}}{N^q}.
\end{equation}
The third term is bounded using the above inequality \eqref{e:etalambda} on $\Lambda$ as well as
\begin{equation}\label{e:etaoverE}
\eta_0 \mathrm{Im} S_{\rho} \leq \frac{C\eta_0}{\sqrt E} = \frac{CM}{N}.
\end{equation}
Now for the integral, we note that it suffices to study the part of the contour where $\mathrm{Im} (z) > 0$ since $\Lambda(\bar z) = \overline{\Lambda( z)}$. Thus we obtain:
\begin{align*}
\E\left|\int_{L(z_0)}\Lambda(z) dz\right|^q &\leq C^q \bigg (\E \left|\int_0^{\eta_0} \Lambda(-1 + iy) dy \right|^q \\
&+\E \left|\int_{\eta_0}^Q \Lambda(-1 + iy) - \Lambda(E + iy)dy \right|^q+  \E \left|\int_{-1}^E  \Lambda(x + iQ) dx \right|^q \bigg )
\end{align*}
Since all eigenvalues are positive we bound $|\Lambda|$ for $-1<0$ by $|\Lambda(-1+i\eta)| \leq 2$ which yields:
\begin{equation}
\left|\int_0^{\eta_0} \Lambda(-1 + iy) dy \right|^q \leq \left(\int_0^{\eta_0} |\Lambda(-1 + iy)| dy \right)^q \leq C^q \eta_0^q.
\end{equation}
Next we note that:
\begin{equation}
\E \left|\int_{-1}^E  \Lambda(x + iQ) dx \right|^q \leq \frac{(Cq)^{cq^2}}{(NQ)^q}
\end{equation}
Now we can bound the expected value of the integrals $ \E \left ( \int_{\eta_0}^Q | \Lambda( E+iy) |dy \right )^q $ and $ \E \left ( \int_{\eta_0}^Q | \Lambda( -1 +iy) |dy \right )^q $ for $E \leq 4$, noting that the argument is identical at $E$ and $-1$,
\begin{align*}
&\E \left ( \int_{\eta_0}^Q | \Lambda(E +iy) |dy \right )^q = \E \int_{\eta_0}^Q | \Lambda(E +iy_1) |dy_1 \int_{\eta_0}^Q | \Lambda(E+iy_2) |dy_2 \cdots \int_{\eta_0}^Q | \Lambda(E+iy_q) |dy_q \\
&= \E \int_{\eta_0}^Q \cdots \int_{\eta_0}^Q \prod_{j=1}^q | \Lambda(E +iy_j) | \prod_{j=1}^q dy_j = \int_{\eta_0}^Q \cdots \int_{\eta_0}^Q \E \prod_{j=1}^q | \Lambda(E +iy_j) | \prod_{j=1}^q dy_j \\
&\leq
 \int_{\eta_0}^Q \cdots \int_{\eta_0}^Q \prod_{j=1}^q \left (\E | \Lambda(E+iy_j) |^q \right )^{\frac{1}{q}}  \prod_{j=1}^q dy_j
\leq
 \frac{1}{N^q}  \int_{\eta_0}^Q \cdots \int_{\eta_0}^Q \prod_{j=1}^q \frac{(Cq)^{cq}}{y_j} \prod_{j=1}^q dy_j \\
&= \frac{(Cq)^{cq^2}}{N^q} \left ( \int_{\eta_0}^Q \frac{1}{y} dy \right )^q \leq (Cq)^{cq^2}  \frac{(\log N)^q}{N^q} ,
\end{align*}
where we can apply \eqref{e:finalLambdabound} inside the integral because our estimates on $\Lambda$ are uniform on compact sets.

To prove the other part of \eqref{e:counting} near the "hard" edge, we use the  \eqref{e:pleijel2} and study the interval $[-E, E]$, noting that $ \rho_N(E) = \mathcal{N}([-E, E])/N $  and $P(-E,E) = P(E)$. Similarly to the above, we have that:
\[
\rho_N(-E,E) - \rho(-E,E) =  \frac{1}{2\pi i } \int_{\gamma(-E,E)} \Lambda(z) dz + \mathcal{O} \left ( \eta_0 \Lambda(E+i\eta_0) \right ) + \mathcal{O} \left ( \eta_0 S_{\rho}(E+i\eta_0) \right ) .
\]
We now take expectation and power $q$ and use H\"olders inequality. The corresponding integral can be bounded similarly to above:
\begin{align*}
&\E \left | \int_{\gamma(-E,E)} \Lambda(z)dz \right |^q = 
\mathbb{E}\left|\int_{-E}^E \Lambda\left(x+i \eta_0\right)-\Lambda\left(x-i \eta_0\right) d x\right|^q \\
&=\mathbb{E}\left|\int_{-E}^E 2 \mathrm{Im} \Lambda \left(x+i \eta_0\right)\right|^q
\leq \frac{(C q)^{c q^2} E^q}{\left(N \eta_0\right)^q} \leq \frac{(C q)^{c q^2}(\sqrt{E})^q}{M^q}
\end{align*}
and, similarly to \eqref{e:etaoverE}:
\begin{equation}\label{e:deltabound}
\eta_0 S_{\rho}(E) \leq \frac{\eta_0}{\sqrt{E}}=\frac{M}{N} ,
\end{equation}
which together with \eqref{e:etalambda} yields the $\sqrt{E}$ - bound of \eqref{e:counting} for $E<4$.

To establish the \eqref{e:counting} for $E > 4$, we use \eqref{e:counting} for $E = 4$ to establish bounds on the number of eigenvalues outside of the spectrum. Letting $\mathcal{N}_I$ be the number of eigenvalues in an interval $I$, we see that:
\begin{equation}
 \mathcal{N}_{(4, \infty)} = N - N \rho_N(4) = N(P(4) - \rho_N(4))
\end{equation}
which by \eqref{e:counting} for $E = 4$ yields that:
\begin{equation}\label{e:outside}
\PP\left(\frac{\mathcal{N}_{(4, \infty)}}{N} > \frac{K\log N}{N}\right) \leq \frac{(Cq)^{q^2}}{K^q}
\end{equation}
and for $E > 4$,
\begin{equation}
\PP\left(|\rho_N(E) - P(E)| \geq \frac{K \log N}{N}\right) \leq \PP\left(\frac{\mathcal{N}_{(4, \infty)}}{N} > \frac{K\log N}{N}\right),
\end{equation}
thus \eqref{e:outside} gives the desired bound.
\end{proof}

\end{section}

\begin{section}{Rigidity of the eigenvalues}

The aim of this section is a proof of Theorem \ref{t:rigidity}.
\begin{proof}[Proof of Theorem \ref{t:rigidity}]
Let $i \leq \frac{N}{2}$. We will make use of the following inequalities near the "hard" edge and away from the "soft" edge:
\beq \label{est1} c \sqrt{x} \leq  P(x) \leq C \sqrt{x} , \eeq and \beq \label{est2} c  P(x)^{-1} \leq \rho(x) \leq C  P(x)^{-1} . \eeq
valid for $x\in (0,3] .$ The second inequality implies that
\beq \label{14}
c \frac{N}{i} \leq \rho(\gamma_i) \leq C  \frac{N}{i},
\eeq
for any $i \leq \frac{N}{2}.$

For $\epsilon > 0$, we have that:
\begin{align*}
&\mathbb{P} \left ( | \lambda_i - \gamma_i | \geq K \epsilon \frac{i}{N}  \right ) \\
&\leq \mathbb{P} \left ( | \lambda_i - \gamma_i | \geq K\epsilon\frac{i}{N}  \ \ \text{and} \ \ \lambda_i \leq \gamma_i \right ) + \mathbb{P} \left ( | \lambda_i - \gamma_i | \geq K\epsilon \frac{a}{N} \  \ \text{and} \ \ \lambda_i >\gamma_i \right ) \\
&= A + B \numberthis .
\end{align*}
We consider first the term $A.$ We set
$$\ell = K \varepsilon\frac{i}{N}.$$
From $ \lambda_i \leq \gamma_i $ and $| \lambda_i - \gamma_i | \geq \ell$ we find that $\lambda_i \leq \gamma_i - \ell .$ This implies that $ \rho_N( \gamma_i - \ell ) \geq \frac{i}{N} = P(\gamma_i)$. By the mean value theorem for the function $P$, there exists a point $x^* \in [\gamma_i - \ell , \gamma_i ] $ such that
$P(\gamma_i ) - P(\gamma_i - \ell ) = \rho(x^*) \ell$, yielding that:
\begin{align*}
\rho_N(\gamma_i - \ell) - P(\gamma_i - \ell) &= \rho_N(\gamma_i - \ell) - P(\gamma_i) + \rho(x^*)\ell
\geq \rho(x^*) \ell \\
&\geq \rho(\gamma_i) K \epsilon \frac{i}{N} \geq  cK \epsilon, \numberthis
\end{align*}
because $\rho$ is non-increasing, $i < N/2$, and we used \eqref{14}.
Setting $\varepsilon = \frac{\log N}{N} $ we deduce from Theorem \ref{t:counting} that:
\begin{equation} A \leq \PP \left ( |\rho_N(\gamma_i - \ell) - P(\gamma_i - \ell) | \geq  \frac{cK\log N}{N} \right ) \leq \frac{(Cq)^{cq^2}}{K^q}
\end{equation}
For $i \leq \log N$, set $\epsilon = \frac{i}{N} \geq c\sqrt{\gamma_i},$ from \eqref{est1}, to obtain the better bound:
\begin{equation}
A \leq \PP \left ( |\rho_N(\gamma_i - \ell) - P(\gamma_i - \ell) | \geq cK \sqrt{(\gamma_i - \ell)_+} \right )
\leq \frac{(Cq)^{cq^2}}{K^q},
\end{equation} 
because $ c\sqrt{(\gamma_i - \ell)_+} \leq c\sqrt{\gamma_i} \leq \frac{i}{N} \leq \frac{\log N}{N}$ and we used Theorem \ref{t:counting}. 

We now estimate the term $B.$ From the estimate \eqref{est1} near the "hard" edge we have that $P(x) \sim \sqrt{x}$, so:
$$ \gamma_i \leq C \left ( \frac{i}{N} \right )^2 ,$$
for some constant $C>0$ for all $i < N/2$. We consider the number
$$ 
y = 2C \left ( \frac{i}{N} \right )^2 
$$
and we further consider the cases that $\gamma_i + \ell \leq y $ or $ \gamma_i + \ell > y .$

In the first case, since $\lambda_i > \gamma_i $ and $| \lambda_i - \gamma_i | \geq \ell ,$ we have that  $\lambda_i > \gamma_i + \ell $ and so $ \rho_N( \gamma_i + \ell) \leq \frac{i}{N} = P(\gamma_i).$

Hence, from the mean value theorem, we find $x^* \in [\gamma_i , \gamma_i + \ell ] \subset [\gamma_i , y ] $ such that $P(\gamma_i + \ell) - P(\gamma_i ) = \rho(x^*) \ell$, yielding that:
\begin{align*}
&P(\gamma_i + \ell) - \rho_N(\gamma_i + \ell) = P(\gamma_i) - \rho_N(\gamma_i + \ell) + \rho(x^*)\ell \\
&\geq \rho(x^*) \ell = \rho(x^*)  K\epsilon \frac{i}{N} \geq \rho(y) K\epsilon \frac{i}{N} \geq  c K\epsilon, \numberthis
\end{align*}
where we used that $\rho$ is non-increasing and that $\rho(y) \geq \frac{c}{\sqrt{y}}$ near the "hard edge".

Setting $\epsilon = \frac{\log N}{N}$ and using Theorem \ref{t:counting}, we conclude that
\begin{equation} B \leq \PP \left( |P(\gamma_i + \ell) - \rho_N(\gamma_i + \ell) | \geq cK \frac{ \log N}{N}\right)\leq \frac{(Cq)^{cq^2}}{K^q},
\end{equation}
as required. To obtain rigidity at the "hard" edge ( equation \eqref{e:rigidityedge}, let $\epsilon = \frac{i}{N} $ to obtain:
\begin{align*}
B &\leq \PP \left( |P(\gamma_i + \ell) - \rho_N(\gamma_i + \ell) | \geq \frac{cKi}{N} \right) \\
&\leq \PP \left( |P(\gamma_i + \ell) - \rho_N(\gamma_i + \ell) | \geq c\sqrt{K} \sqrt{\gamma_i + \ell}  \right)
\leq \frac{(Cq)^{cq^2}}{K^{q/2}} , \numberthis
\end{align*}
where the second line follows as before because $\sqrt{\gamma_i} \leq c\frac{i}{N}$ and $\ell = \frac{K i^2}{N^2}$.

In the other case we have that $ \gamma_i + \ell > y$ so the inequality $\lambda_i > \gamma_i + \ell$ implies that $\lambda_i > y$ and therefore $\rho_N (y) \leq \frac{i}{N} = P(\gamma_i).$ Hence from the mean value theorem, there exists $x^* \in [\gamma_i, y]$ such that $P(y) - P(\gamma_i) \geq \rho(x^*) \ell$, which yields:
\begin{equation*}
P(y) - \rho_N(y) \geq P(\gamma_i) - \rho_N(y) + \rho(x^*)\ell \geq \rho(x^*) \ell = \rho(x^*)  K \epsilon \frac{i}{N} \geq \rho(y) K \epsilon \frac{i}{N} \geq  c K \epsilon,
\end{equation*}
and we can conclude \eqref{e:rigiditybulk} and \eqref{e:rigidityedge} as above. This finishes the proof of Theorem \ref{t:rigidity}.
\end{proof}
\end{section}
 \chapter{The Dyson equation method}

 In this chapter we generalize the Stieltjes transform technique to the Dyson equation technique.

 Let $H$ be a $N \times N$ Hermitian random matrix, whose spectrum we want to compute for large $N$. We recall that with the Stieltjes transform method we had to define the \textit{resolvent} matrix $G(z) = (H - zI)^{-1}$ and then the Stieltjes transform of the empirical spectral measure would be equal to the normalized trace of $G(z).$ We then had to show that the normalized trace of $G(z)$ convergences to a suitable Stieljes tranform $S_{\mu}(z)$ of a limiting measure $\mu$. Usually, $S_{\mu}(z)$ satisfies an algebraic equation which is "close" to the probabilistic one satisfied by $\langle G(z) \rangle$, which can be found using resolvent identities.
 
 In the Dyson equation technique, instead of studying the normalized trace $\langle G(z) \rangle$ of the resolvent matrix, we study the whole matrix $G(z)$ and deduce a basic probabilistic equation satisfied by it in the Gaussian case which we generalize for any distribution for the entries. If some conditions are satisfied about the initial Hermitian random matrix $H$, then we can simplify this equation to deduce a deterministic one which is "close" to the original and whose determistic solution $M(z)$ would be "close" to $G(z).$ Analogously, $M(z)$ turns out be the Stieltjes transform of a \textit{matrix-valued} measure, whose normalized trace $\langle M(z) \rangle$ can give us the liming spectral measure for the eigenvalues of $H$.

\pagebreak

 This technique is especially useful when we have correlations inside the Hermitian matrix $H$. In section \ref{s:GUE} we will apply this technique to the Gaussian Unitary Ensemble (GUE) and find out that it coincides with the Stieltjes transform technique. We then find out some differences for Wigner-type matrices and correlated Hermitian matrices in section \ref{s:W+C}.

 The technique is then studied and established theoretically in the next sections. There are two steps for this method to work, after proving the existence and uniqueness of the solution $M(z)$:
\begin{enumerate}
    \item Show that the probabilistic equation satisfied by $G(z)$ and the deterministic equation satisfied by $M(z)$ are "close" to each other. This is done by treating the first as a perturbation of the second one and showing that the error is small. This is achieved through the \textit{multivariate cumulant expansion} technique and is summarized in \ref{s:Cumulants}.

    \item Show that the deterministic equation is \textit{stable}. This means that solutions to it are "close" to solutions of small perturbations of it. This is achieved by showing the invertibility of a suitable \textit{linear stability operator}. This is summarized in section \ref{s:Stability}.
\end{enumerate}
We finally give the resulting theorems in section \ref{s:Results}. 

This analysis is based on the lecture notes \cite{erdos2019matrix}. The publications \cite{erdHos2019random} and \cite{ajanki2019stability} analyze these two steps and give the final theorems, while the existence and uniqueness of the solution for the Dyson equation can be found in \cite{helton2007operator} for the matrix case and in \cite{ajanki2019quadratic} for the vector case.
 
\begin{section}{Gaussian Hermitian random matrix models} \label{s:GUE}

The definition of the resolvent $G=G(z)$ results in the basic identity:
\beq \label{HG1}
HG = I + zG.
\eeq
We will include in this identity the information provided by the expected value of $HG$, which reminds of the formula (for Gaussian random variables):
\beq \label{gauss1}
\E [ h f(h)] = \E [h^2] \E[ f'(h) ],
\eeq
where $h$ is real centered Gaussian and $f$ a smooth function. The version of this for complex Gaussian random variables is the following:
\beq \label{gauss2}
\E[zf(z)] = \E \|z \|^2 \E \left [ \overline{\partial} f (z) \right ],
\eeq
where $\overline{\partial}$ denotes the Wirtinger derivative and $f$ is sufficiently smooth. This gets generalized for Gaussian random vectors $V \in \mathbb{R}^n$ as follows:
\beq \label{gauss3}
\E[ V f(V)] = K \E \left [ \nabla f(V) \right ],
\eeq
where $K=\E [ VV^T ]$ and $f:\mathbb{R}^n \mapsto \mathbb{R} $  is sufficiently smooth. 

The application of this formula to resolvent matrices in random matrices firstly appeared in \cite{khorunzhy1996asymptotic}. In our case, where $f$ represents the resolvent function of a complex centered Gaussian matrix, we have the following:
\begin{lem} \label{Gauss4}
    For a Hermitian Gaussian random matrix $H$ and its resolvent $G,$ it holds that:
    \beq
    \E [H G] = - \E \left [ \Tilde{\E}[\Tilde{H}G\Tilde{H}] G \right ] ,
    \eeq
    where $\Tilde{H}$ is an independent copy of $H$ and the second expectation is with respect to $\Tilde{H}$, i.e. $\Tilde{\E}[\Tilde{H}G\Tilde{H}] = {\E} \left [ {H}G {H} \middle | G \right ].$
\end{lem}
\begin{proof}
    For the proof, we apply \eqref{gauss3}, while identifying $\mathbb{C}$ with $\mathbb{R}^2$. We obtain that:
    \begin{align*}
        &\E[H_{jk}G_{ab}] \\ 
        &= \E \left [ K_{11} \partial_{\mathrm{Re}(H_{jk})} G_{ab} + K_{12} \partial_{\mathrm{Im}(H_{jk})}G_{ab} + i K_{21}\partial_{\mathrm{Re}(H_{jk})} G_{ab} + i K_{22}\partial_{\mathrm{Im}(H_{jk})} G_{ab} \right ] \\
        &= - \E \bigg [ ( K_{11} - K_{22} +iK_{12}+iK_{21} ) G_{aj}G_{kb} + ( K_{11}+K_{22}-iK_{12} +iK_{21} ) G_{ak}G_{jb} \bigg ] ,
    \end{align*}
    where $K$ is the real $2N \times 2N$ covariance matrix of the $k$-th column of $H$ which belongs in $\mathbb{C}^{N}$. 
    
    We observe that $K_{11} = K_{22} = \E [\mathrm{Re}H_{jk} ]^2 = \E [\mathrm{Im}H_{jk} ]^2$ and $$ K_{12}=K_{21} = \E [\mathrm{Re}H_{jk} \mathrm{Im}H_{jk} ] = 0.$$
    Therefore, we have that:
    \begin{equation} \label{HG2}
        \E[H_{jk}G_{ab}] = - \E \left [ \| H_{jk}\|^2 G_{ak}G_{jb} \right ] = - \E \| H_{12}\|^2 \E \left [ G_{ak}G_{jb} \right ] ,
    \end{equation}
    for any $j,k,a,b \in \{1,...,N\}.$ Calculating the desired expected value, we see that:
    \begin{align*} \label{tilde-exp}
    &\Tilde{\E}[\Tilde{H}G\Tilde{H}]_{jk} = \sum \limits_{a,b} \tilde{\E} [ \tilde{H}_{ja} G_{ab} \tilde{H}_{bk} ] = \sum \limits_{a,b} G_{ab} \tilde{\E} [ \tilde{H}_{ja} \tilde{H}_{bk}] \\
    &= \delta_{jk} \E\| H_{12} \|^2 \sum \limits_{a=1}^N G_{aa} = 
    \delta_{jk} \E\| H_{12} \|^2 \mathrm{Tr}[G], \numberthis
    \end{align*}
    and this means that:
    \begin{align*}
    &\E \left [ \Tilde{\E}[\Tilde{H}G\Tilde{H}] G \right]_{jk} =  \E\| H_{12} \|^2 \E \left [ G_{jk} \mathrm{Tr}(G) \right ] .
    \end{align*}
    On the other hand, by \eqref{HG2}, we have that:
    \begin{align*}
    &\E [ HG]_{jk} = \sum \limits_{a=1}^N \E \left [ H_{ja} G_{ak} \right ] = - \sum \limits_{a=1}^N \E \| H_{12} \|^2 \E \left [ G_{aa} G_{jk} \right ] \\
    &= - \E\| H_{12} \|^2 \E \left [ G_{jk} \mathrm{Tr}(G) \right ] ,
    \end{align*}
    and the proof is complete.
\end{proof}
We can now define a linear operator $S: \mathbb{C}^{N \times N} \mapsto \mathbb{C}^{N \times N},$ by:
\beq \label{S-op}
S[R] = \tilde{\E} [\tilde{H} R \tilde{H} ] ,
\eeq
so that we may write:
\beq
\E [HG] = - \E [S[G]G].
\eeq
Plugging this information to \eqref{HG1} we obtain the \textit{Dyson equation}:
\beq \label{DysonE1}
I + (z+S[G])G = D , \quad D := HG + S[G]G,
\eeq
where $S$ defined in \eqref{S-op} is called the \textit{self-energy} operator in the Dyson equation and $D$ is called the \textit{error matrix}. Observe that $\E[D]=0$, so that we can always work with the probabilistically "close" equation:
\beq \label{DysonE2}
I + (z+S[M])M = 0 .
\eeq

In the GUE ensemble we have that $K_{11} = K_{22} = \frac{1}{2N}$ and of course $K_{12}=K_{21}=0.$ This gives the following expression (see \eqref{tilde-exp}) for the self-energy operator:
\beq
(S[R])_{jk} = \tilde{\E} [\tilde{H} R \tilde{H} ]_{jk} = \delta_{jk} \frac{1}{N} \mathrm{Tr}(R),
\eeq
so that the Dyson equation becomes:
\beq
I + \left (z + \frac{1}{N} \mathrm{Tr}(G) \right )G = 0.
\eeq
Taking normalized trace this is identical to the Stieltjes transform equation:
\beq
1 + (z + S_N)S_N = 0 .
\eeq
\end{section}

\begin{section}{Wigner-type and correlated Hermitian models}
\label{s:W+C}

 In the Wigner-type model we still assume that the entries of $H$ are independent up to the Hermitian symmetry, but we drop the identically-distributed condition. We define the variances of each random entry $h_{ij}$ as:
\beq \label{s-var}
s_{ij} = \E\|h_{ij}\|^2,
\eeq
which correspond to the matrix of variances $S$ with $(S)_{ij}:=s_{ij}.$ In the Wigner model, we have that $s_{ij} = \frac{1}{N}$ but here we just impose the condition that:
\beq \label{s-cond}
\frac{c}{N} \leq s_{ij} \leq \frac{C}{N},
\eeq
uniformly for each $i,j=1,...,N$ and some $c,C>0$.

We calculate the self-energy operator for this model as follows:
\begin{align*} \label{vec-self-e}
&\Tilde{\E}[\Tilde{H} R \Tilde{H}]_{jk} = \sum \limits_{a,b} R_{ab} \tilde{\E} [ \tilde{H}_{ja} \tilde{H}_{bk}] \\
&= \delta_{jk} \sum \limits_{a=1}^N R_{aa} s_{aj} = \delta_{jk} \left \langle S_j , \mathrm{diag}(R) \right \rangle , \numberthis
\end{align*}
where $S_j$ is the $j$-th column of $S$ and $\mathrm{diag}(R)$ is the vector corresponding to the diagonal of $R$. 

For the Dyson equation, by using \eqref{vec-self-e}, we now have to consider a system of $N$ equations for the unknown vector $m=(m_1,...,m_N)$ which represents the diagonal of $G$:
\beq \label{vec-D}
1 + (z + \left \langle S_j , m \right \rangle ) m_{j} = 0,
\eeq
for $j=1,...,N$.

For the stability analysis, the linear stability operator is given by:
\[ 
I - m^2 S : \ \ \mathbb{C}^N \mapsto \mathbb{C^N},
\]
with:
\beq \label{vec-stab}
\left [ (I - m^2 S)(x) \right ]_j := x_j - m_j^2  \sum \limits_{a=1}^N s_{ja} x_a .
\eeq
We will see a proof of this in section \ref{s:Stability} where the linear stability operator is given in its general form.

In the correlated Hermitian model, we drop the independence condition in the matrix entries of $H$, so that the matrix elements may have non-trivial correlations in addition to the one required by the Hermitian symmetry. 

This is the most general model and the one mostly used in applications where we use the \textit{Hermitization} technique to "hermitize" the matrix and then apply the Dyson equation technique. This is what we will do with the model of the next chapter, see section \ref{s:model}.

The self-energy operator here is in its most general form:
\[
(S[R])_{ij} = \E[H R H]_{ij} = \sum \limits_{a,b} \E[H_{ia}H_{bj}] R_{ab} ,
\]
and we have to use the analogue of \eqref{s-cond} which is:
\beq \label{S-cond}
c \langle R \rangle \preceq S[R] \preceq C \langle R \rangle,
\eeq
where $\langle \ \cdot \ \rangle$ represents the normalized trace of a matrix and $A \preceq B$ means that the matrix $B-A$ is positive semi-definite.

Conditions \eqref{s-cond} and \eqref{S-cond} are known as \textit{mean-field} conditions which make the spectrum map into an interval or area of order $1$.

The Dyson equation is in its general form:
\[
I +(z+S[G])G = 0,
\]
while the linear stability operator is also in its general form and given by:
\beq \label{super-oper1}
I - C_{M} \circ S  : \ \ \mathbb{C}^{N \times N} \mapsto \mathbb{C}^{N \times N},
\eeq
where $C_M$ is the "sandwich" operator $C_M[R]= MRM,$ and $M$ is the solution matrix, so that the linear stability operator becomes:
\beq \label{matrix-stab}
R \mapsto R - M S[R] M .
\eeq
Notice that this form implies \eqref{vec-stab} for the vector Dyson equation. A summary for the proof of \eqref{matrix-stab} will be given in the deterministic stability analysis in section \ref{s:stability}.
\end{section}

\begin{section}{The multivariate cumulant expansion}
\label{s:Cumulants}

In this section we give a summary of the proof that the probabilistic Dyson equation is "close" to the original one, which is treated as a perturbation of the probabilistic one with error given by the matrix $D$:
\[
I + (z + S[M])M = 0 , \quad I + (z+S[M])M = D .
\]
In order to prove this we have to bound expectations of high powers of the quantity $|D_{ij}|$, for any $i,j=1,...N$. 

We will follow the multivariate cumulant expansion technique which was given in \cite{erdHos2019random} and involves more general matrix cases with slow correlation decay.

Another approach for the proof was given in \cite{ajanki2019stability} for correlated
matrices with fast exponential correlation decay. There they used the Schur complement method together with concentration estimates on quadratic functionals of independent or essentially independent random vectors.

Let: 
\beq \label{D-error}
D = HG + S[G]G .
\eeq 
We extend formula \eqref{gauss1} to non-Gaussian random variables as follows:
\pagebreak
\begin{lem} Let $h$ be a general non-Gaussian real random variable such that all its moments and cumulants exist. Then for $f\in C^\infty,$ we have that:
\beq \label{non-Gauss}
\E[hf(h)] = \sum \limits_{k=0}^\infty \frac{\kappa_{k+1}}{k!} \E [f^{(k)}(h)],
\eeq
where $\kappa_{k}$ is the $k$-th cumulant of $h$.
\end{lem}
We recall that the cumulants of a real random variable are defined by the relation:
\beq \label{cumulant}
\log \E [e^{th}] = \sum \limits_{k=0}^\infty \frac{\kappa_k}{k!}t^k ,
\eeq
which is similar to the moment-generating formula:
\beq \label{moment}
\E [e^{th}] = \sum \limits_{k=0}^\infty \frac{m_k}{k!}t^k ,
\eeq
and so the moments can define the cumulants and the cumulants can define the moments, by comparing the two power series.
\begin{proof}[Proof of \eqref{non-Gauss}]
    For the proof, we can use the Fourier transform of $f$,
    \[
    \Hat{f}(t) = \int_{\mathbb{R}} f(x) e^{itx} dx
    \]
    and then differentiate to get that:
    \[
    \Hat{f'}(t) = i \int_{\mathbb{R}} xf(x) e^{itx}dx,
    \]
    so by the Fourier-inversion formula for the function $xf(x)$ we have that:
    \beq \label{parseval}
    xf(x) = i \int_{\mathbb{R}} \overline{\Hat{f'}(t) } e^{itx} dt \Rightarrow \E[xf(x)] = i \int_{\mathbb{R}} \overline{\Hat{f'}(t) } \Hat{\mu}(t) dt,
    \eeq 
    where we defined the measure $\mu$ to be the distribution of the real random variable $h$ and
    \[
    \Hat{\mu}(t) = \int_{\mathbb{R}} e^{itx} d \mu(x) = \E[e^{ith}] .
    \]
    We then perform integration by parts on the RHS of \eqref{parseval} to get that:
    \begin{align*} \label{parseval2}
    &i \int_{\mathbb{R}} \overline{\hat{f}^{\prime}(t)} \hat{\mu}(t) d t  =-i \int_{\mathbb{R}} \overline{\hat{f}(t)} \hat{\mu}^{\prime}(t) dt =-i \int_{\mathbb{R}} \overline{\hat{f}(t)} \hat{\mu}(t)(\log \hat{\mu}(t))^{\prime} dt \\
    &=\sum_{k=0}^{\infty} \frac{\kappa_{k+1}}{k!} \int_{\mathbb{R}}(i t)^k \overline{\hat{f}(t)} \hat{\mu}(t) d t=\sum_{k=0}^{\infty} \frac{\kappa_{k+1}}{k!} \mathbb{E} [f^{(k)}(h) ] ,
    \end{align*}
    where in the last step we used the Fourier-inversion formula once more for $f$ and differentiated it $k$ times.
\end{proof}
In order to apply \eqref{non-Gauss} to random matrices, we have to define the \textit{joint cumulants} for a family of random variables, such as a random matrix.

If $\mathbf{h}=\left(h_1, h_2, \ldots h_m\right)$ is a collection of random variable, then
$$
\kappa(\mathbf{h})=\kappa\left(h_1, h_2, \ldots h_m\right)
$$
are the coefficients of the logarithm of the moment-generating function:
\beq \label{log-moments-multi}
\log \mathbb{E} [e^{\mathbf{t} \cdot \mathbf{h}}] =\sum_{\mathbf{k}} \frac{\mathbf{t}^{\mathbf{k}}}{\mathbf{k} !} \kappa_{\mathbf{k}},
\eeq
where $\mathbf{t}=\left(t_1, t_2, \ldots, t_n\right) \in \mathbb{R}^n$ and $\mathbf{k}=\left(k_1, k_2, \ldots, k_n\right) \in \mathbb{N}^n$ is a multi-index with $\mathrm{n}$ components and:
$$
\mathbf{t}^{\mathbf{k}}:=\prod_{i=1}^n \mathrm{t}_{\mathrm{i}}^{{k}_{{i}},}, \quad \mathbf{k} ! :=\prod_{i=1}^n {k}_{{i}} !, \quad \kappa_{\mathbf{k}}=\kappa \left(\mathrm{h}_1, \mathrm{~h}_1, \ldots \mathrm{h}_2, \mathrm{~h}_2, \ldots\right),
$$
where $h_j$ appears $k_j$-times. The analogue of \eqref{non-Gauss} is:
\beq \label{non-Gauss-multi}
\mathbb{E} [h_1 f(\mathbf{h})]=\sum_{\mathbf{k}} \frac{\kappa_{\mathbf{k}+\mathbf{e}_1}}{\mathbf{k} !} \mathbb{E} [f^{(\mathbf{k})}(\mathbf{h})], \quad \mathbf{h}=\left(\mathrm{h}_1, \mathrm{~h}_2, \ldots, \mathrm{h}_{\mathrm{n}}\right),
\eeq
where $f^{(\mathbf{k})} = \frac{\partial^{k_1}...\partial^{k_n}}{\partial{h_1}...\partial{h_n} } f$ and the summation is for all $n$ multi-indices with:
\[
\mathbf{k}+\mathbf{e}_1 = (k_1+1,k_2,...,k_n) .
\]
The proof is similar with the proof of \eqref{non-Gauss}.

Now, back to the definition of the error matrix $D$ in \eqref{D-error}, we will use this multivariate cumulant expansion \eqref{non-Gauss-multi} to prove that the expectation of high powers of the quantity $|D_{ij}|$ is small. This is done by writing:
\beq \label{D^p}
\E |D_{ij}|^{2p} = \E \left [ (HG + \mathcal{S}[G]G )_{ij} D_{ij}^{p-1} \overline{D_{ij}^p} \right ]
\eeq
and then use \eqref{non-Gauss-multi} to do an integration by parts in the first $H$ factor, while considering everything else as a function $f(H)$. This involves some heavy analysis of cumulant expansions and combinatorics, which can be found in \cite{erdHos2019random}.

The authors there arrive at the result of [Theorem 4.1] in \cite{erdHos2019random}, which takes into account a couple of different possible random matrix norms for the error matrix $D$. A simple corollary is for the following maximum-norm:
\begin{thm}[Bound on the error matrix D] \label{error-thm}
    Under the mean-field condition \eqref{S-cond} and a finite-moments condition \eqref{finite-m} for the matrix $H$, we have that for any $\gamma, \epsilon , D > 0 $ and some bounded $z=E+i\eta$ with $\eta \geq N^{-1+\gamma}$, there exists $C>0$ such that:
    \beq
    \mathbb{P} \left ( \| D(z) \|_{\max} \geq \frac{N^\epsilon}{\sqrt{N\eta}} \right ) \leq \frac{C}{N^D} .
    \eeq
\end{thm}
Here the maximum-norm of a matrix $T$ is defined by:
\beq \label{max-norm}
\|T\|_{\max} := \max \limits_{i,j} |T_{ij}|,
\eeq
while the finite-moments condition can be described by:
\beq \label{finite-m}
\max \limits_{i,j} \E  |\sqrt{N}H_{ij} |^p  \leq \mu_p,
\eeq
for a sequence of constants $\mu_p,$ where $H$ is our initial Hermitian matrix.
\end{section}

\begin{section}{The deterministic stability step} \label{s:Stability}

In this section we give a summary for the proof of the stability of the deterministic equation:
\beq \label{deter-eq}
I + (z + \mathcal{S}[M])M = 0.
\eeq
Since the solution $M$ of the equation \eqref{deter-eq} as well as the solution $M'$ of its perturbed version \eqref{DysonE1} can be expressed as a function of the variable $z$, to prove its stability it is enough to control the derivative:
\[
\partial_z M = \frac{\partial M(z) }{ \partial z} ,
\]
therefore we differentiate \eqref{deter-eq} to get a formula for the derivative of the solution:
\begin{align*}
    &(I + \mathcal{S}[\partial_z M] ) M + (z + \mathcal{S}[M] ) \partial_z M = 0 \Rightarrow \\
    &M^2 + M \mathcal{S}[\partial_z M] M  - \partial_z M = 0  \Leftrightarrow \\
    &\partial_z M - M\mathcal{S}[\partial_z M] M = M^2 \Leftrightarrow \\
    &( I - C_M \circ S )[\partial_z M] = M^2 , \numberthis
\end{align*}
where $C_M$ is the "sandwich" operator defined in \eqref{super-oper1}. 

If we now show that the "linear stability operator" $I - C_M \circ S$ is invertible, then the derivative $\partial_z M$ will have reasonable behaviour, which means that the equation \eqref{deter-eq} will be stable to perturbations.

To show invertibility, we have to bound the norm of the inverse of the linear stability operator. For this, we will use the "super-operator" norm $\| \ \cdot \ \|_{op}$ which is the operator norm over the Hilbert-Schmidt norm in the Hilbert space of complex matrices.

The idea of the proof in \cite{ajanki2019stability} is to write the linear stability operator as a product of invertible operators with an operator of the form \ $\mathcal{U} - \mathcal{T},$ where $\mathcal{U}$ is a unitary operator and $\mathcal{T}$ is a self-adjoint operator, called the "saturated self-energy" operator. In \cite{erdHos2019random} the multivariate cumulant expansion technique is used instead, to give the same result.

They both arrive at the following theorem:
\begin{thm}[Stability of the Dyson equation] \label{Dyson-stab}
    If the mean-field condition \eqref{S-cond} is satisfied, then for $z$ in the "bulk" of the support of $\rho$, we have that:
    \[
    \left \| ( I - C_{M(z)} \circ S )^{-1} \right \|_{op} \lesssim \frac{1}{\rho(z)^C},
    \]
\end{thm}
for some constant $C>0$, where $\rho$ is called the "density of states" and is defined as:
\beq \label{states}
\rho(z) := \frac{1}{\pi} \left \langle \mathrm{Im} M(z) \right \rangle .
\eeq
We will see in section \ref{s:Results} that this is indeed a density for a measure and it is actually the harmonic extension of the real density of the limiting spectral measure of $H$.
 
The following Lemma is derived by the analysis in \cite{ajanki2019stability}:
\begin{lem}[Linear stability operator estimate] \label{lso-estimate} Under the mean-field condition \eqref{S-cond}, we have the following estimate, for $z$ in the "bulk" of $\rho$:
\[
\left \| ( I - C_{M(z)} \circ S )^{-1} \right \|_{op} \lesssim \frac{1}{\rho(z)^C} \left \| \left ( \mathcal{U} - \mathcal{T} \right )^{-1} \right \|_{op},
\]
\end{lem}
for some constant $C>0,$ where $\mathcal{U}$ is a unitary operator and $\mathcal{T}$ a self-adjoint one.

Under some further conditions we can get the same result by the following supplementary Lemma in \cite{ajanki2019stability}, which is of independent interest and inspired our stability analysis in section \ref{s:stability} for a different random matrix model:
\begin{lem}[Rotation-Inversion Lemma] \label{rotate-inverse}
Let $\mathcal{T}$ be a self-adjoint and $\mathcal{U}$ a unitary operator on $\mathbb{C}^{N \times N}$. Suppose that $\mathcal{T}$ has a spectral gap, i.e., there is a constant $\theta>0$ such that:
$$
\text { Spec } (\mathcal{T}) \subset \left[-\|\mathcal{T}\|_{{op}}+\theta,\|\mathcal{T}\|_{{op}}-\theta\right] \cup\left\{\|\mathcal{T}\|_{{op}}\right\},
$$
with a non-degenerate largest eigenvalue $\|\mathcal{T}\|_{op} \leq 1$. Then there exists a constant $C>0$ such that:
$$
\left\|(\mathcal{U}-\mathcal{T})^{-1}\right\|_{op} \leq \frac{C}{\theta} \bigg| 1-\|\mathcal{T}\|_{op}\langle K, \mathcal{U}[K] \rangle \bigg |^{-1},
$$
where $K$ is the normalized eigenmatrix of $\mathcal{T}$, corresponding to the largest eigenvalue $\|\mathcal{T}\|_{op}.$
\end{lem}
The combination of Lemma \ref{lso-estimate} and Lemma \ref{rotate-inverse} gives again Theorem \ref{Dyson-stab}.

\end{section}

\begin{section}{Results about the Dyson equation} \label{s:Results}

In this section, we give three results concerning the Dyson equation analysis of a Hermitian matrix. Notice that the first and second result is independent of the third one for which we need the small perturbation theorem as well as the stability theorem.

The first two results summarize the properties of the solution matrix $M$ which can be regarded as a Stieltjes transform of a matrix-valued measure, which gives rise to a certain real density $\rho$.

The third result asserts that $\rho$ should be the limiting spectral density of the initial Hermitian matrix $H$ since its resolvent $G$ is indeed "close" to the solution matrix $M$, as the dimensions go to infinity.

\underline{1. Solution of the Dyson equation.}

We remind that the existence and uniqueness of the solution $M$ for the Dyson equation, with a positive semi-definite imaginary part, is established in \cite{helton2007operator}, through a certain fixed-point theorem and can be deduced by a well-defined iterative scheme.

Since $M(z):\mathbb{C}_+ \mapsto \mathbb{C}^{N \times N}_+$ is actually a function from $\mathbb{C}_+ := \left \{ z \in \mathbb{C} \ | \ \mathrm{Im}(z) > 0 \right \}$ to the Hilbert space of complex matrices with $\mathrm{Im} (M(z)) \succeq 0$ (because $\mathrm{Im}(z) = \eta > 0$), we can view it as a matrix-valued Herglotz function and apply its Nevanlinna representation for a suitable matrix-valued measure on the real line. All the underlying theory can be found in section 5 of \cite{gesztesy2000matrix}.

\pagebreak

We get the following theorem for the solution matrix $M(z)$:
\begin{thm}[Stieltjes transform representation of $M$] \label{ST-M} 
    Let $M$ be the unique solution of \eqref{deter-eq} with positive semi-definite imaginary part. Then, $M$ admits a Stieltjes transform representation:
\[
M_{xy} = \int_{\mathbb{R}} \frac{V_{xy}(dt)}{t - z} ,
\]
    for $x,y=1,...,N$ and $z \in \mathbb{C}_+$. The measure $V(dt) = (V_{xy}(dt))_{x,y=1}^N$ on the real line with values in positive semi-definite matrices is unique and satisfies the normalization $V(\mathbb{R})=I$. 
\end{thm}

\underline{2. Density of states for the solution matrix.}

Following the analysis from the Nevanlinna functions in \cite{gesztesy2000matrix}, we have that the diagonal elements of $V$ are actually finite Borel measures in $\mathbb{R}$. The normalization $V(\mathbb{R})=I$ makes them probability measures. This makes the normalized trace of $V$ a probability measure, as a convex combination of probability measures. Hence, from the previous theorem \ref{ST-M}, we derive the following:
\begin{thm}[Density of states] \label{States-d}
    Let $\langle M(z) \rangle : \mathbb{C}_+ \mapsto \mathbb{C}_+$ denote the normalized trace of the solution matrix. Then $\langle M(z) \rangle$ admits a Stieltjes tranform representation of a probability density $\rho$, such that:
    \[
    \langle V(dt) \rangle = \rho(t) dt,
    \]
    i.e.
    \[
    \langle M(z) \rangle = \int_{\mathbb{R}} \frac{\rho(t)}{t - z} dt .
    \]
\end{thm}
We remark that under the flatness condition \eqref{S-cond}, the density $\rho$ becomes Hölder continuous. [Proposition 2.2] in \cite{ajanki2019stability}. 

The density $\rho$ is called the \textit{density of states} as we have seen in section \ref{s:Stability} because of the following fact. By the Stieltjes inversion formula, we have that:
\[
 \lim \limits_{\eta \to 0^+} \frac{1}{\pi} \langle \mathrm{Im} M(z) \rangle = \rho(t),
\]
 hence $$ \rho(z) = \frac{1}{\pi} \langle \mathrm{Im} M(z) \rangle $$ is the harmonic extension of $\rho$.

 We will see now that $\rho$ is actually the limiting density for the eigenvalue counting function of the Hermitian matrix $H$.
 
\underline{3. Approximation of the resolvent matrix - Local law.}

Here we use our two ingredients, the small error term $D$ in section \ref{s:Cumulants} and the stability of the Dyson equation in section \ref{s:Stability}, to present the last important theorem of this chapter, which is the approximation of the resolvent matrix $G = (H-zI)^{-1}$ by the solution matrix $M(z)$ for large dimensions.
\begin{thm}[Local law for the solution matrix] \label{local-entry}
    Assume that the flatness condition \eqref{S-cond} and the finite-moments condition \eqref{finite-m} hold, as well as a suitable decay of correlations condition like \eqref{exp-decay} for the matrix $H$. Let $M$ be the solution of the Dyson equation \eqref{deter-eq} with the self-energy operator defined as in \eqref{S-op}. Then for the spectral parameter $z=E+i\eta$ inside the "bulk" of the support of $\rho$, where $\rho$ is the density of states, i.e. with $\rho(E) \geq \delta$ and with $\eta \geq N^{-1+\gamma},$ for some $\gamma, \delta > 0,$ we have the "entrywise local law":
    \beq 
    \mathbb{P}\left(\left|G_{i j}(z)-M_{i j}(z)\right| \geq \frac{N^{\varepsilon}}{\sqrt{N \eta}}\right) \leq \frac{C}{N^D},
    \eeq
    for some constant $C$ depending on $\gamma, \delta, \varepsilon, D$ and the constants on \eqref{S-cond} and \eqref{finite-m}.
\end{thm}

We note the correlation decay condition for an exponential decay of the matrix correlations, is:
\beq \label{exp-decay}
\operatorname{Cov}\left(\phi\left(W_A\right) , \psi\left(W_B\right)\right) \leq C(\phi, \psi) e^{-d(A, B)},
\eeq
 where $W=\sqrt{N} H$ is the re-scaled random matrix, $A, B$ are two subsets of the index set $[1, N ] \times [1, N]$ and the distance $\mathrm{d}$ is the usual Euclidean distance between the sets $\mathrm{A} \cup \mathrm{A}^{\mathrm{t}}$ and $\mathrm{B} \cup \mathrm{B}^{\mathrm{t}}$ and $\mathrm{W}_{\mathrm{A}}=\left(w_{i j}\right)_{(i, j) \in \mathrm{A}}$ .

For a more general slow correlation decay, the condition is:
\beq \label{slow-decay}
\left|\kappa\left(f_1(W), f_2(W)\right)\right| \leq \frac{C}{1+d\left(\operatorname{supp} f_1, \operatorname{supp} f_2\right)^s}\left\|f_1\right\|_2\left\|f_2\right\|_2,
\eeq
for some $s>12$ and $f_1,f_2$ square integrable functions on random $N \times N$ complex Hermitian matrices.

A simple corollary of Theorem \ref{local-entry} is the following "averaged local law":
\begin{cor} \label{local-average}
    With the same conditions as in Theorem \ref{local-entry}, we have that:
    \beq
        \mathbb{P}\left(\left|\frac{1}{N} \operatorname{Tr}[G(z)-M(z)]\right| \geq \frac{N^{\varepsilon}}{N\eta}\right) \leq \frac{C}{N^D},
    \eeq
    again for a constant $C$ depending on $\gamma, \delta, \varepsilon, D$ and the constants on \eqref{S-cond} and \eqref{finite-m}.
\end{cor}

The main key technical tool and observation is still the fact that the Stieltjes transform of the eigenvalue counting function $\rho_N$ is given by the normalized trance of the resolvent, because:
\[
S_N(z) = \int_{\mathbb{R}} \frac{\rho_N(t)}{t - z} dt = \frac{1}{N} \sum \limits_{i=1}^N \frac{1}{\lambda_i - z} = \langle G(z) \rangle,
\]
so that:
\[
S_N(z) = \langle G(z) \rangle \to \langle M(z) \rangle = S_{\rho}(z)
\]
and hence according to all of our analysis in section \ref{s:ImS}: $\rho_N \to \rho,$ where $\rho$ is the density of states.
\end{section}
\chapter{A non-Hermitian generalization of the Marchenko-Pastur distribution}

\begin{section}{The random matrix model} \label{s:model}
In order to define our model, we firstly define two auxiliary square random matrices $P$ and $Q$ of a given dimension $N \times N.$ We assume that all the entries of $P$ and $Q$ are i.i.d. complex random variables with mean value $0$ and complex variance $\frac{1}{2N}.$

The key difference in this random matrix model compared to the previous results stated in section \ref{prev-r} is that we allow the entries of $P$ and $Q$ to have any possible common distribution $\mathcal{F}$ with the following properties:
\begin{itemize}
    \item \label{a1} $\mathcal{F}$ has mean zero, complex variance $\frac{1}{2N}$ and second moment zero, that is if $x \sim \mathcal{F}$ then $\E[\mathrm{Re} (x)]=\E [\mathrm{Im} (x)] = 0,$ $\mathrm{Var}[\mathrm{Re} (x)] = \mathrm{Var}[\mathrm{Im} (x)]=\frac{1}{4N}$ and $\E[\mathrm{Re} (x) \mathrm{Im} (x)]=0.$ The last assumption is made to ease the computations of the model while the first two are vital for the results.
    
    \item \label{a2} $\mathcal{F}$ has all of its moments finite, this means that if $x \sim \mathcal{F}$ then $\E[x^n] < \infty$ for any $n \in \mathbb{N}.$ This is necessary for the Dyson equation method to take place.

    \item \label{a3} $\mathcal{F}$ has a bounded density. Specifically, we make the assumption that there exist $q\geq 1,\ \kappa>0$ and a probability density $\psi \in L^q(\mathbb{C})$ with $\| \psi \|_q \leq N^\kappa$ such that:
    \beq \label{prob-dens}
    \mathbb{P} \left ( \sqrt{N} P_{ij} \in B \right ) = \mathbb{P} \left ( \sqrt{N} Q_{ij} \in B \right ) = \int_B \psi (z) d^2 z,
    \eeq
    for each $i,j=1,...,N$ and $B \subset \mathbb{C}$ a Borel set. This condition makes the least singular value problem (see section \ref{lsv}) much easier.
\end{itemize}
Our model is defined as follows. We pick
\beq \label{X1,X2}
    X_1 := \sqrt{1+\tau} P + \sqrt{1-\tau}Q \ \ \text{and} \ \ X_2 := \sqrt{1+\tau} P - \sqrt{1-\tau}Q,
\eeq
for a real parameter $\tau \in [0,1].$ Notice that $X_1$ and $X_2$ are correlated through this parameter. For $\tau=0,$ we have that $X_1$ and $X_2$ are completely uncorrelated, while for $\tau \in (0,1)$ and $\tau \to 1$ they become more and more correlated, reaching complete correlation for $\tau=1$ where they become identical matrices. 
Through $X_1$ and $X_2$ we now form the \textit{correlated sample covariance ensemble}:
\beq \label{X}
\mathbf{X} := X_1 X_2^*
\eeq
What can we say about the eigenvalues of this ensemble? For $\tau=0,1$ we get two well-known cases. Specifically, for $\tau=0$ we have a product of two independent circular matrices, while for $\tau=1$ we have a product of a circular matrix and its conjugate. Notice that in this case $\mathbf{X}$ becomes hermitian, that's why $\tau$ is called the $\textit{non-hermiticity}$ parameter.

In the first case, we have a finite product of circular random matrices, which means that the empirical spectral distribution converges to a power of a circular law. This distribution has the unit disk as support and exhibits a concentration of eigenvalues around the origin according to the number of product elements. In our case, we have the second power of the circular law. See \cite{o2011products} for this result.

In the second case, we have the classical sample covariance ensemble, which means that the empirical spectral distribution converges to the Marchenko-Pastur distribution. This means that the support in this case collapses to the real line in $[0,4]$. This is expected since the matrix becomes Hermitian and the eigenvalues real. See \cite{marchenko1967distribution} for this result.

For $\tau \in (0,1)$ we are expecting a transition between these two distributions. Since there is a concentration of eigenvalues near the origin both in the second power of the circular law as well as in the Marchenko-Pastur distribution, we are expecting this to hold for the transitional distribution while the initial disk moves to the right and shrinks to the real line. Of course, we are also expecting local laws to hold for this random matrix ensemble.
\end{section}

\begin{section}{Previous results on the spectral distribution}
\label{prev-r}

This model was previously studied in \cite{akemann2021non} where it was assumed that $\mathcal{F} \sim \mathcal{CN}(0,\frac{1}{2N}),$ that is $\mathcal{F}$ follows a complex normal distribution with mean $0$ and variance $\frac{1}{2N}.$ It was therefore given the name \textit{non-Hermitian Wishart ensemble}.

In that study, since there was a Gaussian distribution assumption for the entries of $P$ and $Q$, different and more concrete methods were used through orthogonal polynomials. This is because in this case we have specific potentials in our joint eigenvalue density for $X_1$ and $X_2,$ which can then be used to find the joint eigenvalue density of $\mathbf{X}$ and proceed with asymptotic estimates.

These methods are rather concrete and powerful, so they gave rise to the calculation of the transitional distribution ($0<\tau<1$) even for rectangular matrices $P,Q \in \mathbb{C}^{N \times (N+\nu)},$ where $\nu = \mathcal{O}(N)$ is a \textit{non-square} parameter.

Using the parameter $\alpha:= \lim \limits_{N \to \infty} \frac{\nu}{N}$ they arrived at the following result concerning the transitional distribution of the non-Hermitian Wishart ensemble:
\begin{thm}[Non-Hermitian Wishart ensemble law from \cite{akemann2021non}] 
As $N \to \infty$ the empirical spectral measure of the non-Hermitian Wishart ensemble converges to the following deterministic measure on the complex plane:
\beq \label{NHW-law}
d \widehat{\mu}(\zeta):=\frac{1}{1-\tau^2} \frac{1}{\pi \sqrt{4|\zeta|^2+\left(1-\tau^2\right)^2 \alpha^2}} \cdot \mathbf{1}_{\widehat{S}_{\tau,\alpha}}(\zeta) d^2 \zeta ,
\eeq
where the support $\widehat{S}_{\tau,\alpha}$ of the spectrum is given by:
\beq \label{NHW-supp}
\widehat{S}_{\tau,\alpha}:=\left\{\zeta=x+i y:\left(\frac{x-\tau(2+\alpha)}{\left(1+\tau^2\right) \sqrt{1+\alpha}}\right)^2+\left(\frac{y}{\left(1-\tau^2\right) \sqrt{1+\alpha}}\right)^2 \leq 1\right\} .
\eeq
\end{thm}
\begin{figure*}[h]
    \centering
     \includegraphics{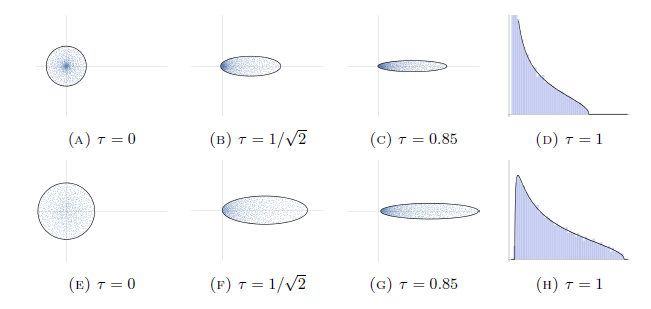}
     \caption{The eigenvalues of $\mathbf{X}$ for large N. We have that $\alpha = 0$ for the figures (A)-(D) and $\alpha=1$ for (E)-(H). This figure is from \cite{akemann2021non}. }
     \label{fig:NHW}
 \end{figure*}
\end{section}

\begin{section}{Main theorems}
Our main theorem is about the spectrum of the correlated covariance matrix $\mathbf{X_N}=X_1X_2^*$ where $X_1$ and $X_2$ are square matrices and its corresponding limiting local law, as the dimension increases.

The empirical measure $\widehat{\mu}_N$ associated with $\mathbf{X}_N$ is given by:
$$ \widehat{\mu}_N := \frac{1}{N} \sum \limits_{j=1}^N \delta_{\widehat{\zeta}_j},$$
where $ \{ \widehat{\zeta_j} \}_{j=1}^N $ are the $N$ complex eigenvalues of $\mathbf{X}_N.$
We firstly introduce the limiting measure $\widehat{\mu}$ of the empirical spectral distribution of this matrix as follows:
\begin{definition}[Limiting measure B] The measure $\widehat{\mu}$ is defined by
\beq \label{measure1}
    d\widehat{\mu}(z) := \frac{1}{1-\tau^2}\frac{1}{2\pi|z|} \mathbf{1}_{\widehat{S}\tau}(z) d^2 z,
\eeq
with the following support $\widehat{S}_\tau:$
\beq \label{supp1}
\widehat{S}_\tau(z) := \left \{ z=x+iy : \left ( \frac{
x-2\tau}{1+\tau^2} \right )^2 + \left ( \frac{y}{1-\tau^2}\right )^2 \leq 1 \right \}.
\eeq
\end{definition}
Notice that $\widehat{S}_\tau(z)$ defines a shifted ellipse centered at $(2\tau,0)$ while the measure predicts a concentration of eigenvalues around the edge point of the ellipse at $(0,0).$

The theorem about the correlated covariance matrix ensemble will follow from a similar theorem about the Dirac matrix, defined as follows:
$$ \mathbf{Y}_N := \begin{pmatrix} 0 & X_1 \\ X_2^* & 0 \end{pmatrix} \in \mathbb{C}^{2N \times 2N}$$
The limiting measure $\mu$ of the empirical spectral distribution of the Dirac matrix is defined as follows:
\begin{definition}[Limiting measure A] The measure $\mu$ is defined by
\beq \label{measure2}
d\mu(z) := \frac{1}{1-\tau^2} \frac{1}{\pi} \mathbf{1}_{S_{\tau}}(z) d^2 z,
\eeq
 with the following support $S_\tau$:
 \beq \label{supp2}
 S_\tau := \left \{ z=x+iy: \left (\frac{x}{1+\tau} \right )^2 + \left ( \frac{y}{1-\tau}\right )^2 \leq 1 \right \}.
 \eeq
\end{definition}
Here, $S_\tau$ is a classical ellipse centered at $(0,0)$ with its eccentricity and edge points depending on $\tau.$

For the proof of our theorems we will work primarily on the following domain for the Dirac matrix:
\begin{definition}[Spectral parameter domain A] \label{pardom}
For any $\delta, \kappa>0,$ define the ellipse that arises from the exclusion of the edges and a small ball around 0 as follows:
\beq \label{par-domB}
S_{\tau,\delta, \kappa} := \left \{ \zeta=x+iy : \left ( \frac{x}{1+\tau}\right )^2 + \left ( \frac{y}{1-\tau} \right )^2 \leq 1 - \delta \right \}\backslash B_0(\kappa),
\eeq
where $B_0(\kappa)$ is a ball of radius $\kappa$ around 0.
\end{definition}
Analogously, we define the spectral domain for the correlated covariance matrix:
\begin{definition}[Spectral parameter domain B]
    For any $\delta,\kappa >0$ define the shifted ellipse that arises from the exclusion of the edges and a small ball around $0$ as follows:
    \beq \label{par-domA}
    \widehat{S}_{\tau,\delta,\kappa} := \left \{ \zeta=x+iy : \left ( \frac{x - 2\tau}{1+\tau^2}\right )^2 + \left ( \frac{y}{1-\tau^2} \right )^2 \leq 1 - \delta \right \}\backslash B_0(\kappa) .
    \eeq
\end{definition}
For the statement of the main theorem we need to define the \textit{test functions} as complex functions $f:\mathbb{C}\to \mathbb{C}$ with compact support and at least two times complex-differentiable with continuous complex second derivative. Their corresponding \textit{zoom functions} are defined as follows:
\beq \label{zoomf}
f_{\zeta_0,a}(\zeta) := N^{2a} f(N^a(\zeta - \zeta_0)),
\eeq
where $N$ is the dimension of the matrix, $\zeta \in \mathbb{C}$ and $\zeta_0 \in \mathbb{C}$ is the point we want to "zoom in" at a zooming scale given by $a>0$.

We are now ready to state our main theorem.
\begin{thm}[Local universal law for the spectrum of the correlated sample covariance ensemble] \label{main'}
Let $a\in (0,1/2)$ and $\kappa,\delta \in (0,1).$ Then, for any $\epsilon>0$ and $\nu \in \mathbb{N}$, there is a constant $C>0,$ such that:
\[
\mathbb{P} \left ( \left | \frac{1}{N}\sum \limits_{\widehat{\zeta} \in Spec({X_1X_2^*})}f_{\widehat{\zeta_0},a} (z) - \int_{\mathbb{C}} f_{\widehat{\zeta_0},a} (z) d{\widehat{\mu}}(z)  \right | \leq N^{-1+2a+\epsilon} \|\Delta f\|_{L^1} \right ) \geq 1 - CN^{-\nu},
\]
uniformly for all $N \in \mathbb{N},$ $\widehat{\zeta_0} \in \widehat{S}_{\tau,\delta, \kappa}$ and test functions $f$ with corresponding zoom function $f_{\widehat{\zeta_0},a},$ such that $\|\Delta f\|_{L^{2+a}} \leq N^D \| \Delta f \|_{L^1}$ for some $D\in \mathbb{N}.$
\end{thm}
The Dirac matrix follows an elliptic law, as formulated in the following theorem.
\begin{thm}[Local universal law for the spectrum of the Dirac matrix] \label{main}
Let $a\in (0,1/2)$ and $\kappa,\delta \in (0,1).$ Then, for any $\epsilon>0$ and $\nu \in \mathbb{N}$, there is a constant $C>0,$ such that:
\[
\mathbb{P} \left ( \left | \frac{1}{2N}\sum \limits_{\zeta \in Spec(\mathbf{Y_N})} f_{\zeta_0,a} (\zeta) - \int_{\mathbb{C}} f_{\zeta_0,a} (z) d{\mu}(z)  \right | \leq N^{-1+2a+\epsilon} \|\Delta f\|_{L^1} \right ) \geq 1 - CN^{-\nu},
\]
uniformly for all $N \in \mathbb{N},$ $\zeta_0 \in {S}_{\tau,\delta, \kappa}$ and test functions $f$ with corresponding zoom function $f_{\zeta_0,a},$ such that $\|\Delta f\|_{L^{2+a}} \leq N^D \| \Delta f \|_{L^1}$ for some $D\in \mathbb{N}.$
\end{thm}
Here we will recover the identical Dyson equation for the approximation of the trace of the resolvent of the Dirac matrix as in \cite{alt2022local}, and we can rely on the work done in \cite{alt2022local} to deduce Theorem \ref{main}, as most proof details carry over directly to our case. We will focus on the differences between the case of the Dirac matrix and the elliptical law matrix from \cite{alt2022local}. The two ingredients we must prove will be as follows:
\begin{enumerate}
\item We have to show that the trace of the resolvent for the Dirac matrix can be approximated by the $v$ as in the $ 2 \times 2$ Dyson equation \eqref{Dyson2}. The precise statement of this is formulated in Proposition \ref{Local-law1}. This furthermore requires a proof of stability for the $4 \times 4$ Dyson equation as formulated in Proposition \ref{stab-est1} and proved in Section \ref{s:stability}.
\item We need to control the relevant least singular values, formulated in Theorem \ref{sv} and proved in Section \ref{lsv}.
\end{enumerate}
\begin{proof}[Proof of Theorem \ref{main}]
The proof is identical to the proof of Theorem 2.1 in \cite{alt2022local}, using our Proposition \ref{Local-law1} in place of their Proposition 3.1 and our Theorem \ref{sv} instead of their Theorem 3.4 to control the smallest singular value.
\end{proof}
\end{section}
\begin{section}{Hermitization and Green function estimate}
We start with Girko's formula
\beq \label{Girko}
\frac{1}{2N} \sum \limits_{\xi \in Spec (Y_N)} f(\xi) = \frac{1}{4\pi N} \int_{\mathbb{C}} \Delta f( \zeta ) \log | \det \mathbf{H}_\zeta | d^2 \zeta,
\eeq
where we introduced the Hermitization:
\beq \label{Herm}
\mathbf{H}_\zeta := \begin{pmatrix} 0 & \mathbf{Y}_N - \zeta \\ \mathbf{Y}_N^* - \bar{\zeta} & 0 \end{pmatrix}
\in \mathbb{C}^{4N \times 4N} \eeq
The log-determinant of $\mathbf{H}_{\zeta}$ can be obtained from the resolvent matrix $\mathbf{G}(\zeta,\eta):= (\mathbf{H}_{\zeta}-i \eta)^{-1} $ through the identity:
\beq \label{log-det}
\log | \det \mathbf{H}_{\zeta} | = -4N \int_{0}^T \langle \mathrm{Im} \mathbf{G}(\zeta,\eta) \rangle d\eta + \log | \det(\mathbf{H}_{\zeta}-iT) |,
\eeq
valid for any $T>0$, where we defined the normalized trace of a matrix $\mathbf{R} \in \mathbb{C}^{k \times k}$ as $\langle \mathbf{R} \rangle := \frac{1}{k} \mathrm{Tr} \mathbf{R}.$

Formula \eqref{Girko} is called the \textit{logarithmic potential} method and we will use approximations of the normalized resolvent trace and of the quantity $\log | \det(H_{\zeta}-iT) |$ to estimate the left-hand side. The spectral distribution can be recovered from the Laplacian of the log-determinant of the Hermitization on the right-hand side of \eqref{Girko} or its approximation as the \textit{spectral resolution} $\eta$ goes to $0$. 

This method is based on the complex retrieval of measure identity in section \ref{wirtinger} and we will not see the details here except from the approximation of the normalized trace of the resolvent as in proposition \ref{Local-law1}. The rest of the proof can be found in the proof of Theorem 2.1 in \cite{alt2022local}, including the proof of Lemma 3.5.
\end{section}

\begin{section}{Solution of the Dyson equation}

We consider now, as in \cite{alt2022local}, the Dyson equation, which will provide an approximation of the normalized resolvent trace. Since we are working with a $4 \times 4$ matrix here the equation will be modified:
\beq \label{Dyson}
-M^{-1} = \begin{pmatrix} iH & Z \\ \overline{Z} & iH \end{pmatrix} + \mathscr{P}(M), \ \mathscr{P}\left [ (M_{ij})_{i, j \leq 4} \right] = \begin{pmatrix} M_{44}  & 0  & \tau   M_{42}   & 0 \\ 0  &   M_{33}  & 0 & \tau   M_{31} \\ \tau   M_{24}   & 0 &   M_{22}   & 0   \\ 0 & \tau   M_{13}   & 0   &   M_{11}   \end{pmatrix},
\eeq
where $\zeta \in \mathbb{C}$, $\eta>0$ and $M=M(\zeta,\eta)\in \mathbb{C}^{4 \times 4}$ is the unknown matrix. Since there are many possible solutions, we require
\begin{equation}\label{e:eqImM1}
\mathrm{Im} M = \frac{1}{2i} (M - M^*)
\end{equation} to be positive definite. We have defined
$$ Z := \begin{pmatrix} \zeta & 0 \\ 0 & {\zeta} \end{pmatrix} \  \ , \ \ iH := \begin{pmatrix} i\eta & 0 \\ 0 & i\eta \end{pmatrix} .$$
The operator $\mathscr{P}$ arises from the calculation in \eqref{self-energy2}.
\begin{prop}[Solution of the Dyson equation] \label{sol-prop}
There exists a unique solution $M$ of \eqref{Dyson} with the constraint \eqref{e:eqImM1} and it satisfies:
\beq \label{sol}
M = \begin{pmatrix} iV & \overline{B} \\ B & iV \end{pmatrix},
\eeq
where we defined
\begin{equation}\label{e:V,b}
iV := \begin{pmatrix} iv(\zeta,\eta) & 0 \\ 0 & iv(\zeta,\eta) \end{pmatrix} \ \ , \ \ B := \begin{pmatrix}  {b(\zeta,\eta)} & 0 \\ 0 & b(\zeta,\eta) \end{pmatrix}  \end{equation}
for each $\eta>0$ , $\zeta \in \mathbb{C}$ and for some $v = v(\zeta,\eta) \in (0,\infty)$ and $b=b(\zeta,\eta) \in \mathbb{C}.$
\end{prop}
\begin{proof}
Let $W = \begin{pmatrix} iH & Z \\ \overline{Z} & iH \end{pmatrix}. $ 

The claim follows from Theorem 2.1 of \cite{helton2007operator}, which demonstrates the existence and uniqueness of the solution $M$ and gives it as a limit of iterates $M_n = \mathcal{F}^n(M_0)$ with $\mathcal{F}(M_j) = - (W + \mathscr{P}(M_j))^{-1}$. We see that if we take $M_0$ of the form \eqref{sol}, it remains of the form \eqref{sol} upon application of $\mathcal{F}$, thus $M$ is of the form \eqref{sol}. Indeed, using a Schur complement formula for inversion and noting that all the 2 $\times$ 2 blocks commute, we see that for $M_j$ of the form \eqref{sol}, explicitly, we compute
\begin{align}
&\mathcal{F}(M_j)_{11} = i \frac{v+\eta}{(v+\eta)^2 + |\zeta +\tau b|^2} = \mathcal{F}(M_j)_{33} \\
&\mathcal{F}(M_j)_{13} = \frac{-\zeta - \tau b}{(v+\eta)^2 + |\zeta +\tau b|^2} = \overline{\mathcal{F}(M_j)_{31}}.
\end{align}
This completes the proof of \ref{sol-prop}.
\end{proof}
Due to $\eqref{e:V,b}$ and \eqref{sol}, we can take partial normalized traces on \eqref{Dyson}  (see definition after equation \eqref{Delta-eq}) to recover a $2 \times 2$ Dyson Equation, which we will work with henceforth and which matches the equation in \cite{alt2022local}:
\beq \label{Dyson2}
-M^{-1} = \begin{pmatrix} i\eta & \zeta \\ \overline{\zeta} & i\eta \end{pmatrix} + \mathscr{P}_2(M), \ \mathscr{P}_2\left [ (M_{ij})_{1 \leq i, j \leq 2} \right] = \begin{pmatrix} M_{22}  & \tau   M_{21} \\ \tau   M_{12}  &   M_{11}   \end{pmatrix},
\eeq
where $\zeta \in \mathbb{C}$, $\eta>0$ and $M=M(\zeta,\eta)\in \mathbb{C}^{2 \times 2}$ is the unknown matrix with
\begin{equation}\label{e:eqImM}
\mathrm{Im} M = \frac{1}{2i} (M - M^*) \succeq 0.
\end{equation}
As the equation is identical to the equation in \cite{alt2022local}, the solution will satisfy all the same properties. Thus we can import without changes the following lemma and identities from \cite{alt2022local}:
\pagebreak
\begin{lem} [Lemma 4.1 in \cite{alt2022local}, Basic estimates of $M$] \label{EST} We have the following estimates:
\begin{enumerate}[label=(\roman*)]
    \item  Uniformly for all $\eta>0$ and $\zeta \in \mathbb{C}$, we have that:
\beq \label{est(i)}
\|M\| \leq \min\{1,\eta^{-1}\} \leq 2 (1+\eta^{-1}),
\eeq
\beq \label{est(ii)}
\|M^{-1}\| \lesssim 1+ \eta + |\zeta|.
\eeq
\item Let $r>0$. Then uniformly for all $\eta>0$ and $\zeta \in \mathbb{D}_r$, we have:
\begin{align} \label{est3}
M &= i\eta^{-1} + \mathcal{O}_r(\eta^{-2}), \\
M^* M &\succsim_r (1+\eta)^{-2}, \\
\mathrm{Im} M &\succsim_r \eta (1+\eta)^{-2}
\end{align}
\end{enumerate}
and in particular, \beq \|M\|^2 = v^2 + |b|^2 = \frac{v}{v+\eta}.\eeq
\end{lem}
Here the $\lesssim$ symbol means that a quantity is less than another up to a constant. The symbol $\lesssim_{a}$ means that a quantity is less than another up to a constant that may depend on $a$.

Furthermore, we import equation (4.7) from \cite{alt2022local}, which gives the following expression for $b$:
\begin{equation}\label{e:b}
b = -\frac{\mathrm{Re} \zeta}{1+ \tau+ \eta/v} + i\frac{\mathrm{Im} \zeta}{1 - \tau + \eta/v}.
\end{equation}
The following proposition will give an estimate for the normalized trace of the Green function of the Hermitization of the Dirac matrix, by the solution matrix of the Dyson equation, the proof of which will be the theme of section \ref{Local}:
\begin{prop}[Local law for $\mathbf{H}_{\zeta}$, averaged version] \label{Local-law1} Let $v$ be defined as in \eqref{sol} and let $\gamma , \delta, \kappa > 0$. Then for any $\epsilon>0$ and $\nu>0$, there exists $C_{\epsilon,\nu}>0$ such that:
$$ \mathbb{P} \bigg ( \left | \langle \mathbf{G}(\zeta,\eta) \rangle - i v(\zeta,\eta) \right |  \leq \frac{N^\epsilon}{N\eta} \bigg ) \geq 1 - C_{\epsilon,\nu} N^{-\nu} ,$$
uniformly for all bounded $\eta \geq N^{-1+\gamma},$ ${\zeta \in S_{\tau,\delta, \kappa}} $ and $N \in \mathbb{N}.$
\end{prop}
\end{section}

\begin{section}{Local law for the Hermitization matrix} \label{Local}

In this section we prove Proposition \ref{Local-law1}.

To express the approximation $\mathbf{G} \stackrel{(N \to \infty)}{\approx} \mathbf{M}$ we introduce as in \cite{alt2022local} some appropriate norms. For any random matrix $R\in \mathbb{C}^{l \times l}$ in dimension $l \in \mathbb{N}$ we define the $p$-norms
$$ \|R\|_p:=\|R\|_p^{iso} := \sup \limits_{\|x\|,\|y\|\leq 1} ( \mathbb{E} | \langle x, Ry \rangle|^p )^{1/p} , $$
\beq \label{norms}
\|R\|_p^{avg} := \sup \limits_{\|W\|\leq 1} ( \mathbb{E} | \langle WR \rangle |^p)^{1/p}, \eeq
where the supremum is taken over $x,y \in \mathbb{C}^l$ and $W \in \mathbb{C}^{l \times l},$ respectively. The symbol $\langle \cdot , \cdot  \rangle$ denotes the Euclidean scalar product on $\mathbb{C}^l$ and the symbol $\langle \ \cdot \  \rangle$ the normalized trace on $\mathbb{C}^{l \times l}.$

Notice that for any finite $l\in \mathbb{N}$ the norms are equivalent to each other and also equivalent to the standard norm of the random variable $\|R\|$ where $\|\cdot\|$ denotes the operator norm. ($\|R\|$ is a positive real random variable).

The main Theorem \ref{Local-law2} states the approximation of the $4N \times 4N$ resolvent matrix $\mathbf{G}=\mathbf{G}(\zeta,\eta)=(\mathbf{H}_\zeta - i\eta)^{-1} \in \mathbb{C}^{4N \times 4N}$ to the deterministic matrix $\mathbf{M}=\mathbf{M}(\zeta,\eta)\in \mathbb{C}^{4N \times 4N}$, as $N \to \infty$. We recall that $\mathbf{M}$ is defined as $\mathbf{M} := M \otimes \mathbf{1}_{N \times N}$ which means that
\beq \label{M-matrix}
\mathbf{M}(\zeta,\eta) = \begin{pmatrix} iv(\zeta,\eta) & \overline{b(\zeta,\eta)} \\ b(\zeta,\eta) & iv(\zeta,\eta) \end{pmatrix}
\eeq
where every entry in this $2 \times 2$ block structure is a multiple of the identity matrix $\mathbf{1}_{2N \times 2N},$ where $v(\zeta,\eta)$ and $b(\zeta,\eta)$ are defined as in \ref{sol}.
\pagebreak
\begin{thm}[Local law for the Hermitization matrix $\mathbf{H}_\zeta$, main version] \label{Local-law2} Let $\gamma \in (0,1)$ and $p \in \mathbb{N}$. Uniformly, for all bounded $\eta \geq N^{-1+\gamma}$ and $\zeta \in E_{\rho,\gamma},$ the following local law holds:
\beq \label{local-law-eq}
\| \mathbf{G} - \mathbf{M} \|_p^{iso} \lesssim_{p , \gamma} \frac{N^\gamma}{\sqrt{N\eta}} \ \ , \ \ \  \| \mathbf{G} - \mathbf{M} \|_p^{avg} \lesssim_{p,\gamma} \frac{N^\gamma}{N\eta} .
\eeq
\end{thm}
\textit{Proof of Proposition \ref{Local-law1}.} By the definition of $\mathbf{M}$ in \eqref{M-matrix} we have that $\langle \mathbf{M} \rangle = iv.$ Thus, for suitably chosen $\gamma$ and $p$, we get as in \cite{alt2022local} that
\begin{align*}
&\mathbb{P} \bigg ( \left | \langle \mathbf{G}(\zeta,\eta) \rangle - iv(\zeta,\eta) \right | \geq \frac{N^\epsilon}{N\eta} \bigg ) \leq \frac{(N\eta)^p}{N^{\epsilon p}} \E \left |  \langle \mathbf{G}(\zeta,\eta) \rangle - iv(\zeta,\eta) \right |^p \leq  \\
&\leq \frac{(N\eta)^p}{N^{\epsilon p}} \left ( \| \mathbf{G} - \mathbf{M} \|_p^{avg} \right )^p ,
\end{align*}
by Markov's inequality and the result follows.

Here we prove Theorem \ref{Local-law2} in the local regime $\eta \ll 1$ so as to get our universal local spectral distribution for the Dirac matrix.
\begin{proof}[Proof of Theorem \ref{Local-law2}] As in \cite{alt2022local} the proof follows from the fact that $\mathbf{G}$ approximately satisfies the matrix Dyson equation:
\beq \label{Dyson-eq}
1 + ( i\eta + \mathbf{Z} + \bm{S} \mathbf{M} ) \mathbf{M} = 0,
\eeq
where each constant is a multiple of the identity matrix $\boldsymbol{1}_{4N \times 4N}$, $\mathbf{Z} := - \E \mathbf{H}_{\zeta} $, and the unknown matrix $\mathbf{M}$ will approximate the resolvent matrix $\mathbf{G}$ in $\mathbb{C}^{4N \times 4N}.$ Here $\bm{S}$ is the natural extension of $\mathscr{P}$ in \eqref{Dyson} to $\mathbb{C}^{4N \times 4N}$, i.e. for
\begin{equation}
\mathbf{A} = \begin{pmatrix} \mathbf{A}_{11}& \mathbf{A}_{12} & \mathbf{A}_{13} & \mathbf{A}_{14} \\
\mathbf{A}_{21}& \mathbf{A}_{22} & \mathbf{A}_{23} & \mathbf{A}_{24} \\
\mathbf{A}_{31}& \mathbf{A}_{32} & \mathbf{A}_{33} & \mathbf{A}_{34} \\
\mathbf{A}_{41}& \mathbf{A}_{42} & \mathbf{A}_{43} & \mathbf{A}_{44} 
\end{pmatrix} \in \mathbb{C}^{4N \times 4N} ,
\end{equation}
\beq \label{dyson-comp}
\mathbf{Z} = \begin{pmatrix} 0  & \zeta \boldsymbol{1}_{2N \times 2N} \\ \overline{\zeta} \boldsymbol{1}_{2N \times 2N} & 0  \end{pmatrix} , \  \bm{S} \mathbf{A} = \begin{pmatrix} \langle \mathbf{A}_{44} \rangle  & 0 & \tau \langle \mathbf{A}_{42} \rangle  & 0 \\ 0 & \langle \mathbf{A}_{33}\rangle  & 0 & \tau \langle \mathbf{A}_{31} \rangle  \\ \tau \langle \mathbf{A}_{24} \rangle  & 0 & \langle \mathbf{A}_{22} \rangle   & 0 \\ 0 & \tau \langle \mathbf{A}_{13} \rangle  & 0 & \langle \mathbf{A}_{11} \rangle  \end{pmatrix} ,
\eeq
where each entry is a multiple of the $N \times N$ identity matrix. The matrix $\mathbf{M}$ defined in \eqref{M-matrix} solves \eqref{Dyson-eq}.

As in \cite{alt2022local} and \cite{erdHos2019random}, we introduce the \textit{self-energy} operator:
$$ \bm{\hat{S}}: \mathbb{C}^{4N \times 4N} \mapsto \mathbb{C}^{4N \times 4N} \ \ \text{by} $$
\beq \label{self-energy} \bm{\hat{S}} \mathbf{A} = \E \left \{ ( \mathbf{H}_\zeta + \mathbf{Z} ) \mathbf{A} (\mathbf{H}_\zeta + \mathbf{Z} ) \right \} .\eeq
We show that $\bm{\hat{S}} = \bm{{S}} $ which justifies the form that we have for the Dyson equation \eqref{Dyson-eq}. Careful calculation of the self-energy operator for a general matrix $\mathbf{A}$ shows that:
\begin{align*}
\bm{\hat{S}} \mathbf{A} &= \E \{ ( \mathbf{H}_\zeta + \mathbf{Z} ) \mathbf{A} (\mathbf{H}_\zeta + \mathbf{Z} ) \} = \E \left \{ \begin{pmatrix} 0 & \mathbf{Y_N} \\ \mathbf{Y_N}^* & 0 \end{pmatrix} \begin{pmatrix} \mathbf{A}_{11} & \mathbf{A}_{12} \\ \mathbf{A}_{21} & \mathbf{A}_{22} \end{pmatrix} \begin{pmatrix} 0 & \mathbf{Y_N} \\ \mathbf{Y_N}^* & 0 \end{pmatrix} \right \}
\\
&= \E \begin{pmatrix}
\mathbf{Y_N} A_{22} \mathbf{Y_N}^* & \mathbf{Y_N} A_{21} \mathbf{Y_N} \\
\mathbf{Y_N}^* A_{12} \mathbf{Y_N}^* & \mathbf{Y_N}^* A_{11} \mathbf{Y_N}
\end{pmatrix},
\end{align*}
where $\mathbf{A}_{ij}$ are $2N \times 2N$ matrices for $i,j=1,2$. We now calculate each term in the resulting block matrix. We only show what happens in the first and second block entry and the rest calculations will be similar:
\begin{align*}
&\E [\mathbf{Y_N} \mathbf{A}_{22} \mathbf{Y_N}^*] = \E \left \{ \begin{pmatrix} 0 & X_1 \\ X_2^* & 0 \end{pmatrix} \begin{pmatrix} [\mathbf{A}_{22}]_{11} & [\mathbf{A}_{22}]_{12} \\ [\mathbf{A}_{22}]_{21} & [\mathbf{A}_{22}]_{22} \end{pmatrix} \begin{pmatrix} 0 & X_2 \\ X_1^* & 0 \end{pmatrix}\right \} \\
&=
 \begin{pmatrix} \langle[\mathbf{A}_{22}]_{22} \rangle & [\mathbf{A}_{22}]_{21}^T \left [ (1+\tau)\E p_{11}^2 - (1-\tau) \E q_{11}^2 \right ]  \\ [\mathbf{A}_{22}]_{12}^T \left [ (1+\tau)\E p_{11}^2 - (1-\tau) \E q_{11}^2 \right ] & \langle [\mathbf{A}_{22}]_{11} \rangle \end{pmatrix}\\
&=
 \begin{pmatrix} \langle[ \mathbf{A}_{22}]_{22} \rangle  & 0 \\ 0 & \langle [\mathbf{A}_{22}]_{11} \rangle \end{pmatrix},
\end{align*}
where $[\mathbf{A}_{22}]_{kl}$ are $N \times N$ matrices for $k,l=1,2$. We used the first assumption \ref{a1} for the distribution of the matrix entries in the last equality. We now compute the second block:
\begin{align*}
&\E [\mathbf{Y_N} \mathbf{A}_{21} \mathbf{Y_N}] = \E \left \{ \begin{pmatrix} 0 & X_1 \\ X_2^* & 0 \end{pmatrix} \begin{pmatrix} [\mathbf{A}_{21}]_{11} & [\mathbf{A}_{21}]_{12} \\ [\mathbf{A}_{21}]_{21} & [\mathbf{A}_{21}]_{22} \end{pmatrix} \begin{pmatrix} 0 & X_1 \\ X_2^* & 0 \end{pmatrix}\right \}
\\
&= \begin{pmatrix} \tau \langle  [\mathbf{A}_{21}]_{22} \rangle & [\mathbf{A}_{21}]_{21}^T (1+\tau)\E p_{11}^2 + [\mathbf{A}_{21}]_{21}^T (1-\tau) \E q_{11}^2 \\  [\mathbf{A}_{12}]_{12}^T (1+\tau)\E p_{11}^2 + [\mathbf{A}_{12}]_{12}^T (1-\tau) \E q_{11}^2 & \tau \langle [\mathbf{A}_{21}]_{11} \rangle \end{pmatrix} \\
&=
 \begin{pmatrix} \tau \langle  [\mathbf{A}_{21}]_{22} \rangle & 0\\ 0 & \tau \langle [\mathbf{A}_{21}]_{11} \rangle \end{pmatrix},
\end{align*}
where $[\mathbf{A}_{21}]_{kl}$ are $N \times N$ matrices for $k,l=1,2$ and $A^T$ denotes the transpose matrix of the matrix $A$.

The result will be
\beq \label{self-energy2}
\bm{\hat{S}} \mathbf{A} = \begin{pmatrix} \langle \mathbf{A}_{44} \rangle & 0& \tau\langle \mathbf{A}_{42} \rangle & 0 \\ 0 & \langle \mathbf{A}_{33}\rangle & 0 & \tau \langle \mathbf{A}_{31} \rangle \\  \langle \mathbf{A}_{24} \rangle & 0 & \langle \mathbf{A}_{22} \rangle & 0\\ 0 & \tau \langle \mathbf{A}_{13} \rangle & 0& \langle \mathbf{A}_{11} \rangle \end{pmatrix},
\eeq
with each entry being a multiple of the $N \times N$ identity matrix, which coincides with $\bm{S} \mathbf{A}$.

We continue with the proof and mark that \eqref{Dyson-eq} is an approximation of the resolvent identity:
\beq \label{peturb}
1 + ( i\eta + \mathbf{Z} + \bm{S} \mathbf{G} ) \mathbf{G} = \mathbf{D}, \quad \mathbf{D}:= (\mathbf{H}_\zeta +\mathbf{Z} + \bm{S}\mathbf{G}) \mathbf{G} ,
\eeq
which can be regarded as a perturbation of \eqref{Dyson-eq} with error matrix $\mathbf{D}$. 

The following theorem from \cite{erdHos2019random} gives bounds for the error matrix $\mathbf{D}$ when $\bm{S}$ is given by the self-energy operator $\bm{\hat{S}}$ of our random matrix model:
\pagebreak
\begin{thm}[Bound on the error matrix, from \cite{erdHos2019random}, Theorem 4.1] \label{t:errorbounds}
Under our assumptions \ref{a2} for our random matrix model, we have the following for the error matrix of the Dyson equation. Let $\epsilon>0$ and $p \in \mathbb{N}.$ There exists a constant $C_*>0$ such that, uniformly for $\eta \in [\frac{1}{N},1],$ we have the following bounds on the error matrix:
\begin{align} \label{error-bounds}
    \|\mathbf{D}\|_p^{iso} &\lesssim_{p,\epsilon} N^\epsilon \sqrt{\frac{\|\mathrm{Im}\mathbf{G}\|_q}{N\eta}}(1+\|G\|_q)^{C_*}\left (1+\frac{\|\mathbf{G}\|_q}{N^{1/4}} \right )^{C_*p} \\
    \|\mathbf{D}\|_p^{avg} &\lesssim_{p,\epsilon} N^\epsilon {\frac{\|\mathrm{Im}\mathbf{G}\|_q}{N\eta}}(1+\|G\|_q)^{C_*}\left (1+\frac{\|\mathbf{G}\|_q}{N^{1/4}} \right )^{C_*p},
\end{align}
with $q=C_*p^4/\epsilon,$ where $\mathbf{D}$ is defined by equation \eqref{peturb}.
\end{thm}
With the help of this proposition we are able to prove the local law \eqref{local-law-eq} for the Hermitization matrix $\mathbf{H}_{\zeta}.$ By subtracting \eqref{Dyson-eq} from \eqref{peturb} and defining our target matrix $\mathbf{\Delta}:= \mathbf{G} - \mathbf{M} $ we arrive at the equation:
\beq \label{Delta-eq}
\mathbf{\Delta} - \mathbf{M} (\bm{S} \mathbf{\Delta}) \mathbf{M} = \mathbf{M} (\bm{S} \mathbf{\Delta}) \mathbf{\Delta} - \mathbf{M} \mathbf{D}.
\eeq
We now use the \textit{partial trace operator}: $\mathbb{C}^{2N \times 2N} \mapsto \mathbb{C}^{2 \times 2},$ \ $\mathbf{A} \to \underline{\mathbf{A}},$ where:
\[
\underline{\mathbf{A}} := \begin{pmatrix}
    \langle \mathbf{A_{11}} \rangle & \langle \mathbf{A_{12}} \rangle \\ \langle \mathbf{A_{21}} \rangle & \langle \mathbf{A_{22}} \rangle
\end{pmatrix}, \]
whenever \[
{\mathbf{A}} = \begin{pmatrix}
    \mathbf{A_{11}} & \mathbf{A_{12}} \\ \mathbf{A_{21}} & \mathbf{A_{22}}
\end{pmatrix} ,
\]
with each entry $A_{ij}$ being an $N \times N$ block matrix. This allows us to reduce our equation from $\mathbb{C}^{4N \times 4N}$ to $\mathbb{C}^{4 \times 4}$ and deduce a first bound for the matrix $\Delta:= \underline{\mathbf{\Delta}}.$ We arrive at the equation:
\beq \label{Delta-44}
\mathscr{L} \Delta = M (\mathscr{P} \Delta) \Delta - M D,
\eeq
where $D:= \underline{\mathbf{D}},$ and $\mathscr{L}:\mathbb{C}^{4 \times 4} \mapsto \mathbb{C}^{4 \times 4}$ is the linear stability operator from (4.10) of \cite{alt2022local} - see section \ref{s:stability}. From Corollary \ref{stab-est2} we have that $\mathscr{L}$ leaves a certain subspace of $\mathbb{C}^{4\times 4}$ which we call $ A_+$ invariant - see section \ref{s:stability}, and
\beq \label{kappa}
K_{\delta, \kappa} := \sup_{\zeta \in {S}_{\tau,\delta, \kappa}}\|\mathscr{L}^{-1}|_{A_+} \| \lesssim_{\delta, \kappa} 1.
\eeq
Note, importantly, that $\Delta \in A_+$ because $\langle \mathbf{G}_{11} \rangle = \langle \mathbf{G}_{22} \rangle  $ (which can be seen using the Schur formula for inverse) and $\langle \mathbf{M}_{11} \rangle = \langle \mathbf{M}_{22} \rangle = iv$ (see \eqref{M-matrix}).

Thus, we can invert $\mathscr{L}$ on $A_+$ in \eqref{Delta-44} and take the $p$-norm defined on \eqref{norms} and derive the inequality:
\[
\|\Delta\|_p ( 1- K_{\delta, \kappa} \|\Delta\|_p ) \lesssim K_{\delta, \kappa} \|D\|_p ,
\]
where we used the comparability of the norms for $l=2.$ This means that if we choose $c$ as a small enough constant, then:
\beq \label{Delta-ineq}
\| \Delta \mathbf{1}( \|\Delta\|_p \leq {c}/{K_{\delta, \kappa}}) \|_p \lesssim K_{\delta, \kappa} \| D \|_p .
\eeq
We also note two important inequalities that connect $D$ with $\mathbf{D}$ and $\Delta$ with $\mathbf{\Delta}$. We have that:
\beq \label{D-D}
\|D\|_p \lesssim \|\mathbf{D}\|_p^{avg}
\eeq
and
\beq \label{Delta-Delta}
\|\mathbf{\Delta}\|_p^{\#} \lesssim \|\Delta\|_p + \|\Delta\|_{2p} \|\mathbf{\Delta}\|_{2p}^{\#} + \|\mathbf{D}\|_{p}^{\#},
\eeq
which comes from \eqref{Delta-eq} after taking norms and using Cauchy-Schwartz as well as the definition of $\bm{S}.$ This relationship between norms of $\mathbf{\Delta}$ and $\Delta$ is analogous to (5.22) in \cite{alt2022local} and is  the final ingredient that is needed to complete the proof of Lemma \ref{Boot}, which yields the next Proposition \ref{Local-law1}. The proof of Lemma \ref{Boot} is identical to the proof of Lemma 5.4 in \cite{alt2022local}.
\end{proof}
\begin{lem}[Bootstrapping] \label{Boot}
There is a constant $c_*>0$ depending only on the distribution of the entries of the random matrices $P$ and $Q$ such that $\|\mathbf{\Delta}\|_p \lesssim_{p,\delta,\gamma} N^{-\gamma/6}$ for all $p \in \mathbb{N}$ in $\mathbb{A}_{\delta,\gamma}$ implies $\|\mathbf{\Delta}\|_p \lesssim_{p,\delta,\gamma} N^{-\gamma/6}$ for all $p \in \mathbb{N}$ in $\mathbb{A}_{\delta,(1-c_*)\gamma}.$
\end{lem}
\pagebreak
Here we recall the main steps of the bootstrapping argument from \cite{alt2022local}. Inequality \eqref{Delta-ineq}, together with the inequalities \eqref{D-D} , \eqref{Delta-Delta} and \eqref{error-bounds} allows us to show proposition \ref{Local-law2}, i.e. $\|\mathbf{\Delta}\|_p \ll 1,$ by bootstapping from $\eta \sim 1$ to $\eta \sim N^{-1+\gamma},$ for $1 \gg \gamma>0.$ We remark a sketch of the proof and the rest of the details can be found in \cite{alt2022local}:
\begin{itemize}
    \item \underline{Implication inside induction.} Through proposition \eqref{error-bounds} a bound of the form $\|\mathbf{\Delta}\|_p \lesssim 1$ implies that
    $\|D\|_p \ll 1$ because then $\|\mathbf{G}\|_p \leq \|\mathbf{M}\|_p + \|\mathbf{\Delta}\|_p \lesssim 1 ,$ for all $p \in \mathbb{N}$ and $N\eta \gg 1.$ This in turn, through \eqref{Delta-ineq}, implies that $\|\Delta\|_p \ll 1.$ Finally we estimate that $\|\mathbf{\Delta}\|_p \ll 1$ because of \eqref{Delta-Delta}. Altogether, this argument shows that $\|\mathbf{\Delta}\|_p \lesssim 1$ implies $\|\mathbf{\Delta}\|_p \ll 1$ on all of $\mathbb{A}_{\delta,\gamma},$ for any $\gamma>0$, where we introduced the parameter set:
    $$(\zeta,\eta) \in \mathbb{A}_{\delta,\gamma}:= S_{\tau,\delta, \kappa} \times [N^{-1+\gamma},1].$$
    This implication is bootstrapped from $\eta \sim 1$ all the way to $\eta \sim N^{-1+\gamma},$ and is formulated in the following bootstrapping lemma \ref{Boot} which we import from \cite{alt2022local}. Its proof in our model is similar.

    \item \underline{Induction basis.} The regime $\eta \sim 1$ is established by the global law in \cite{erdHos2019random}, [Theorem 2.1] in combination with [Lemma 5.4.1] and the bound $\|G\| \leq \eta^{-1} \leq N $ on $\mathbb{A}_{\delta,\gamma}.$ Notice that the $M$ matrix there can be replaced by our solution matrix $\mathbf{M}$. Specifically, we get the following proposition:
    \begin{prop}[Global law] \label{global} There exists a universal constant $c>0$ such that for any $\epsilon>0:$
    \beq
    \|\mathbf{\Delta}\|_p \lesssim \frac{N^{\epsilon}}{\sqrt{N\eta}}, \ \ \  \|\mathbf{\Delta}\|_p^{avg} \lesssim \frac{N^{\epsilon}}{N\eta}
    \eeq
    \end{prop}
    for any bounded $\zeta$ and $\eta \in [N^{-c\epsilon},1].$
\end{itemize}
\end{section}

\begin{section}{Stability of the Dyson equation} \label{s:stability}

To relate $\Delta$ and $\mathbf{\Delta}$ in the boostrapping argument above, we will need to show the invertibility of the \textit{stability operator} defined by:
\begin{align*} \label{stabop} 
&\mathscr{L}: \mathbb{C}^{4 \times 4} \mapsto \mathbb{C}^{4 \times 4} \\
\mathscr{L} &(R) := R - M \mathscr{P}(R) M \numberthis
\end{align*}
Specifically, we will need bounds on the norm of the inverse of this operator, in a subspace of $\mathbb{C}^{4 \times 4}$ where it can be inverted.

We introduce the $4 \times 4$ matrix:
\[
E_- :=  \begin{pmatrix}
I_{2 \times 2} & {0} \\ {0} & - I_{2 \times 2} \end{pmatrix}
\]
We will work on the subspace $A_+ := E_{-}^{\perp}$ orthogonal to this matrix where orthogonality is understood with respect to the Hilbert-Schmidt inner product on $\mathbb{C}^{4 \times 4}.$ We note that:
\[
A_+ = \left \{ \begin{pmatrix}
    A & B \\ C & D
\end{pmatrix} \in \mathbb{C}^{4\times 4} \ \bigg | \ \mathrm{Tr}A = \mathrm{Tr}D \right \}
\]
We write $\mathscr{L}^{-1} |_{A_+}$ to denote the restriction of the inverse of  $\mathscr{L}$ in the domain $A_+$, i.e. the operator is first restricted on $A_+$ and then inverted. In the proof of this proposition we will use $\|.\|$ for the operator norm.
\begin{prop}[Linear stability estimate] \label{stab-est1} For any $\zeta \in \mathbb{C}$ and $\eta >0 ,$ the operator $ \mathscr{L}$ leaves $A_+$ invariant and is invertible on that subspace. Let
\begin{equation}\label{alpha}
\alpha := 1 - |v^2 - |b|^2|
\end{equation}
The norm of the inverse of the restriction operator on $A_+$ satisfies:
\beq \label{stabest1}
\| \mathscr{L}^{-1} |_{A_+} \| \lesssim \alpha^{-2},
\eeq
uniformly for all bounded $\zeta$ and $\eta \in (0,1]$.
\end{prop}
\begin{proof}
As in \cite{alt2022local}, the linear stability operator $\mathscr{L}$ can be written in this form $\mathcal{U}-\mathcal{T},$ where $\mathcal{U}$ is a unitary matrix and $\mathcal{T}$ a self-adjoint matrix. From \eqref{sol}, we deduce that the solution matrix $M$ is normal and therefore there exists a polar decomposition, $M=|M|U=U|M|,$ for some unitary matrix $U.$ By direct computation, we find that:
\[
|M|^2 = M^*M = \|M\|^2 \cdot I_4,\hspace{3mm} \|M\|^2= v^2+|b|^2 = \frac{v}{\eta+v}
\]
and so we have that $U$ should be equal to:
\[
U = (v^2 + |b|^2)^{-1/2} M .
\]
We write the linear stability operator as
$ \mathscr{L} = \mathcal{U}^*(\mathcal{U} - \|M\|^2\mathscr{P}),$ where we defined $\mathscr{P}$ as in \eqref{Dyson} and $\mathcal{U}$ is defined as $\mathcal{U}R = U^* R U^*,$ for any $R \in \mathbb{C}^{4 \times 4}.$ Note that $\mathcal{U}$ and $\mathscr{P}$ leave $A_+$ invariant and $\mathscr{P}$ is self-adjoint whereas $\mathcal{U}$ is unitary. We observe that:
\[
\mathrm{Spec}(\mathscr{P}|_{A_+}) = \left \{ 1,-1,0, \tau, -\tau \right \},
\]
with the following eigenspaces: 
\begin{itemize}
    \item $ P_1 = \left \langle \begin{pmatrix}
        1 & 0 & 0 & 0 \\ \ 0 & 0 & 0 & 0 \\ 0 & 0 & 0 & 0 \\ 0 & 0 & 0 & 1
    \end{pmatrix}   ,
    \begin{pmatrix}
        0 & 0 & 0 & 0 \\ \ 0 & 1 & 0 & 0 \\ 0 & 0 & 1 & 0 \\ 0 & 0 & 0 & 0
    \end{pmatrix}  \right \rangle $
    \item $ P_{-1} = \left \langle \begin{pmatrix}
        1 & 0 & 0 & 0 \\ \ 0 & -1 & 0 & 0 \\ 0 & 0 & 1 & 0 \\ 0 & 0 & 0 & -1
    \end{pmatrix} \right \rangle $
\end{itemize}
Hence, for $\mathcal{T} :=\|M\|^2 \mathscr{P},$ we have that $\|\mathcal{T}\|= \|M\|^2 \cdot \|\mathscr{P}\| = \|M\|^2 = \frac{v}{\eta+v}\leq 1$ .
To complete the proof, we will take inspiration from \cite{ajanki2017singularities}, (Lemma 5.8) and \cite{alt2022local}, (Lemma 4.5). In our case, the Hermitian operator has no spectral gap. Instead it has a double eigenvalue at 1 and a non-degenerate eigenvalue at $-1$. However, using a projection onto the span of the eigenvectors corresponding to $\pm 1$ instead of a projection onto the top eigenvector in the case of a non-trivial spectral gap and performing explicit computations for our particular operator $\UU$, we are able to recover necessary estimates on the inverse of the stability operator. 

We will show that \begin{equation}
\| (\mathcal{U} - \mathcal{T})^{-1} \| \leq C \left(\min \left \{ (1 - v^2 + |b|^2)^2, (1+v^2-|b|^2)^2\right \}  \right )^{-1} ,
\end{equation}
which is equivalent to proving that
\begin{equation}
\|(\UU-\TT)w\| \geq c \min \left \{(1 - v^2 +|b|^2)^2, (1+v^2-|b|^2)^2\right \}\|w\|
\end{equation}
for some constant $c> 0,$ for any unit vector $w$. Let $w \in \mathbb{C}^{4 \times 4}$ with $\|w\| = 1$. We will decompose $w$ according to the spectral projections of $\mathcal{T}$,
\begin{equation}\label{e:wproj}
w = P_1w + P_2w
\end{equation}
where $P_1$ is the projection onto $span(h_1, ..., h_k)$ and $P_2$ is the projection onto its orthogonal complement. The matrices $h_1,...,h_k$ are defined as the eigenvectors corresponding to the eigenvalues $\pm \|M\|^2$ of $\mathcal{T}$ which correspond to the eigenvalues $\pm 1$ of $\mathscr{P}.$ Notice that $k=3$.

We will consider 2 cases:
\begin{enumerate}
\item[Case 1:] $  \|P_2w\|^2 \geq c \alpha^2$
\item[Case 2:] $  \|P_2w\|^2 <  c \alpha^2$ ,
\end{enumerate}
for some suitably small $c>0$.

\pagebreak
In Case 1,
using triangle inequality we obtain
\begin{align*}
&\|(\UU - \TT)w\| \geq \|w\| - \|\TT w\| = 1 -(\|\TT P_1w\|^2 + \|\TT P_2w\|^2)^{1/2} \\ 
&\geq 1 -(\|\TT\|^2 \|P_1w\|^2 + \|\TT P_2w\|^2)^{1/2}
\end{align*}
and using that $1 - \sqrt{1-y} \geq y/2$ for any $y \in [0, 1]$ we obtain:
\begin{align*}
2\|(\UU - \TT)w\| &\geq 1 - \|\TT\|^2 \|P_1w\|^2 - \|\TT P_2w\|^2 \\
&= 1 - \|\TT\|^2 \|P_1w\|^2 - \tau^2\|P_2w\|^2 \\
&\geq 1 - (\|P_1w\|^2+ \|P_2w\|^2) + (1-\tau^2)\|P_2w\|^2\\
&= (1-\tau^2)\|P_2w\|^2 . \numberthis
\end{align*}
Using the definition of Case 1 in the above inequality, we obtain:
\begin{equation}
2\|(\UU - \TT)w\| \geq (1+\tau)(1-\tau)c \alpha^2
\geq c \alpha^2.
\end{equation}.
For Case 2, we note:
\begin{align*}
\| (\UU - \TT)w\| &= \|(I - \UU^*\TT)w\|
\geq \|P_{1} (1 - \UU^* \TT)w\|
\\
&\geq  \|P_{1} (1 - \UU^* \TT)P_1w\| - \|P_{1} (1 - \UU^* \TT)P_2w\|.
\end{align*}
Looking at the first term, we obtain by direct computation:
\begin{align*}\label{e:stabComputation}
\|P_1 (1 - \UU^* \TT)P_1w\|^2 &= \left \|\sum_{i = 1}^3 \langle h_i, (1 - \UU^* \TT) \sum_{j=1}^3 \langle h_j, w \rangle h_j\rangle h_i \right \|^2 \\
&= \left \|\sum_{i,j = 1}^3 \langle h_j, w \rangle (\langle h_i, h_j\rangle - \langle h_i, \UU^* \TT  h_j\rangle) h_i \right \|^2
\\
&= \left \|\sum_{i,j = 1}^3 \langle h_j, w \rangle (\langle h_i, h_j\rangle - \lambda_j \langle h_i, \UU^*  h_j\rangle) h_i \right \|^2
\\
&= \sum_{i=1}^3 \left( \sum_{j=1}^3 \langle h_j, w \rangle (\langle h_i, h_j\rangle - \lambda_j \langle h_i, \UU^*  h_j\rangle) \right)^2
\end{align*}
We expand now everything to get that:
\begin{align*}
&\|P_1 (1 - \UU^* \TT)P_1w\|^2 \\
&=\left( \langle h_1, w \rangle (1 +v^2) - |b|^2 \langle h_2, w \rangle\right)^2 + \left( -|b|^2\langle h_1, w \rangle (1 +v^2) \langle h_2, w \rangle\right)^2 \\
&+ \langle h_3, w \rangle^2(1 - v^2 + |b|^2)^2 \\
&= \left ((1+v^2)^2+|b|^4 \right )\langle h_1, w \rangle^2 + \left ((1+v^2)^2+|b|^4 \right )\langle h_2, w \rangle^2 - 4 |b|^2 (1+v^2)\langle h_1, w \rangle\langle h_2, w \rangle \\
&+ \langle h_3, w \rangle^2(1 - v^2 +|b|^2)^2 \\
&= \left ((1+v^2-|b|^2 \right )^2 \left (\langle h_1, w\rangle^2 + \langle h_2, w \rangle^2 \right ) + 2|b|^2(1+v^2+ |b|^2)(\langle h_1, w\rangle-\langle h_2, w\rangle)^2 \\
&+ \langle h_3, w \rangle^2(1 - v^2+ |b|^2)^2 \\
&\geq \|P_1w\|^2 \min \left \{(1 - v^2+|b|^2)^2, (1+v^2-|b|^2)^2 \right \} . \numberthis
\end{align*}
Now to upper bound the second term:
\begin{align*}
\left \|\sum_{i=1}^3 \langle h_i, \UU^* \TT P_2 w\rangle h_i \right \|^2 = \sum_{i=1}^3 \langle h_i, \UU^* \TT P_2 w\rangle^2 = \sum_{i=1}^3 \langle P_2 \UU h_i, \TT P_2 w\rangle^2
\\
\leq \|\TT\|^2 \|P_2w\|^2 \sum_{i=1}^3 \|P_2 \UU h_i\|^2 
\leq
3\|\TT\|^2 c\alpha^2 \leq c \alpha^2 \numberthis
\end{align*}
From \eqref{e:b}, we notice that $0\leq \alpha \leq 1 $, and we have that:
\begin{equation} \label{e:bdP1w}
\|P_1 w\|^2 = 1 - \|P_2 w\|^2 \geq  1 - c\alpha^2 .
\end{equation}
Putting this together, we obtain that:
\begin{equation}
\| (\UU - \TT)w\| \geq |\alpha|\sqrt{1-c\alpha^2} - c|\alpha| \geq ca^2
\end{equation}
which yields the result.
\end{proof}

\begin{cor}[Linear stability estimate in the bulk] \label{stab-est2} Let $\delta, \kappa \in (0,1).$ Then uniformly, for all $\zeta \in \ S_{\tau,\delta,\kappa}$ and $\eta \in (0,1] ,$ we have:
\beq \label{stabest2}
\| \mathscr{L}^{-1} |_{A_+} \| \lesssim_{\delta, \kappa} 1
\eeq
\end{cor}
\begin{proof}
It follows from \eqref{stabest1} and \eqref{e:b}, and noting that 
\begin{align*}
1 - v^2 + |b|^2 &= \frac{\eta}{\eta + v} + 2|b|^2, \text{ and } \\
1 + v^2 - |b|^2 &= \frac{\eta}{\eta + v} + 2v^2.
\end{align*}
Notice that we have $\alpha = 0$ when $\eta \to 0$ and $u=b=0,$ i.e. when we are near the origin, that's why we exclude it from our domain when we get our local laws.
\end{proof}
\end{section}

\begin{section}{The least singular value problem} \label{lsv}

We prove here an important theorem about the least singular value of the matrix $\mathbf{Y_N} - \zeta.$

Not so surprisingly, the Dirac matrix
\[
\mathbf{Y_N} = \begin{pmatrix}
    0 & {X_1} \\ {X_2^*} & 0
\end{pmatrix}
\]
has almost elliptic correlations. This can be seen as follows. We pick elements $\mathbf{X_1}_{ij}$ and $\mathbf{X_2^*}_{ab}$ for $i,j=1,...,N$ and $a,b=1,...,N$ and compute that:
\begin{align*}
{X_1}_{ij} {X_2^*}_{ab} &= (1+\tau)P_{ij}P^*_{ab} - (1-\tau)Q_{ij}Q^*_{ab} - \sqrt{1-\tau^2}P_{ij}Q^*_{ab} + \sqrt{1-\tau^2}Q_{ij}P^*_{ab} \\
&= (1+\tau)P_{ij}\overline{P_{ba}} - (1-\tau)Q_{ij}\overline{Q_{ba}} - \sqrt{1-\tau^2}P_{ij}\overline{Q_{ba}} + \sqrt{1-\tau^2}Q_{ij}\overline{P_{ba}}
\end{align*}
If we take expectation in this expression the only non-zero terms in this correlation will appear for $b=i$ and $a=j,$ which means that:
\beq \label{corr}
N \E {X_1}_{ij} {X_2^*}_{ab} = \begin{cases} \tau , \ \ \text{if} \ \ a=j,b=i \\ 0,\ \ \text{otherwise}
\end{cases}
\eeq
The matrix $\sqrt{N}\mathbf{Y_N}$ however doesn't fully satisfy the assumptions of [\cite{nguyen2015elliptic},Theorem 1.9] with the elliptic correlation parameter $\tau \in [0,1]$, (see also \cite{gotze2015minimal}), because of the deterministic zero entries.

Due to this we are assuming a bounded density on the entries (see \ref{a3}) and follow the proof in \cite{alt2021inhomogeneous} to deduce a lower bound for the least singular value of $\mathbf{Y_N} - \zeta$ with high probability.
\begin{thm}[Smallest singular value of $\mathbf{Y_N} - \zeta.$] \label{sv}
 Let $\sigma_{2N}(\zeta)$ denote the least singular value of the random matrix $\mathbf{Y_N} - \zeta$. Then for any $B>0$ there exists $A>0$ and $C>0$ such that
\[
\mathbb{P} \left ( \sigma_{2N}(\zeta) \leq N^{-A} \right ) \leq C N^{-B},
\]
uniformly for all suitably large $N \in \mathbb{N}$ and bounded $\zeta \in \mathbb{C}.$
\end{thm}
\begin{proof}
    We will use assumption \ref{a3}. We begin with the usual estimate about the least singular value connecting it with the rows of the matrix $\mathbf{Y}_N - \zeta$. This estimate was firstly established in \cite{rudelson2008littlewood} and the general technique can be found for example in (Lemma 4.12,  \cite{bordenave2012around}) and (Proposition 7.1, \cite{alt2021inhomogeneous}):
    \[
    \sigma_{2N}(\zeta) \geq (2N)^{-1/2} \min_{i \in [2N]} \mathrm{dist}(R_i, R_{-i}),
    \]
    where $R_1,...,R_{2N}$ are the rows of $\mathbf{Y}_N - \zeta$. Continuing with a union bound, we get that, for $0<u<1$:
    \begin{equation} \label{union}
    \mathbb{P}(\sigma_{2N}(\zeta) \leq u) \leq 2N \max_{i \in [2N]} \mathbb{P}[(2N)^{-1/2} \mathrm{dist} (R_i, R_{-i}) \leq u ]
    \end{equation}
    We fix $i \in [2N]$ and pick a unit vector $y$ orthogonal to $R_{-i}$ and measurable with respect to $A_i:=\{ R_j \ | \ j \neq i \}.$ We remind that $R_{-i}:= \mathrm{span}({A_i}).$
    We notice that:
    \[
    \mathrm{dist}(R_i,R_{-i}) = \| \pi_i(R_i)\| \cdot \|y\| \geq |\langle R_i,y \rangle|,
    \]
    by the Cauchy-Schwartz inequality, where $\pi_i$ is the orthogonal projection onto the orthogonal complement of $R_{-i}.$ Since $y$ is normalized there exists $j\in [2N]$ such that $|y_j| = \max \limits_k |y_k| \geq (2N)^{-1/2}$. We then estimate that:
    \begin{equation} \label{expect-prob}
    \mathbb{P}\left(\left|\left\langle R_i, y\right\rangle\right| \leq u \sqrt{2N} \right) =\mathbb{E}\left [ \mathbb{P}\left(\left|\left\langle R_i, y\right\rangle \right| \leq u \sqrt{2N} \bigg | y \right )  \right].
    \end{equation}
    We now estimate the inequality $\left|\left\langle R_i, y\right\rangle\right| \leq u \sqrt{2N}$ as:
    \begin{align*}
        \left|\left\langle R_i, y\right\rangle\right| = \left | \sum \limits_{k} \overline{x_{ik}} y_k - \overline{\zeta}y_i  \right |,
    \end{align*}
     where the sum is over $N$ indices and $x_{ik}$ are elements of either $X_1$ or $X_2^*.$  
     
     To deduce the bound $u \sqrt{2N}$ it is enough that:
     \begin{equation*}
     \left |\overline{x_{ik}}y_k - \frac{\overline{\zeta}}{N} y_i \right | \leq \frac{1}{N} u\sqrt{2N}, \quad \text{for each} \ k,
     \end{equation*}
     or:
    \begin{equation*}
    |y_k| \left |x_{ik} - \frac{\zeta}{N}\frac{\overline{y_i}}{\overline{y_k}} \right | \leq \frac{1}{N} u\sqrt{2N} ,  \quad \text{for each} \ k \ \text{with} \ y_k \neq 0.
    \end{equation*}
    But,
    \[
    |y_k| \left |x_{ik} - \frac{\zeta}{N}\frac{\overline{y_i}}{\overline{y_k}} \right | \leq |y_j| \left |x_{ik} - \frac{\zeta}{N}\frac{\overline{y_i}}{\overline{y_k}} \right |,
    \]
    so it is enough that:
    \beq \label{enough}
    \left |\sqrt{N}x_{ik} - \frac{\zeta}{\sqrt{N}}\frac{\overline{y_i}}{\overline{y_k}} \right | \leq \frac{\sqrt{2}u}{|y_j|} ,  \quad \text{for each} \ k \ \text{with} \ y_k \neq 0,
    \eeq
    Thus,
    \begin{align*}
        &\mathbb{P}\left(\left|\left\langle R_i, y\right\rangle \right| \leq u \sqrt{2N} \bigg | y \right ) \leq \sum \limits_k \mathbb{P}\left(\left| \sqrt{N}x_{ik} - \frac{\zeta}{\sqrt{N}}\frac{\overline{y_i}}{\overline{y_k}} \right| \leq \frac{\sqrt{2}u}{|y_j|}  \right ) \\ 
        &= \sum \limits_k \int_{B \left (\frac{\zeta}{\sqrt{N}}\frac{\overline{y_i}}{\overline{y_k}}, \frac{\sqrt{2}u}{|y_j|} \right )} \psi(z)  d^2z \leq N \left ( \pi \frac{2u^2}{|y_j|^2} \right )^{\frac{q-1}{q}} \| \psi \|_q \\
        &\leq N^{\kappa+1} \left ( \frac{\pi u^2}{N} \right )^{\frac{q-1}{q}} = \pi N^{\kappa - \frac{1}{q}} u^{2 - \frac{2}{q}}.
    \end{align*}
    From \eqref{union}, we deduce that, for every $u \in (0,1):$
    \begin{equation} \label{sing1}
        \mathbb{P} \left (\sigma_{2N}(\zeta) \leq u \right )  \leq 2\pi N^{\kappa + 1- \frac{1}{q}} u^{2 - \frac{2}{q}}.
    \end{equation}
    Let $B>0.$ We pick $A>0$ and $q>1$ such that:
    \[
    \kappa + 1 - \frac{1}{q} - 2A + \frac{2A}{q} \leq -B,
    \]
    or:
    \[
    A \geq \frac{B+\kappa+1-\frac{1}{q}}{2-\frac{2}{q}},
    \]
    and then choose $u = N^{-A} \in (0,1)$ to get the result we want.
\end{proof}
\end{section}

\begin{section}{Universality results}

We now perform the change of variables to deduce the local law for the correlated covariance matrix. We follow the linearization of a random matrix technique from \cite{o2015products}.
\begin{proof}[Proof of Theorem \ref{main'}]
For the Dirac matrix $\mathbf{Y_N}$, it holds that:
\[
\mathbf{Y^2_N} = \begin{pmatrix}
    X_1 X_2^* & 0 \\ 0 & X_2^*X_1 .
\end{pmatrix}
\]
Notice that $X_1X_2^*$ has the same eigenvalues with $X_2^* X_1 $. Define now a test function $f: \mathbb{C}\mapsto \mathbb{C}$ with corresponding zoom function $f_{\widehat{\zeta_0},a}$ for $\widehat{\zeta_0} \in \widehat{S}_{\tau,\delta,\kappa}$ and $a\in (0,1/2)$ such that $\|\Delta f\|_{L^{2+a}} \leq N^D \|\Delta f \|_{L^1}$ for some $D \in \mathbb{N}$. Observe that:
\begin{align*}
\int_{\mathbb{C}} f_{\widehat{\zeta_0},a} (z) d \widehat{\mu}_N &= \frac{1}{N}  \sum \limits_{\zeta \in Spec(X_1X_2^*)} f_{\widehat{\zeta_0},a} (\zeta) \\
&= \frac{1}{2N} \sum \limits_{\zeta \in Spec(\mathbf{Y^2_N})} f_{\widehat{\zeta_0},a} (\zeta)\\
&= \int_{\mathbb{C}} f_{\widehat{\zeta_0},a} (z^2) d\mu_N (z),
\end{align*}
where $\widehat{\mu}_N$ is the counting measure for the correlated covariance matrix, whereas $\mu_N$ is the counting measure for the Dirac matrix. By Theorem \ref{main} we get the local weak convergence $\mu_N \to \mu$, where $\mu$ is the measure for the spectrum of the Dirac matrix, which means that:
\beq \label{prob-ineq}
\mathbb{P} \left ( \left | \frac{1}{N} \sum_{\zeta \in Spec(X_1X_2^*) }f_{\widehat{\zeta_0},a}( \zeta)   - \int_{\mathbb{C}} f_{\widehat{\zeta_0},a} (z^2) d\mu(z) \right | \leq N^{-1+2a+\epsilon} \right ) \geq 1 - CN^{-\nu},
\eeq
for any $\epsilon>0,\ \nu \in \mathbb{N}$ and some suitable $C>0$.

We now perform the change of variables $z^2 \mapsto z.$ The Jacobian of the transformation is equal to:
\beq \label{Jacob}
|J(z)| = \left | \frac{d}{dz} (\sqrt{z}) \right |^2 = \frac{1}{4|z|} .
\eeq
We are left to prove the transformation from the centered ellipse to the shifted one. Indeed, by firstly expanding the elliptic spectral domain of the Dirac matrix to an equivalent form and then introducing complex variables $z=x+iy$ we arrive at the desired domain:
\begin{align*}
    &\frac{x^2}{(1+\tau)^2} + \frac{y^2}{(1-\tau)^2} - 1 \leq 0 \Leftrightarrow \\
    &\left [ \frac{x^2}{(1+\tau)^2} + \frac{y^2}{(1-\tau)^2} - 1 \right ] \left [ \frac{x^2}{(1-\tau)^2} + \frac{y^2}{(1+\tau)^2} + 1 \right ] \leq 0 \Leftrightarrow \\
    &\frac{x^4}{(1-\tau^2)^2} + \frac{y^4}{(1-\tau^2)^2} + \frac{x^2 y^2}{(1+\tau)^4} + \frac{x^2y^2}{(1-\tau)^4} + \frac{x^2 - y^2}{(1+\tau)^2} - \frac{x^2-y^2}{(1-\tau)^2} \leq 1 \Leftrightarrow \\
    &\frac{x^4+y^4}{(1-\tau^2)^2} + \frac{(1+\tau)^4+(1-\tau)^4}{(1-\tau^2)^4}x^2y^2 + \frac{(1-\tau)^2 - (1+\tau^2)}{(1-\tau^2)^2} (x^2-y^2) \leq 1 \Leftrightarrow \\
    &x^4+y^4 - 4\tau(x^2-y^2) + 2\frac{1+6\tau^2 +\tau^4}{(1-\tau^2)^2}x^2y^2 \leq (1-\tau^2)^2 \Leftrightarrow \\
    &x^4+y^4 - 4\tau (x^2-y^2) + 2\frac{(1-\tau^2)^2 + 8\tau^2}{(1-\tau^2)^2}x^2y^2 \leq (1-\tau^2)^2 \Leftrightarrow \\
    &(x^2+y^2)^2 + \frac{16\tau^2}{(1-\tau^2)^2}x^2y^2 - 4\tau(x^2-y^2) \leq (1-\tau^2)^2 .
\end{align*}
\pagebreak

In this more convenient form we introduce complex variables before making the substitution:
\begin{align*}
    &(x^2+y^2)^2 + \frac{16\tau^2}{(1-\tau^2)^2}x^2y^2 - 4\tau (x^2-y^2) \leq (1-\tau^2)^2 \Leftrightarrow \\
    &|z|^4 + \frac{16\tau^2}{(1-\tau^2)^2}x^2y^2 - 4\tau (z^2 - 2ixy) \leq (1-\tau^2)^2 \Leftrightarrow \\
    &|z|^4 - 4\tau z^2 + 8\tau i \frac{z+\overline{z}}{2}\frac{z-\overline{z}}{2i} + \frac{16\tau^2}{(1-\tau^2)^2} \left ( \frac{z+\overline{z}}{2}\frac{z-\overline{z}}{2i} \right )^2 \leq (1-\tau^2)^2 \Leftrightarrow \\
    &|z|^4 -4\tau z^2 + 2\tau (z^2 - \overline{z}^2) - \frac{\tau^2}{(1-\tau^2)^2} (z^2 - \overline{z}^2)^2 \leq (1-\tau^2)^2 .
\end{align*}
This is the right place to make the substitution $z^2 \mapsto z$ and then we will get back to real variables, expand everything and complete the squares. We arrive at the new domain $\widehat{S}_{\tau}$ as follows:
\begin{align*}
    &|z|^2 - 4\tau z + 2\tau (z-\overline{z}) - \frac{\tau^2}{(1-\tau^2)^2} (z - \overline{z})^2 \leq (1-\tau^2)^2 \Leftrightarrow \\
    &x^2+y^2 - 4\tau x + \frac{4\tau^2}{(1-\tau^2)^2} y^2 \leq 1 + \tau^4 - 2\tau^2 \Leftrightarrow \\
    &x^2 - 4\tau x + 2\tau^2 + \frac{(1-\tau^2)^2y^2 + 4\tau^2y^2}{(1-\tau^2)^2} \leq 1 + \tau^4 \Leftrightarrow \\
    &x^2 - 4\tau x + 2\tau^2 + \frac{y^2}{(1-\tau^2)^2}(1+\tau^2)^2 \leq 1 + \tau^4 \Leftrightarrow \\
    &(x- 2\tau)^2 + \frac{y^2}{(1-\tau^2)^2}(1+\tau^2)^2 \leq (1+\tau^2)^2 \Leftrightarrow \\
    &\frac{(x-2\tau)^2}{(1+\tau^2)^2} + \frac{y^2}{(1-\tau^2)^2} \leq 1 .
\end{align*}
That means that
\begin{align} \label{int-eq}
\int_{\mathbb{C}} f_{\widehat{\zeta_0},a}(z^2) d \mu (z) = \int_{{\mathbb{C}}} f_{\widehat{\zeta_0},a}(z) d \widehat{\mu} (z),
\end{align}
as we wanted, where we used the fact that the transformation $z \mapsto z^2$ maps the complex plane two times onto itself and the identity for the Jacobian \eqref{Jacob} to retrieve the measure $\widehat{\mu}$ for the correlated covariance matrix.

Inequality \ref{prob-ineq} together with equality \ref{int-eq} implies theorem \ref{main'}.
\end{proof}
\end{section}
\chapter{Conclusion}

\begin{section}{Overview}
In this thesis we started by giving some preliminary definitions and results about the usual deterministic matrices of linear algebra and functional analysis. We then moved on to define them as random objects in some proper probability spaces in which they are called random matrices.

We then gave definitions regarding the convergence behaviour of a sequence of usual deterministic measures to another measure. We managed to transform the space of all probability measures in a field into a measurable space, which allowed us to study measure-valued random variables, which are called random measures. Just like the convergence behaviour of usual deterministic measures, we discussed about the probabilistic convergence of a sequence of random measures to another deterministic measure. A classical example of a random measure in random matrix theory which encodes a lot of information about the eigenvalues of a random matrix as we saw was the empirical spectral measure, which as a sequence depends on the dimension of the matrix. Its convergence properties as well as its deterministic limiting measure can give valuable insights about the initial random matrix.
\newpage

After that, we defined an appropriate transformation of a measure to a complex analytic function which encodes a lot of the properties of the measure, such as its convergence. This was about the Stieltjes tranformation which gives a 1-1 correspondence between each finite Borel measure and its Stieltjes function. We saw that there were easy formulas to recover the measure by its Stieltjes transformation, while the limiting measure corresponds to the limiting Stieltjes function in terms of uniform convergence of functions. This means that in order to establish the convergence of the measures it is enough to establish the convergence of the Stieltjes functions. What's more, we saw that this technique is specifically useful for establishing local laws for the limiting measure, as the Stieltjes function also encodes the scale of the convergence of the measures.

We used the Stieltjes transform technique to analyse the convergence of the empirical spectral measure of a random sample covariance matrix to the deterministic Marchenko-Pastur measure. Initially, we proved an optimal rate of convergence of the Stieltjes functions of the sequence of the spectral measures to the Stieltjes function of the Marchenko-Pastur distribution. We proved this in the optimal scaling for the Stieltjes functions. Special care was taken for the concentration of the eigenvalues around zero and the corresponding singularity of the limiting measure. Using this functional convergence we could go back to the measure-theoretic convergence of the spectral measures to the Marchenko-Pastur distribution and establish it in an almost optimal rate. We then used this result to prove the rigidity of the eigenvalues depending on their position with special treatment given near zero. Their locations turn out to be really close to the ones predicted by the Marchenko-Pastur measure while the size of the fluctuations proved was almost optimal.

In the second part of this thesis, we started analysing the Dyson equation method. It turns out that the Stieltjes transform method is a special case of this technique, as the random empirical spectral measure gets Stieltjes-transformed to the complex function which matches with the trace of the resolvent of the initial random matrix. The Dyson equation method is more general as it involves the treatment of the resolvent as a whole matrix and not just its trace. Of course, this means that the resolvent matrix includes almost all the information about the empirical spectral measure in which we are interested. The Dyson equation method gives a self-consistent matrix equation for the resolvent matrix which takes into account only the first and second order correlations of the initial random matrix. A stability analysis is always needed then for such a technique so as to prove that the solution matrix of the Dyson equation remains close to the resolvent matrix subject to perturbations. If some assumptions are satisfied for the initial random matrix, then the method works well and we get as a theorem that the solution matrix is indeed close to the resolvent matrix and admits now a Stieltjes transform representation of a matrix whose trace can give us the limit measure of the empirical spectral measure sequence.

We used the Dyson equation technique in a random matrix model whose limitng spectral measure lies on the complex plane and interpolates between the second power of the circular law, the second power of a shifted elliptical law and the Marchenko-Pastur distribution. There were previous results about this limit measure when the entries of the random matrix model were Gaussian random variables. The Dyson equation technique could then be used for proving the same result for a universal case of random variables. We began by hermitizing the random matrix and setting up a Dyson equation for the new hermitized matrix. We solved this equation and proved that the trace of the solution was close to the trace of the resolvent matrix of the initial random matrix. For the solution, we needed to apply a stability analysis. After that, we were able to retrieve the same limit measure from the trace of the solution matrix while also using a least singular value control of the hermitized matrix.
\end{section}
\newpage

\begin{section}{Further directions}
There were some missing targets while producing all the previous results. Specifically, in the local Marchenko-Pastur analysis, we wanted to prove a rate of convergence of the rate of $\mathcal{O}\left (\frac{\sqrt{\log N}}{N} \right )$ and a fluctuation estimate for the bulk eigenvalues of order $\mathcal{O}\left (\frac{\sqrt{\log N}}{N} \right )$ as these are the optimal bounds and they have been proved for the Wishart ensemble. Secondly, in the Correlated Covariance matrices, the initial goal was to prove everything for rectangular random matrices $X_1$ and $X_2$ and not just square ones, as was the case in the previous result for the non-Hermitian Wishart ensemble. What's more, our analysis avoided some local laws around the origin of the spectrum of this ensemble as well as in its edges.

We invite all random matrix theory researchers to deal with these issues by providing even more refined results in this analysis while advancing and pushing forward this rather new and exciting scientific field.
\end{section}
@article{cacciapuoti2013local,
  title={Local {M}archenko-{P}astur law at the hard edge of {S}ample {C}ovariance {M}atrices},
  author={Cacciapuoti, Claudio and Maltsev, Anna and Schlein, Benjamin},
  journal={Journal of Mathematical Physics},
  volume={54},
  number={4},
  pages={043302},
  year={2013},
  publisher={American Institute of Physics}
}

@article{cacciapuoti2015bounds,
  title={Bounds for the {S}tieltjes transform and the density of states of {W}igner matrices},
  author={Cacciapuoti, Claudio and Maltsev, Anna and Schlein, Benjamin},
  journal={Probability Theory and Related Fields},
  volume={163},
  number={1},
  pages={1--59},
  year={2015},
  publisher={Springer}
}

@article{erdos2019matrix,
  title={The matrix {D}yson equation and its
 applications for {R}andom {M}atrices},
  author={L{\'a}szl{\'o} Erdős},
  journal={Random Matrices},
  year={2019},
}

@phdthesis{fleermann2019global,
  title={Global and local semi--circle laws for {R}andom {M}atrices with correlated entries},
  author={Fleermann, Michael},
  year={2019},
  school={FernUniversit{\"a}t in Hagen}
}

@article{erdHos2019random,
  title={Random matrices with slow correlation decay},
  author={Erd{\H{o}}s, L{\'a}szl{\'o} and Kr{\"u}ger, Torben and Schr{\"o}der, Dominik},
  journal={Forum of Mathematics, Sigma},
  volume={7},
  pages={e8},
  year={2019},
  publisher={Cambridge University Press}
}

@article{gotze2014rate,
  title={Rate of Convergence of the Expected Spectral Distribution Function to the {M}archenko--{P}astur Law},
  author={Tikhomirov, AN},
  journal={Siberian Advances in Mathematics},
  year={2009},
  volume = {19},
  pages = {277--286}
}

@article{gotze2016optimal,
  title={Optimal bounds for convergence of expected spectral distributions to the semi-circular law},
  author={G{\"o}tze, Friedrich and Tikhomirov, Alexander},
  journal={Probability Theory and Related Fields},
  volume={165},
  number={1-2},
  pages={163--233},
  year={2016},
  publisher={Springer}
}

@article{marchenko1967distribution,
  title={Distribution of eigenvalues for some sets of random matrices},
  author={Marchenko, Vladimir Alexandrovich and Pastur, Leonid Andreevich},
  journal={Matematicheskii Sbornik},
  volume={114},
  number={4},
  pages={507--536},
  year={1967},
  publisher={Russian Academy of Sciences, Steklov Mathematical Institute of Russian~…}
}

@inproceedings{gustavsson2005gaussian,
  title={Gaussian fluctuations of eigenvalues in the {GUE}},
  author={Gustavsson, Jonas},
  booktitle={Annales de l'IHP Probabilit{\'e}s et statistiques},
  volume={41},
  pages={151--178},
  year={2005}
}

@article{su2006gaussian,
  title={Gaussian fluctuations in complex sample covariance matrices},
  author={Su, Zhonggen},
  journal={Electronic Journal of Probability},
  volume={11},
  pages={1284--1320},
  year={2006},
  publisher={The Institute of Mathematical Statistics and the Bernoulli Society}
}

@article{pillai2014universality,
  title={Universality of covariance matrices},
  author={Pillai, Natesh S and Yin, Jun},
  journal={The Annals of Applied Probability},
  volume={24},
  number={3},
  pages={935--1001},
  year={2014},
  publisher={Institute of Mathematical Statistics}
}

@article{pleijel1963theorem,
  title={On a theorem by {P. M}alliavin},
  author={Pleijel, {\AA}ke},
  journal={Israel Journal of Mathematics},
  volume={1},
  number={3},
  pages={166--168},
  year={1963},
  publisher={Springer}
}

@article{erdHos2016fluctuations,
  title={Fluctuations of functions of {W}igner matrices},
  author={Erd{\H{o}}s, L{\'a}szl{\'o} and Schr{\"o}der, Dominik},
  pages = {1--15},
  journal={Electronic Communications in Probability},
  volume={21},
  year={2016},
  publisher={The Institute of Mathematical Statistics and the Bernoulli Society}
}

@article{osekowski2012note,
  title={A note on {B}urkholder-{R}osenthal inequality},
  author={Osekowski, Adam},
  journal={Bull. Pol. Acad. Sci. Math},
  volume={60},
  number={2},
  pages={177--185},
  year={2012}
}

@article{erdos2012local,
  title={The local relaxation flow approach to universality of the local statistics for {R}andom {M}atrices},
  author={Erd{\H{o}}s, L{\'a}szl{\'o} and Schlein, Benjamin and Yau, Horng-Tzer and Yin, Jun},
  journal = {Annales de l'Institut Henri Poincaré, Probabilités et Statistiques},
  number = {1},
  publisher = {Institut Henri Poincaré},
  pages = {1 -- 46},
  year = {2012},
  volume = {48}
}

@article{o2011products,
  title={Products of independent non-{H}ermitian random matrices},
  author={O'Rourke, Sean and Soshnikov, Alexander},
  year={2011},
  volume = {16},
  journal = {Electronic Journal of Probability}
}

@article{akemann2021non,
  title={A non-{H}ermitian generalisation of the {M}archenko--{P}astur distribution: from the circular law to multi-criticality},
  author={Akemann, Gernot and Byun, Sung-Soo and Kang, Nam-Gyu},
  journal={Annales Henri Poincar{\'e}},
  volume={22},
  pages={1035--1068},
  year={2021},
  publisher={Springer}
}

@article{alt2022local,
  title={Local elliptic law},
  author={Alt, Johannes and Kr{\"u}ger, Torben},
  journal={Bernoulli},
  volume={28},
  number={2},
  pages={886--909},
  year={2022},
  publisher={Bernoulli Society for Mathematical Statistics and Probability}
}

@article{o2015products,
  title={Products of independent elliptic random matrices},
  author={O’Rourke, Sean and Renfrew, David and Soshnikov, Alexander and Vu, Van},
  journal={Journal of Statistical Physics},
  volume={160},
  number={1},
  pages={89--119},
  year={2015},
  publisher={Springer}
}

@article{ajanki2017singularities,
  title={Singularities of Solutions to Quadratic Vector Equations on the Complex Upper Half-Plane},
  author={Ajanki, Oskari and Kr{\"u}ger, Torben and Erd{\H{o}}s, L{\'a}szl{\'o}},
  journal={Communications on Pure and Applied Mathematics},
  volume={70},
  number={9},
  pages={1672--1705},
  year={2017},
  publisher={Wiley Online Library}
}

@article{nguyen2015elliptic,
  title={The elliptic law},
  author={Nguyen, Hoi H and O’Rourke, Sean},
  journal={International Mathematics Research Notices},
  volume={2015},
  number={17},
  pages={7620--7689},
  year={2015},
  publisher={Oxford University Press}
}

@article{gotze2015minimal,
  title={On minimal singular values of random matrices with correlated entries},
  author={G{\"o}tze, Friedrich and Naumov, Alexey and Tikhomirov, Alexander},
  journal={Random Matrices: Theory and Applications},
  volume={4},
  number={02},
  pages={1550006},
  year={2015},
  publisher={World Scientific}
}

@article{alt2021inhomogeneous,
  title={Inhomogeneous circular law for correlated matrices},
  author={Alt, Johannes and Kr{\"u}ger, Torben},
  journal={Journal of Functional Analysis},
  volume={281},
  number={7},
  pages={109120},
  year={2021},
  publisher={Elsevier}
}

@article{helton2007operator,
  title={Operator-valued semicircular elements: solving a quadratic matrix equation with positivity constraints},
  author={Helton, J William and Far, Reza Rashidi and Speicher, Roland},
  journal={International Mathematics Research Notices},
  volume={2007},
  number={9},
  year={2007}
}

@ARTICLE{Fyodorov:2011,
AUTHOR = {Fyodorov, Y. },
TITLE   = {{R}andom {M}atrix theory},
YEAR    = {2011},
JOURNAL = {Scholarpedia},
VOLUME  = {6},
NUMBER  = {3},
PAGES   = {9886},
}

@article{sommers1988spectrum,
  title={Spectrum of large random asymmetric matrices},
  author={Sommers, Hans Juergen and Crisanti, Andrea and Sompolinsky, Haim and Stein, Yaakov},
  journal={Physical review letters},
  volume={60},
  number={19},
  pages={1895},
  year={1988},
  publisher={APS}
}

@book{feier2012methods,
  title={Methods of proof in {R}andom {M}atrix {T}heory},
  author={Feier, Adina Roxana},
  year={2012},
  publisher={Harvard University}
}

@article{bryson2021marchenko,
  title={Marchenko--{P}astur law with relaxed independence conditions},
  author={Bryson, Jennifer and Vershynin, Roman and Zhao, Hongkai},
  journal={Random Matrices: Theory and Applications},
  volume={10},
  number={04},
  pages={2150040},
  year={2021},
  publisher={World Scientific}
}

@article{bothner2022complex,
  title={The complex elliptic Ginibre ensemble at weak non-Hermiticity: edge spacing distributions},
  author={Bothner, Thomas and Little, Alex},
  journal={arXiv:2208.04684},
  year={2022}
}

@article{ginibre1965statistical,
  title={Statistical ensembles of complex, quaternion, and real matrices},
  author={Ginibre, Jean},
  journal={Journal of Mathematical Physics},
  volume={6},
  number={3},
  pages={440--449},
  year={1965},
  publisher={American Institute of Physics}
}

@article{girko1985circular,
  title={Circular law},
  author={Girko, Vyacheslav L},
  journal={Theory of Probability \& Its Applications},
  volume={29},
  number={4},
  pages={694--706},
  year={1985},
  publisher={SIAM}
}

@article{naumov2012elliptic,
  title={Elliptic law for real random matrices},
  author={Naumov, Alexey},
  year = {2013},
  pages = {31-38},
  journal = {Moscow University Computational Mathematics and Cybernetics}
}

@article{benaych2016lectures,
  title={Lectures on the local semi--circle law for Wigner matrices},
  author={Florent Benaych-Georges and Antti Knowles},
  journal={arXiv: Probability},
  year={2016},
}

@book{kallenberg1997foundations,
  title={Foundations of modern {P}robability},
  author={Kallenberg, Olav},
  year={1997},
  publisher={Springer}
}

@article{smith1995recursive,
  title={A recursive formulation of the old problem of obtaining moments from cumulants and vice versa},
  author={Smith, Peter J},
  journal={The American Statistician},
  volume={49},
  number={2},
  pages={217--218},
  year={1995},
  publisher={Taylor \& Francis}
}

@book{anderson2010introduction,
  title={An Introduction to Random Matrices},
  author={Anderson, Greg W. and Guionnet, Alice and Zeitouni, Ofer},
  year={2009},
  publisher={Cambridge University Press}
}

@article{parzen1962estimation,
  title={On estimation of a probability density function and mode},
  author={Parzen, Emanuel},
  journal={The annals of mathematical statistics},
  volume={33},
  number={3},
  pages={1065--1076},
  year={1962},
  publisher={JSTOR}
}

@book{lebl2019tasty,
  title={Tasty bits of several complex variables},
  author={Lebl, Jiri},
  year={2019},
  publisher={Lebl, Jiri}
}

@book{martin1966complex,
  title={Complex {A}nalysis},
  author={Martin, D and Ahlfors, LV},
  year={1966},
  publisher={McGraw-Hill, New York}
}

@book{gong2007concise,
  title={Concise complex analysis (revised edition)},
  author={Gong, Sheng and Gong, Youhong},
  year={2007},
  publisher={World Scientific Publishing Company}
}

@book{akemann2011oxford,
  title={The Oxford Handbook of Random Matrix Theory},
  author={Akemann, Gernot and Baik, Jinho and Di Francesco, Philippe},
  year={2015},
  publisher={Oxford University Press}
}

@book{mehta2004random,
  title={Random {M}atrices},
  author={Mehta, Madan Lal},
  year={2004},
  publisher={Elsevier}
}

@article{pastur1973spectra,
  title={Spectra of random self--adjoint operators},
  author={Pastur, Leonid A},
  journal={Russian mathematical surveys},
  volume={28},
  year={1973},
  publisher={IOP Publishing}
}

@article{wigner1958distribution,
  title={On the distribution of the roots of certain symmetric matrices},
  author={Wigner, Eugene P},
  journal={Annals of Mathematics},
  volume={67},
  number={2},
  pages={325--327},
  year={1958},
  publisher={JSTOR}
}

@book{Artin:1998,
  title={Algebra},
  author={Artin, M.},
  year={2011},
  publisher={Pearson Education}
}

@book{dudley2002real,
  title={Real Analysis and Probability},
  author={Dudley, R.M.},
  year={2002},
  publisher={Cambridge University Press}
}

@book{tao2023topics,
  title={Topics in random matrix theory},
  author={Tao, Terence},
  volume={132},
  year={2023},
  publisher={American Mathematical Society}
}

@book{durrett2019probability,
  title={Probability: Theory and Examples},
  author={Durrett, Rick},
  year={2019},
  publisher={Cambridge university press}
}

@article{kargin2013lecture,
  title={Lecture Notes on {R}andom {M}atrix {T}heory},
  author={Kargin, Vladislav and Yudovina, Elena},
  journal={arXiv:1305.2153},
  year={2013}
}

@book{chung2001course,
  title={A course in Probability Theory},
  author={Chung, Kai Lai},
  year={2001},
  publisher={Academic press}
}

@book{kallenberg2017random,
  title={Random measures, theory and applications},
  author={Kallenberg, Olav},
  year={2017},
  publisher={Springer}
}

@book{rudin86real,
author = {Rudin, Walter},
title = {Real and Complex Analysis, 3rd Ed.},
year = {1987},
isbn = {0070542341},
publisher = {McGraw-Hill, Inc.},
address = {USA}
}

@book{voisin_2002, 
place={Cambridge}, 
series={Cambridge Studies in Advanced Mathematics}, 
title={Hodge Theory and Complex Algebraic Geometry},
publisher={Cambridge University Press}, 
author={Voisin, Claire}, 
year={2002}
}

@book{higham2002accuracy,
  title={Accuracy and stability of numerical algorithms},
  author={Higham, Nicholas J},
  year={2002},
  publisher={SIAM}
}

@article{erdos2010universality,
  title={Universality for generalized {W}igner matrices with {B}ernoulli distribution},
  author={Erdos, L{\'a}szl{\'o} and Yau, Horng-Tzer and Yin, Jun},
  journal={arXiv:1003.3813},
  year={2010}
}

@article{erdHos2012rigidity,
  title={Rigidity of eigenvalues of generalized {W}igner matrices},
  author={Erd{\H{o}}s, L{\'a}szl{\'o} and Yau, Horng-Tzer and Yin, Jun},
  journal={Advances in Mathematics},
  volume={229},
  number={3},
  pages={1435--1515},
  year={2012},
  publisher={Elsevier}
}

@article{erdHos2013averaging,
  title={Averaging fluctuations in resolvents of random band matrices},
  author={Erd{\H{o}}s, L{\'a}szl{\'o} and Knowles, Antti and Yau, Horng-Tzer},
  journal={Annales Henri Poincar{\'e}},
  volume={14},
  number={8},
  pages={1837--1926},
  year={2013},
  publisher={Springer}
}

@article{erdHos2013local,
  title={The local semi--circle law for a general class of {R}andom {M}atrices},
  author={Erd{\H{o}}s, L{\'a}szl{\'o} and Knowles, Antti and Yau, Horng-Tzer and Yin, Jun},
  journal={Electronic Journal of Probability},
  year={2012},
  volume={18},
  pages={1-58},
}

@article{rudelson2008littlewood,
  title={The {L}ittlewood--{O}fford problem and invertibility of {R}andom {M}atrices},
  author={Rudelson, Mark and Vershynin, Roman},
  journal={Advances in Mathematics},
  volume={218},
  number={2},
  pages={600--633},
  year={2008},
  publisher={Elsevier}
}

@article{bordenave2012around,
author = {Charles Bordenave and Djalil Chafa{\"i}},
title = {{Around the circular law}},
volume = {9},
journal = {Probability Surveys},
publisher = {Institute of Mathematical Statistics and Bernoulli Society},
pages = { 1 -- 89},
year = {2012},
}

@article{ajanki2019stability,
  title={Stability of the matrix {D}yson equation and random matrices with correlations},
  author={Ajanki, Oskari H and Erd{\H{o}}s, L{\'a}szl{\'o} and Kr{\"u}ger, Torben},
  journal={Probability Theory and Related Fields},
  volume={173},
  pages={293--373},
  year={2019},
  publisher={Springer}
}

@article{ajanki2019quadratic,
  title={Quadratic vector equations on complex upper half-plane},
  author={Ajanki, Oskari and Erd{\H{o}}s, L{\'a}szl{\'o} and Kr{\"u}ger, Torben},
  journal={Memoirs of the American Mathematical Society},
  volume={261},
  number={1261},
  year={2019},
  publisher={American Mathematical Society}
}

@article{khorunzhy1996asymptotic,
  title={Asymptotic properties of large random matrices with independent entries},
  author={Khorunzhy, Alexei M and Khoruzhenko, Boris A and Pastur, Leonid A},
  journal={Journal of Mathematical Physics},
  volume={37},
  number={10},
  pages={5033--5060},
  year={1996},
  publisher={American Institute of Physics}
}

@article{gesztesy2000matrix,
  title={On matrix--valued {H}erglotz functions},
  author={Gesztesy, Fritz and Tsekanovskii, Eduard},
  journal={Mathematische Nachrichten},
  volume={218},
  number={1},
  pages={61--138},
  year={2000},
  publisher={Wiley Online Library}
}

\bibliographystyle{plain}
\end{document}